\documentclass[times,doublespace]{oupau}
\usepackage{amsmath,amssymb,amsthm,enumitem}

\usepackage{xypic} 
\usepackage{mathrsfs} 
\usepackage{xspace}
\usepackage[pagebackref=true, colorlinks]{hyperref}

\newtheorem{thm}{Theorem}[subsection]

\newtheorem{prop}[thm]{Proposition}
\newtheorem{lem}[thm]{Lemma}
\newtheorem{cor}[thm]{Corollary}

\theoremstyle{definition}
\newtheorem{defn}[thm]{Definition}
\newtheorem{qu}[thm]{Question}
\newtheorem{rmk}[thm]{Remark}

\newcommand{\Zb}{\mathbf{Z}}
\newcommand{\Nb}{\mathbf{N}}
\newcommand{\Rb}{\mathbf{R}}

\newcommand{\Nbb}{\mathbb{N}}

\newcommand{\mc}[1]{\mathcal{#1}}

\newcommand{\mf}[1]{\mathfrak{#1}}
\newcommand{\ms}[1]{\mathscr{#1}} 

\newcommand{\ol}[1]{\overline{#1}}

\newcommand{\acts}{\curvearrowright}
\newcommand{\sleq}{\leqslant}

\newcommand{\injects}{\hookrightarrow}

\newcommand{\triv}{\{1\}}
\newcommand{\inv}{^{-1}}
\newcommand{\tdlc}{t.d.l.c.\@\xspace}
\newcommand{\lcsc}{l.c.s.c.\@\xspace}
\newcommand{\tdlcsc}{t.d.l.c.s.c.\@\xspace}
\newcommand{\Sym}{\mathrm{Sym}}
\newcommand{\Alt}{\mathrm{Alt}}
\newcommand{\Soc}{\mathrm{Soc}}
\newcommand{\Ad}{\mathrm{Ad}}
\newcommand{\ad}{\mathrm{ad}}
\newcommand{\Aut}{\mathrm{Aut}}

\newcommand{\rist}{\mathrm{rist}}

\newcommand{\normal}{\trianglelefteq}

\DeclareMathOperator{\Comm}{Comm}

\DeclareMathOperator{\CC}{C}
\DeclareMathOperator{\N}{N}
\DeclareMathOperator{\Z}{Z}
\DeclareMathOperator{\QC}{QC}
\DeclareMathOperator{\QZ}{QZ}

\newcommand{\Res}{\mathrm{Res}}
\newcommand{\Mon}{\mathrm{Mon}}

\newcommand{\la}{\langle}
\newcommand{\ra}{\rangle}
\newcommand{\lla}{\langle \langle}
\newcommand{\rra}{\rangle \rangle}
\newcommand{\Es}{\ms{E}}
\newcommand{\Ss}{\ms{S}}
\newcommand{\Rs}{\ms{R}}
\newcommand{\LN}{\mc{LN}}
\newcommand{\LC}{\mc{LC}}
\newcommand{\con}{\mathrm{con}}
\newcommand{\rest}{\upharpoonright}

\newcommand{\grp}[1]{\langle #1 \rangle}
\newcommand{\cgrp}[1]{\overline{\langle #1 \rangle}}
\newcommand{\ngrp}[1]{\overline{\langle\langle #1 \rangle\rangle}}

\newcommand{\wh}[1]{\widehat{#1}}

\usepackage[usenames, dvipsnames]{color}

\begin{document}

\title{Approximating  simple locally compact groups by their dense locally compact subgroups}
\shorttitle{Dense l.c. subgroups of l.c. groups}

\volumeyear{2017}
\paperID{?}

\author{Pierre-Emmanuel Caprace\affil{1}, Colin Reid\affil{2}, and Phillip Wesolek\affil{3}}
\abbrevauthor{P.-E. Caprace, C. D. Reid, and P. Wesolek}
\headabbrevauthor{Caprace, P.-E., Reid, C. D., and Wesolek, P.}

\address{%
\affilnum{1}Universit\'e catholique de Louvain, IRMP, Chemin du Cyclotron 2, bte L7.01.02, 1348 Louvain-la-Neuve, Belgique; \affilnum{2} University of Newcastle, School of Mathematical and Physical Sciences, Callaghan, NSW 2308, Australia;
and
\affilnum{3} Binghamton University, Department of Mathematical Sciences, PO Box 6000, 	Binghamton, New York 13902-6000, USA}

\correspdetails{pwesolek@binghamton.edu}

\received{}
\revised{}
\accepted{}


\begin{abstract}
	The class $\Ss$ of totally disconnected locally compact (\tdlc) groups that are non-discrete, compactly generated, and topologically simple contains many compelling examples. In recent years, a general theory for these groups, which studies the interaction between the compact open subgroups and the global structure, has emerged. In this article, we study the non-discrete \tdlc groups $H$ that admit a continuous embedding with dense image into some $G\in \Ss$; that is, we consider the dense locally compact subgroups of groups $G\in \Ss$. We identify a class $\Rs$ of almost simple groups which properly contains $\Ss$ and is moreover stable under passing to a non-discrete dense locally compact subgroup. We show that $\Rs$ enjoys many of the same properties previously obtained for $\Ss$ and establish various original results for $\Rs$ that are also new for the subclass $\Ss$, notably concerning the structure of the local Sylow subgroups and the full automorphism group. 
\end{abstract}

\maketitle

\tableofcontents
\addtocontents{toc}{\protect\setcounter{tocdepth}{2}}

\section{Introduction}
%


%
\begin{flushright}
	\begin{minipage}[t]{0.55\linewidth}\itshape\small
		\dots the early hot dense phase was unavoidable\dots
		
		\hfill\upshape (Stephen Hawking, \emph{A Brief History of Time}, 1998)
	\end{minipage}
\end{flushright}

\medskip

In exploring the structure of locally compact groups, the occurrence of dense locally compact subgroups is unavoidable. Indeed, given a locally compact group $G$ and closed subgroups $A, N \leq G$ such that $N$ is normal,   the abstract isomorphism 
$$A/A \cap N \to AN/N$$
can notoriously fail to be  a homeomorphism (see \cite[(5.39.d)]{HR}). Nevertheless,  it is always a continuous, injective homomorphism of the locally compact group $A/A \cap N$ onto a dense subgroup of the locally compact group $\overline{AN}/N$. This  naturally leads us to  define   a \textit{dense embedding} to be a continuous, injective homomorphism $\psi \colon H \to G$ of a  locally compact group $H$ to a locally compact group $G$ such that $\psi(H)$ is dense in $G$. A \textit{dense locally compact subgroup} of a locally compact group $G$ is defined as  a locally compact group $H$ that admits a dense embedding into $G$. 

The point of view adopted in this paper is to view a dense embedding $\psi \colon H \to G$ as an algebraic approximation of  $G$ by $H$. We are thus led to the following question: \textit{To what extent is the structure of $H$ governed by that of $G$?} We are especially interested in the case where $G$ is topologically simple.

An obstruction to such a transfer of structure from $G$ to $H$ is provided by the inescapable situation where $H$ is discrete. Every locally compact group admits at least one, and usually many, dense embeddings of discrete groups. A topologically simple group often contains finitely generated dense free subgroups; see \cite{BreuGela} for the emblematic case of simple Lie groups. 

It can also happen that every proper dense locally compact subgroup is discrete. This occurs in a simple Lie group or in a simple algebraic group over a local field (see \S\ref{sec:denseLCLie}). This rigidity phenomenon, however, happens to be rather atypical. By work of the second named author, every totally disconnected locally compact (\tdlc) group containing two infinite subgroups that are respectively pro-$p$ and pro-$q$ for two distinct primes $p$ and $q$ admits natural \textit{non-discrete} dense locally compact subgroups that are often proper (see  \cite{Reid} and \S\ref{sec:Localization} below).  In particular, groups belonging  to the class $\Ss$ consisting of the  non-discrete, compactly generated, topologically simple \tdlc groups often have non-discrete proper dense locally compact subgroups.\footnote{The class $\Ss$ contains many compelling examples of locally compact groups, including simple algebraic groups over non-archimedean local fields, many groups acting on trees, groups acting on $\mathrm{CAT}(0)$ cube complexes, and groups almost acting on trees. We refer to \cite[Appendix~A]{CRW2} and references therein. }

Non-discrete dense locally compact subgroups of groups in $\Ss$ can nevertheless fail to be simple (see \cite[Corollary 4.20]{LeBou16}). They may be topologically simple but not   $\sigma$-compact (see Proposition~\ref{prop:ex-non-sigma-cpt}). Moreover, a non-discrete dense locally compact subgroup $H$ of a group $G$  in $\Ss$ can fail to have any compactly generated subgroup that is dense in $G$ (see Example~\ref{ex:Radu}). The starting point of this work is the observation that for groups in $\Ss$, there is a form of structural transfer from the ambient group to a \textit{non-discrete} dense locally compact subgroup:  whenever $H$ is non-discrete, the simplicity of $G$ is strongly reflected in the normal subgroup structure of $H$. In establishing a precise description of that phenomenon, we are led to consider a class of almost simple groups, denoted by $\Rs$, that is strictly larger than $\Ss$ and is closed under taking non-discrete dense locally compact subgroups.

\begin{rmk}\label{rmk:regional}
In topological group theory, it is customary to use the term ``local" to qualify a property that is satisfied by all members of some basis of identity neighborhoods (in the case of \tdlc groups, the basis of identity neighborhoods is understood to consist of compact open subgroups), while the term ``global" is used in reference to the properties of the whole ambient group.  We introduce the term ``regional" to qualify a  property satisfied at an intermediate scale: a property is said to hold \emph{regionally} if it is satisfied by all members of an exhaustion of the ambient group by a directed system of compactly generated open subgroups of $G$.

A particular example is the following: in the literature, the term ``locally elliptic'' has been used for the property that every compactly generated closed subgroup is compact.  This use of ``locally'' is at odds with our convention of using ``locally'' only to refer to properties exhibited by arbitrarily small neighborhoods of the identity.  We will therefore use the alternative term ``regionally compact'', which has no such ambiguity and emphasizes the correct contrast with the weaker property ``locally compact''.
\end{rmk}

\subsection{Robustly monolithic groups}

The \textit{normal core} of a subgroup $L$ in $K$ is $\bigcap_{k\in K}kLk^{-1}$.
A \tdlc group $G$ is called \textit{regionally expansive} if some compactly generated open subgroup $O$ has a compact open subgroup $U$ such that the normal core of $U$ in $O$ is trivial. A locally compact group $G$ is \textit{monolithic} if the intersection of all non-trivial closed normal subgroups is non-trivial, and that intersection is then called the \textit{monolith}, and denoted by $\Mon(G)$. In particular $G$ is topologically simple if and only if $G$ is monolithic and $G = \Mon(G)$. More generally, a locally compact group $G$ that is monolithic with a topologically simple monolith is called \textit{topologically almost simple}. Notice that a finite group (with the discrete topology) is topologically almost simple if and only if it is almost simple in the  standard terminology. 

\begin{defn}
A \tdlc group $G$ is called \emph{robustly monolithic} if it is monolithic, and the monolith is non-discrete, regionally expansive, and topologically simple. In particular, every member of $\Rs$ is topologically almost simple. We denote by $\Rs$ the class of robustly monolithic groups. 
\end{defn}

Our first main result ensures that the class $\Rs$ is stable under taking non-discrete dense locally compact subgroups.

\begin{thm}[See Theorem~\ref{thm:R_dense_embedding}]\label{thmintro:R_dense_embedding}
Every non-discrete dense locally compact subgroup of any $G\in \Rs$ also belongs to $\Rs$. 
\end{thm}

Being robustly monolithic is in fact a regional phenomenon, although the definition has obvious global conditions. This is one of the key features of the class $\Rs$.

\begin{thm}[Theorem~\ref{thm:approx_R}]
Let $G$ be a \tdlc group and let  $\{O_i\}_{i\in I}$ be a directed system of compactly generated open subgroups of $G$ with $\bigcup_{i\in I}O_i=G$. The following assertions are equivalent. 
\begin{enumerate}[label=(\roman*)]
\item $G \in \Rs$. 	

\item There is $i\in I$ such that $O_j\in \Rs$ for all $j\geq i$. 	\qedhere
\end{enumerate}
\end{thm}	

A first motivation for considering $\Rs$ is a natural class of dense locally compact subgroups arising from a result of the second named author, described in \cite{Reid} and recalled in Theorem~\ref{thm:Reid_localizations} below. Given a \tdlc group $G$ that is not locally pro-$p$, but that contains an infinite pro-$p$ subgroup, then $G$ admits a non-discrete dense locally compact subgroup $G_{(p)}$ that is locally pro-$p$.  Combining that result with the theorem above, every group in $\Rs$ can be ``approximated" by locally pro-$p$ groups in $\Rs$.

\begin{cor}[Proposition~\ref{prop:RF_pro-pi}]
Let $G\in \Rs$. Then there is a unique finite set of primes $\pi$  such that $G$ is locally pro-$\pi$ and  for each $p \in \pi$, there is a non-discrete dense locally compact subgroup $H$ of $G$, necessarily in $\Rs$, that is locally pro-$p$. 

In particular the monolith $\Mon(H)$ is a non-discrete topologically simple dense locally pro-$p$ subgroup of $\Mon(G)$, so that every topologically simple group $G$ in $\Rs$ admits a  topologically simple dense locally compact subgroup that is locally pro-$p$ for each   $p \in \pi$. 
\end{cor}

Another motivation to consider the class $\Rs$, rather than just $\Ss$, comes from the question of how much the structure of a simple \tdlc group is determined \emph{locally} - i.e.\ by an arbitrarily small open neighborhood of the identity.  In contrast to the case of simple Lie groups, there are a number of examples of groups in $\Ss$ that are locally but not globally isomorphic.  For example, S.M.~Smith (\cite{Smith}) constructs a family of $2^{\aleph_0}$ pairwise non-isomorphic groups that are all locally isomorphic.

The control that the local structure exerts on the global structure can be made precise, using a special case of a construction \cite{BEW} of Barnea--Ershov--Weigel (see also \cite{CapDeM}).

\begin{thm}[{See \S\ref{sec:germs}}]
Let $G$ be a topologically simple group in $\Rs$. Then there is a \tdlc group $\ms{L}(G)$, unique up to isomorphism,  with an open monolith $\widetilde{G}$ such that the following hold.
\begin{enumerate}[label=(\roman*)]
\item $G$ embeds as an open subgroup of $\widetilde{G}$.
\item For any \tdlc group $H$ locally isomorphic to $G$, there is an open, continuous homomorphism $\theta: H \rightarrow \ms{L}(G)$, and $\theta(H) \le \widetilde{G}$ if $H$ is topologically simple.
\item Both $\widetilde{G}$ and $ \ms{L}(G)$ are members of $\Rs$. In particular $\widetilde G$ is topologically simple. \qedhere
\end{enumerate}
\end{thm}

The group $\ms{L}(G)$ encapsulates all possible global structure admissible in a group locally isomorphic to $G$, including all automorphisms of such a group, and the group $\widetilde{G}$ represents the largest possible \emph{topologically simple} group that is locally isomorphic to $G$. Furthermore, both of these groups lie in $\Rs$. We remark that even if $G \in \Ss$, the group $\widetilde{G}$ may not be in $\Ss$. For instance, if $G$ is one of the aforementioned examples of Smith, $\widetilde{G}$ is not even $\sigma$-compact.  Nevertheless, results for $\Rs$ still apply to both $\widetilde{G}$ and $\ms{L}(G)$.

\subsection{Properties of robustly monolithic groups}

The second main thrust of the article is to explore the properties of the groups in $\Rs$. As will become clear, the flexibility of the class $\Rs$ allows us not only to show that $\Rs$ enjoys many of the same properties as $\Ss$ but also to derive properties novel even for $\Ss$.

\begin{thm}[See Theorem~\ref{thm:prime_content_RM}]
	Let $G \in \Rs$. There is a finite set of primes $\pi$ such that any locally compact group $H$ that acts   continuously and  faithfully by topological group automorphisms on $G$  is  a locally pro-$\pi$ \tdlc group.
\end{thm}

\begin{thm}[See Corollary~\ref{cor:ab_sylow_R}]
Let $G\in \Rs$, $U$ be a compact open subgroup of $G$, and $S$ be a pro-$p$-Sylow subgroup of $U$. If $S$ is infinite, then the only virtually solvable normal subgroup of $S$ is the trivial group. In particular, $S$ is not solvable. 
\end{thm}

The class of \textit{regionally elementary groups} is the smallest class of \tdlc groups containing the second countable profinite groups and discrete groups and closed under taking closed subgroups, Hausdorff quotients, group extensions, and directed unions of open subgroups. The class $\Ss$ contains no regionally elementary groups by the results of \cite{W_E_14}. 

\begin{thm}[See Corollary~\ref{cor:R_no_elementary}]
	The class $\Rs$ contains no regionally elementary groups.
\end{thm}

\begin{cor} 
Let $G\in \Rs$. Every regionally elementary dense locally compact subgroup of $G$  is discrete.
\end{cor}

In \cite{CRW1}, a structure theory of \tdlc groups via locally normal subgroups, i.e.\ subgroups with open normalizer, is developed. This theory is applied to groups in $\Ss$ in \cite{CRW2}. It turns out that the results for the class $\Ss$ generalize to the class $\Rs$.

A \tdlc group $G$ is \textit{[A]-semisimple} if the only element with open centralizer is the identity and the only abelian subgroup with open normalizer is the trivial group. The presence of this property allows for the application of the more powerful tools developed in \cite{CRW1}.

\begin{thm}[{Proposition~\ref{prop:Rs_is_RF_[a]}}]
Every element of $\Rs$ is [A]-semisimple.
\end{thm}

Any $[A]$-semisimple \tdlc group $G$ has an associated Boolean algebra on which it acts, denoted by $\LC(G)$ and called the \textit{centralizer lattice}. Recall a Boolean algebra $\mathcal A$ has a canonically associated compact space  $\mf{S}(\mathcal A)$ called the \textit{Stone space}.

\begin{thm}[{Theorem~\ref{thm:StronglyProx}}]
For $G \in \Rs$, the $G$-action on the Stone space $\mf{S}(\LC(G))$ is minimal, strongly proximal, and has a compressible open set. In particular, if $G$ is amenable, then every non-trivial locally normal subgroup of $G$ has trivial centralizer. 
\end{thm}

\begin{cor}[{See Corollary~\ref{cor:AbstractSimplicity}}]
Let $G$ be a topologically simple group in $\Rs$.  If $G$ has an open subgroup of the form $K \times L$ such that $K$ and $L$ are non-trivial closed subgroups, then $G$ is non-amenable and abstractly simple.
\end{cor}

Our investigations conclude with several examples of groups in the class $\Rs$. We exhibit an example of a non-$\sigma$-compact group in $\Rs$ that appears as a dense locally compact subgroup of a group in $\Ss$ (see Subsection~\ref{sec:ConcreteEx}). We additionally characterize the groups $G(F,F')$, defined by A.~Le~Boudec in \cite{LeBou16}, which are members of $\Rs$ (see Theorem~\ref{thm:char_G(F,F')_Rs}). The following independently interesting fact about the groups $G(F,F')$ is discovered along the way.

\begin{thm}[{See Proposition~\ref{prop:virt_simple_G(F,F')}}]
Take $d>2$, let $c$ be a legal coloring of the $d$-regular tree, and $F'\leq \Sym(d)$ be such that the action of $F'$ is not free. If $F\leq F'\leq \wh{F}$, then $G_c(F,F')$ is virtually simple if and only if $F'$ is transitive and generated by its point stabilizers.
\end{thm}

\subsection{Structure of article}
Section~2 covers preliminaries on \tdlc groups; the results of this section are used throughout the work. Section~3 is independent of all sections besides Section~2. It serves primarily as motivation and to provide a setting for our results, and it can be safely skipped. Sections~4 and~5 contain most of the main results of the article. Section~6 is an immediate application of the results of Sections~4 and~5. The results of Section~6 are not used in later sections, so it can be skipped. Sections~7 and~8 contain the remainder of our main results and should be read together. Section~9 presents several examples and is self-contained, except for the use of a theorem from Section~5. 
\section{Preliminaries}

A   family $(N_i)_{i \in I}$ of subsets of a set $X$ indexed by a directed set $I$ is called \textit{filtering} if for all $i, j \in I$ there exists $k \geq i, j$ such that $N_k \subset N_i \cap N_j$. For a group $G$ acting on a set $X$, the pointwise stabilizer of $Z\subseteq X$ in $G$ is denoted by $G_{(Z)}$. For a group $G$, the commutator of $g,h\in G$ is $[g,h]:=ghg^{-1}h^{-1}$. A topological group is always assumed to be Hausdorff.

\subsection{Compactly generated \tdlc groups}

Our work requires a number of results on compactly generated \tdlc groups.

\begin{prop}[{\cite[Proposition 2.5]{CM11}}]\label{prop:filtering}
Let $G$ be a compactly generated \tdlc group, $\mc{F}$ be a filtering family of  non-trivial closed normal subgroups of $G$, and $V$ a compact open subgroup of $G$. If $\mc{F}$ has trivial intersection, then there is $N\in \mc{F}$ and a closed $K\normal G$ with $K\leq V\cap N$ such that $N/K$ is discrete. In particular, if $G$ is without a non-trivial compact or discrete normal subgroup, then every filtering family of non-trivial closed normal subgroups has a non-trivial intersection. 
\end{prop}

A version of the following is given by \cite[Proposition~2.6 (corrected)]{CM_corr}.  For the reader's convenience, we give a simplified statement and proof here.

\begin{thm}\label{thm:minimal_normal} 
Suppose that $G$ is a compactly generated \tdlc group such that $G$ has no infinite discrete normal subgroups and there is $V\leq G$ compact and open such that the normal core of $V$ in $G$ is trivial. Then every non-trivial closed normal subgroup of $G$ contains a minimal non-trivial closed normal subgroup.
\end{thm}

\begin{proof}
Let $N$ be a non-trivial normal subgroup of $G$ and let $\mc{M}$ be the set of closed normal subgroups $M$ of $G$ such that $\triv \neq M \le N$; it suffices to show that $\mc{M}$ has a minimal element.  Let $\mc{F}$ be a chain in $\mc{M}$, let $L = \bigcap_{M \in \mc{F}}M$ and suppose that $L = \triv$.  Then by Proposition~\ref{prop:filtering}, there is $M \in \mc{F}$ and a closed $K \normal G$ with $K \le V \cap M$ such that $M/K$ is discrete.  Since the normal core of $V$ in $G$ is trivial, we must have $K = \triv$, so in fact $M$ is discrete.  By hypothesis, it follows that $M$ is finite.  Since $\mc{F}$ is a chain, we realize $L$ as the intersection of elements of $\mc{F}$ contained in $M$.  But then $L$ is the intersection of a finite chain of non-trivial finite normal subgroups, so $L$ must itself be non-trivial, contrary to our earlier assumption.  In particular, $L$ is a lower bound for $\mc{F}$ in $\mc{M}$.  Hence by Zorn's lemma, $\mc{M}$ has a minimal element as required.
\end{proof}

\begin{prop}[{\cite[Proposition 4.6]{CRW2}}]\label{prop:locallypro-pi}
Let $G$ be a compactly generated \tdlc group and $U$ be a compact open subgroup. If the normal core of $U$ in $G$ is trivial, then $U$ is pro-$\pi$ for some finite set of primes $\pi$. 
\end{prop}

Given a locally compact group $G$, the intersection of all open normal subgroups of $G$ is denoted by $\Res(G)$ and called the \textit{discrete residual} of $G$. If $\Res(G) = \triv$, we say that $G$ is \textit{residually discrete}. A locally compact group has \textit{small invariant neighborhoods} (or is called a \textit{SIN-group}) if it admits a basis of conjugation invariant open identity neighborhoods.

\begin{prop}[{\cite[Corollary 4.1]{CM11}}]\label{prop:ResDiscr->SIN}
Let $G$ be a compactly generated \tdlc group.  Then $G$ is residually discrete if and only if it has a basis of identity neighborhoods consisting of compact open normal subgroups. In particular, if $G$ is residually discrete, then $G$ is a SIN-group. 
\end{prop}

Every closed subgroup of a SIN-group is a SIN-group.  The property also passes to quotients.

\begin{lem}\label{lem:SIN_quotient}
Let $G$ be a locally compact SIN-group, $H$ be a closed subgroup of $G$, and $K\normal H$ be closed.  Then $H/K$ is a SIN-group.
\end{lem}

\begin{proof}
Let $\mc{O}$ be a basis of open conjugation-invariant identity neighborhoods in $G$ and let $V/K$ be an open identity neighborhood in $H/K$.  Then $V$ is an identity neighborhood in $H$, so there is an identity neighborhood $W$ in $G$ such that $H \cap W \subseteq V$.  In turn, $O \subseteq W$ for some $O \in \mc{O}$.  It follows that $H \cap O \subseteq V$, and hence $(H \cap O)K/K$ is an open neighborhood of the identity in $H/K$ contained in $V/K$.  By construction, $H\cap O$ is conjugation-invariant under the action of $H$, and hence $(H\cap O)K/K$ is conjugation-invariant in $H/K$.   The set $\mc{O}' =  \{(H \cap O)K/K \mid O \in \mc{O}\}$ is thus a basis of conjugation-invariant open identity neighborhoods in $H/K$, and $H/K$ is a SIN-group.
\end{proof}

We recall a generation property of compactly generated \tdlc groups.

\begin{lem}[{\cite[Proposition 4.1]{CRW2}}]\label{lem:Cayley-Abels}
	Let $G$ be a compactly generated \tdlc group and $U \leq G$ be a compact open subgroup. Given any abstract subgroup $H \leq G$ such that $G = HU$, there exists a finite subset $\{h_1,\dots, h_n\} \subset H$ satisfying the following properties:
	\begin{enumerate}[label=(\roman*)]
		\item  $\la h_1, \dots, h_n\ra U = G$, and
		\item Given any $\tau_i \in h_i U$ for $i=1, \dots, n$, we have $\la \tau_1, \dots, \tau_n\ra U = G$.\qedhere
	\end{enumerate}
\end{lem}

Say a (not necessarily closed) subgroup $H$ of a locally compact group $G$ is \textit{cocompact} (or \textit{syndetic}) if there is a compact subset $X$ of $G$ such that $G = HX$.

We note that compact generation is stable under taking cocompact subgroups.

\begin{prop}[{\cite{MS59}}]\label{prop:cocompact_generation}
Let $G$ be a locally compact group and let $H$ be a closed cocompact subgroup of $G$.  Then $H$ is compactly generated if and only if $G$ is compactly generated.
\end{prop}

More generally, a variation on Lemma~\ref{lem:Cayley-Abels} allows us to control the structure of cocompact subgroups using compactly generated open subgroups.

\begin{lem}\label{lem:cocompact_restriction}Let $G$ be a \tdlc group, $H$ be a cocompact subgroup of $G$, and $O$ be a compactly generated open subgroup of $G$.  Then there exists a finite subset $\{h_1,\dots, h_n\} \subset H \cap O$ such that $\la h_1, \dots, h_n\ra$ is cocompact in $O$.
\end{lem}

\begin{proof}
Fix a compact open subgroup $U$ of $O$.  Since $H$ is cocompact in $G$, we can write $G = \bigcup^k_{i=1}Hx_iU$ for some finite set $\{x_1,\dots,x_k\}$. The set $HO$, which is a union of $(H,U)$-double cosets in $G$, can be written as $HO = \bigcup^d_{i=1}Hy_iU$ for $y_1,\dots,y_d \in HO$.  We can further choose $y_1,\dots,y_d \in O$, so we may write $O$ as $\bigcup^d_{i=1}(H \cap O)y_iU$. We deduce that $H \cap O$ is cocompact in $O$, and hence $L: = \overline{(H \cap O)}$ is compactly generated.  By Lemma~\ref{lem:Cayley-Abels}, $L = \la h_1,\dots, h_n \ra V$ for a compact subgroup $V$ of $L$ and a finite subset $\{h_1,\dots,h_n\}$ of $H \cap O$.  Thus $O = \bigcup^d_{i=1}\la h_1,\dots, h_n \ra Vy_i U$, and since each of the sets $Vy_iU$ is compact, we conclude that $\la h_1,\dots,h_n \ra$ is cocompact in $O$.
\end{proof}

\subsection{Automorphism groups of \tdlc groups}

Let $G$ be a locally compact group. We denote by $\Aut(G)$ the group of homeomorphic automorphisms of $G$. The group $\Aut(G)$ is naturally endowed with a topology, the so-called \textit{Braconnier topology} (also sometimes called the \textit{Birkhoff topology}), with respect to which it is a Hausdorff topological group.

We collect several facts concerning groups acting on \tdlc groups. Our first lemma is an easy exercise in the definitions.

\begin{lem}\label{lem:centralizer_action}
	Let $H$ be a \tdlc group acting continuously and faithfully on a \tdlc group $G$ by topological group automorphisms. Let $\sigma:H\rightarrow \Aut(G)$ be the induced homomorphism and $\Ad:G\rightarrow \Aut(G)$ be the natural homomorphism. The following hold. 
	\begin{enumerate}[label=(\roman*)]
		\item The group $L:=G\rtimes H$ is a \tdlc group under the product topology; 
		
		\item $\sigma$ is continuous; 
		
		\item $\CC_L(G)=\{(g,h)\in L \mid \Ad(g^{-1})=\sigma(h)\}$, where $G\leq L$ in the obvious way;
		\item if $\Z(G)=\triv$, then $\CC_{L/\CC_L(G)}(\pi(G))=\{1\}$, where $\pi:L\rightarrow L/\CC_L(G)$ is the usual projection; and 
		\item the map $\chi:\CC_L(G)\rightarrow G$ defined by $(g,h)\mapsto g^{-1}$ is a continuous, injective homomorphism.\qedhere
	\end{enumerate}
\end{lem}
\begin{proof}
	Claim (i) is a classical fact; see, for example, \cite[III.2.10 Proposition 28]{Bour98} or \cite[(6.20)]{HR}. 
	
	\medskip \noindent
	Claim (ii) follows from (i) and \cite[Theorem~(26.7)]{HR}.
	
	\medskip \noindent
	For (iii), choose  $(g,h)\in L$ and take $(x,1)\in G\leq L$. Then,
	\[
	(g,h)(x,1)(g,h)\inv = (g,h)(x,1)(\sigma(h^{-1})(g^{-1}),h^{-1})=(g\sigma(h)(x)g^{-1},1).
	\]
	We conclude that $(g,h) \in \CC_L(G)$ if and only if $g\inv xg=\sigma(h)(x)$ for all $x\in G$. That is to say, $(g,h)\in \CC_L(G)$ if and only if $\Ad(g^{-1})=\sigma(h)$.  Hence, $\CC_L(G)= \{(g,h)\in L \mid \Ad(g^{-1})=\sigma(h)\}$.
	
\medskip \noindent
For (iv), take $(g,h)\in L$ such that $\pi((g,h))\in \CC_{L/\CC_L(G)}(\pi(G))$. For every $(x,1)\in G$, it is then the case that $[(x,1),(g,h)]\in \CC_L(G)$. Let us compute this commutator:
	\[
	\begin{array}{rcl}
	[(x,1),(g,h)] & = & (x,1)(g,h)(x^{-1},1)(\sigma(h^{-1})(g^{-1}),h^{-1})\\
					& =& (xg\sigma(h)(x^{-1})g^{-1},1).
	\end{array}
	\]
	Since this commutator lies in $\CC_L(G)$, $\Ad(xg\sigma(h)(x^{-1})g^{-1})=\sigma(1)=\mathrm{id}$. The map $\Ad$ is injective, since $\Z(G)=\{1\}$, so we deduce that $xg\sigma(h)(x^{-1})g^{-1}=1$ for all $x$. Hence, $g^{-1}xg=\sigma(h)(x)$ for all $x\in G$, so $(g,h)$ is such that $\Ad(g^{-1})=\sigma(h)$. In view of Claim~(iii), we conclude that $(g,h)\in \CC_L(G)$, and therefore, $\CC_{L/\CC_L(G)}(\pi(G))=\triv$.
	
\medskip \noindent
(v). Let $(g_1,h_1),(g_2,h_2)\in \CC_L(G)$. We have  
	\[
	(g_1,h_1)(g_2,h_2)=(g_1\sigma(h_1)(g_2),h_1h_2)=(g_1\Ad(g_1^{-1})(g_2),h_1h_2)=(g_2g_1,h_1h_2).
	\]
	Therefore,
	\[
	\chi((g_1,h_1)(g_2,h_2))=g_1^{-1}g_2^{-1}=\chi((g_1,h_1))\chi((g_2,h_2)).
	\]
	We conclude that $\chi$ is a homomorphism, and it is clearly also continuous. That $\chi$ is injective follows from the fact that the action of $H$ is faithful.
\end{proof}

We record another basic property of the Braconnier topology. 

\begin{prop}[{\cite[Theorem~(26.8)]{HR}}]\label{prop:tdlc_action}
For any \tdlc group $G$, the automorphism group $\Aut(G)$ is totally disconnected. 
\end{prop}

\subsection{Quasi-centralizers}

The \textit{quasi-centralizer}, denoted by $\QC_H(K)$, is defined to be the set of $g \in H$ which centralize an open subgroup of $K$. We stress that quasi-centralizers are not closed in general.

The following basic observation about quasi-centralizers will be useful to control discrete subgroups.

\begin{lem}\label{lem:QC:basic}
Let $G$ be a topological group, $H$ be a closed subgroup, and $K$ be a discrete subgroup of $G$ that is normalized by $H$.  Then $K \le \QC_G(H)$.
\end{lem}

\begin{proof}
Fix $x \in K$.  The $H$-conjugacy class of $x$ is discrete. Since $h \mapsto hxh\inv$ is a continuous map from $H$ to $G$, it follows that $\CC_{H}(x)$ is open in $H$, so $x \in \QC_G(H)$.
\end{proof}

The \textit{quasi-center} of a \tdlc group $G$, denoted by $\QZ(G)$, is the set of elements with an open centralizer. For any compact open subgroup $U \leq G$, we see that $\QZ(G) = \QC_G(U)$. In particular, $\QZ(G)$ is normal in $G$.

We recall also the following relationship between the centralizer and quasi-centralizer of a subgroup.

\begin{lem}[{\cite[Lemma~3.8(i)]{CRW1}}]\label{lem:qcent_norm}
Let $G$ be a topological group and $H$ be a closed subgroup of $G$ such that $\QZ(H) = \triv$.  Then
\[
\QC_G(H) \cap \N_G(H) = \CC_G(H).
\]
\end{lem}

\subsection{Commensuration and localization}\label{sec:Localization}

Let $G$ be a   group  and $K \leq G$ be a subgroup. A subgroup $K' \leq G$ is \textit{commensurate} to $K$ if $K \cap K'$ is of finite index in both $K$ and in $K'$. Given $H \leq G$, we define the \textit{commensurator} of $K$ in $H$, denoted by $\Comm_H(K)$, to be the set of $g \in H$ such that $K$ and $ gKg^{-1}$ are commensurate. In topological groups, the relationship between the quasi-centralizer and the commensurator of a compact subgroup is analogous to the relationship between the centralizer and the normalizer.

\begin{lem}\label{lem:QC}
Let $G$ be a topological group and $H \leq G$ be any subgroup. For $K \leq G$ compact, the quasi-centralizer $\QC_H(K)$ is a normal subgroup of the commensurator $\Comm_H(K)$. 
\end{lem}

\begin{proof}
Given $g \in \QC_H(K)$, then $g$ centralizes an open subgroup $K'$ of $K$; since $K$ is compact, $|K:K'|$ has finite index, and hence $g \in \Comm_H(K)$, showing that $\QC_H(K)\leq \Comm_H(K)$. Taking $g \in \Comm_H(K)$ and $h \in \QC_H(K)$, the element $h$ centralizes an open subgroup of $K$, and $ghg^{-1}$ centralizes an open subgroup of $gKg^{-1}$. The element $ghg^{-1}$ thus centralizes an open subgroup of $g K g^{-1} \cap K$. Since $g \in \Comm_H(K)$, the subgroup $g K g^{-1} \cap K$ is open in $K$, hence $ghg^{-1} \in \QC_H(K)$, as required.
\end{proof}

Commensurators additionally admit a canonical \tdlc group topology.
	\begin{prop}[\cite{Reid}] If $K$ is a compact subgroup of a \tdlc group $G$, then $\Comm_G(K)$ admits a unique \tdlc group topology $\tau$ such that the inclusion map from $K$ to $\Comm_G(K)$ is continuous and open.
	\end{prop}
	\begin{cor}
		The inclusion map $\iota:(\Comm_G(K),\tau)\rightarrow G$ is a continuous, injective homomorphism. 
	\end{cor}
	
The \tdlc group topology $\tau$ on $\Comm_G(K)$ is called the \textit{$K$-localized topology}. The  group  $\Comm_G(K)$ endowed with the $K$-localized topology is denoted by $G_{(K)}$ and called the localization of $G$ at $K$. There is a natural source of localizations that produces dense locally compact subgroups of the ambient group $G$.

\begin{defn}
	Let $G$ be a profinite group and let $p$ be a prime.  A \textit{pro-$p$-Sylow subgroup} of $G$ is a closed subgroup $S$ of $G$, such that $S$ is maximal among the pro-$p$ subgroups of $G$.  Equivalently, by Sylow's theorem, $S$ is a closed subgroup of $G$ with the property that for every open normal subgroup $N$ of $G$, the order of $SN/N$ is a power of $p$ and the index $|G:SN|$ is coprime to $p$.
\end{defn}

\begin{thm}[{\cite[Theorem 1.2]{Reid}}]\label{thm:Reid_localizations}
	Suppose that $G$ is a \tdlc group with $U \leq G$ a compact open subgroup. If $S \leq U$ is a pro-$p$-Sylow subgroup of $U$ for some prime $p$, then $\Comm_G(S)$ is dense. Further, the isomorphism type of $G_{(S)}$ is independent of the choice of $U$ and pro-$p$-Sylow subgroup $S$. 
\end{thm}

In view of the previous theorem, we set $G_{(p)}$ to be the localization of $G$ at some (equivalently) any pro-$p$-Sylow subgroup of a compact open subgroup, following \cite{Reid}. We call $G_{(p)}$ the \textit{$p$-localization} of $G$. 

The $p$-localizations provide a tool to approximate arbitrary \tdlc groups by \tdlc groups that are locally pro-$p$. A useful tool to understand when the compact open subgroups of the $p$-localization are topologically finitely generated is provided by a consequence of Tate's normal $p$-complement theorem on finite groups. A profinite group is \textit{topologically finitely generated} if it admits a dense finitely generated subgroup.

For a prime $p$, the set of primes different from $p$ is denoted by $p'$. For a set of primes $\pi$, the {$\pi$-core}\index{$\pi$-core} of a profinite group $U$, denoted by $O_{\pi}(U)$\index{$O_{\pi}(U)$}, is the subgroup generated by all normal pro-$\pi$ subgroups. We write $O^p(G)$ for the smallest normal subgroup such that $G/O^p(G)$ is a pro-$p$ group.  Notice that $O^p(G) \ge O_{p'}(G)$, with equality if and only if $O^p(G)$ is a pro-$p'$ group.

For a group $L$ and $p$ a positive integer, $L^p$ denotes the subgroup $\grp{l^p\mid l\in L}$. 

\begin{thm}[\cite{Tate64}]\label{thm:Tate}
Let $G$ be a finite group and let $S$ be a $p$-Sylow subgroup of $G$.  If $S/S^p[S,S]$ is isomorphic to $G/G^p[G,G]$, then $S \cap O^p(G) = \triv$.
\end{thm}

The following corollary was given as \cite[Corollary~2.2.2]{Reid_PhD}; a similar observation was also used in \cite{Melnikov96}.  For clarity, we reprove it here. 

\begin{cor}\label{cor:Tate}
Let $U$ be a profinite group, $p$ be a prime, and $S\leq U$ be a pro-$p$-Sylow subgroup of $U$. If $S$ is topologically finitely generated, then $U/O_{p'}(U)$ is virtually pro-$p$.
\end{cor}

\begin{proof}
Since $S$ is topologically finitely generated, we see that $S/\ol{S^p[S,S]}$ is a topologically finitely generated elementary abelian group, and hence $\ol{S^p[S,S]}$ is open in $S$.  It follows that there is an open normal subgroup $N$ of $U$ such that $S \cap N \le \ol{S^p[S,S]}$.  Let $V := SN$, let $M$ be an open normal subgroup of $U$ contained in $N$, and write $\tilde{H}$ for the image of any $H \le U$ under the quotient map from $U$ to $U/M$.  The group $\tilde{S}$ is a pro-$p$-Sylow subgroup of $\tilde{V}$ with $\tilde{S} \cap \tilde{N} \le \tilde{S}^p[\tilde{S},\tilde{S}]$, and $\tilde{V} = \tilde{S}\tilde{N}$.  Thus, $\tilde{V}/\tilde{N}$ is a $p$-group of the same number of generators as $\tilde{S}$, ensuring that $\tilde{V}/\tilde{V}^p[\tilde{V},\tilde{V}] \simeq \tilde{S}/\tilde{S}^p[\tilde{S},\tilde{S}]$. Theorem~\ref{thm:Tate} now implies that $\tilde{S} \cap O^p(\tilde{V}) = \triv$.  Since $M$ is allowed to range over a base of identity neighborhoods, we conclude that $S \cap O^p(V) = \triv$.  Therefore, $O^p(V)$ is a pro-$p'$ group, so $O^p(V) = O_{p'}(V)$.  Letting $W$ be the normal core of $V$ in $U$, we see that 
\[
W \cap O_{p'}(V) = O_{p'}(W) \le O_{p'}(U).
\]
It follows that $WO_{p'}(U)/O_{p'}(U)$ is pro-$p$, and hence $U/O_{p'}(U)$ is virtually pro-$p$.
\end{proof}

\section{Connected groups and locally pro-nilpotent groups}\label{sec:denseLCLie}

Some non-discrete simple locally compact groups $H$ have the property that every dense embedding $\psi \colon H \to G$ into an arbitrary locally compact group $G$ is surjective. Examples include the simple Lie groups (see   \cite[Corollary~1.1]{Omori}) and the simple algebraic groups over local fields (see \cite[Corollary~1.4]{BadGel}). We begin our explorations by noting that  simple Lie groups and simple algebraic groups over local fields also enjoy a dual property: Denoting such a group by $G$, every dense embedding $\psi \colon H \to G$ of a non-discrete locally compact group $H$ is surjective. 

\subsection{Connected groups}

\begin{prop}[{\cite[Ch.~XVI, Theorem~2.1]{Hochschild}}]\label{prop:Hochschild}
Let $G$ and $ H$ be connected Lie groups and $\psi \colon H \to G$ be a dense embedding. Then $\psi(H)$ is normal in $G$, and the quotient $G/\psi(H)$ is abelian. 
\end{prop}

\begin{cor} \label{cor:ConnectedSimple}
If $G$ is a connected locally compact group that is topologically simple, then every proper dense locally compact subgroup is discrete.
\end{cor}

\begin{proof} 
Suppose that $H$ is a locally compact group and that $\psi \colon H \rightarrow G$ is a dense embedding. Let us suppose further that $H$ is non-discrete. We argue that $\psi$ is onto.
	 
By the solution to Hilbert's fifth problem, see \cite[Theorem 4.6]{MZ}, $G$ is a connected Lie group over $\Rb$. Applying \cite[III.8.2 Corollary 1]{Bour_LG_89}, $H$ is also a Lie group over $\Rb$. Since $H$ is non-discrete it follows that $H^\circ$ is non-trivial. The non-trivial group $\psi(H^\circ)$ is normalized by the dense subgroup $\psi(H)$ of $G$. Therefore  $\psi(H^\circ)$ is dense since $G$ is topologically simple. We deduce from Proposition~\ref{prop:Hochschild} that $\psi(H^\circ)$ is a normal subgroup of $G$ containing $[G, G]$. Invoking the classical fact that in a topologically simple connected Lie group $G$, we have $G = [G,G]$, we conclude that $\psi(H^{\circ})=G$. The homomorphism  $\psi \colon H\rightarrow G$ is thus onto. 
\end{proof}

\subsection{Pro-nilpotent groups with topologically finitely generated compact open subgroups}	

As mentioned in the introduction, \tdlc groups in general possess non-discrete proper dense locally compact subgroups. The following result shows that the existence of such subgroups depends on the structure of compact open subgroups. The following result should be compared with  Proposition~\ref{prop:Hochschild}. 

Recall that subgroup of a topological group is called \textit{locally normal}\index{locally normal} if its normalizer is open.

\begin{prop}\label{prop:Loc-pro-nilp}
Let $G$ and $H$ be \tdlc groups and $\psi \colon H \to G$ be a dense embedding. Assume that $G$ has a compact open subgroup which is a topologically finitely generated pronilpotent pro-$\pi$ group  for a finite set of primes $\pi$.

\begin{enumerate}[label=(\roman*)]
\item \label{it:Loc-pro-nilp1} 
There is a compact open subgroup $T \leq H$ such that $\psi(T)$ is a commensurated  locally normal subgroup of  $G$, and the normal closure $\lla \psi(T) \rra_G$ is contained in $\psi(H)$. 

\item \label{it:Loc-pro-nilp2}
If $G$ is topologically simple and $H$ is non-discrete, then $\psi(H)$ is normal in $G$, and the quotient $G/\psi(H)$ is abelian.\qedhere
\end{enumerate}
\end{prop}

\begin{proof}
Let $U$ be a compact open subgroup of $G$ that is a topologically finitely generated pronilpotent pro-$\pi$ group. By \cite[Proposition~2.3.8]{RibesZalesskii}, $U$ is a direct product of pro-$p$ groups $U_p$, one for each $p \in \pi$; it follows that $U_p$ is topologically finitely generated for each $p$.  By \cite[Proposition~2.8.10]{RibesZalesskii}, the Frattini subgroup $\Phi(U_p)$ of $U_p$ is open in $U_p$, and hence $\Phi(U) = \prod_{p \in \pi}\Phi(U_p)$ is open in $U$. Since $H$ has dense image in $G$, it contains a finite subset $X\subset H$ such that $U = \la \psi(X) \ra \Phi(U)$, so $U = \overline{\la \psi(X) \ra}$, by the properties of the Frattini subgroup. 

Let now $S$ be a compact open subgroup of $H$ contained in the open subgroup $\psi^{-1}(U)$. We now consider $J := \la X \cup S \ra$, which is a compactly generated open subgroup of $H$. The image $\psi(J)$ is contained in the profinite group $U$, hence $J$  is residually finite. Proposition~\ref{prop:ResDiscr->SIN} now ensures that there is an open subgroup $T \leq S$ which  is normal in $J$. The image $\psi(T)$ is a compact subgroup of $G$ whose normalizer contains $\psi(X)$, so $\N_G(\psi(T))$ contains $U$ and is thus open. In particular, $\Comm_G(\psi(T))$ is open and dense, since it contains $\psi(H)$, hence  $G=\Comm_G(\psi(T))$. We have $G = \psi(H) U = \psi(H) \N_G(\psi(T))$, so the conjugation action of $\psi(H)$ is transitive on the $G$-conjugacy class of $\psi(T)$. This confirms that  $\lla \psi(T) \rra_G$ is contained in $\psi(H)$, proving \ref{it:Loc-pro-nilp1}. 

Assume now that $G$ is topologically simple and $H$ non-discrete. Then $T$ is non-trivial. Hence the group $N := \lla \psi(T) \rra_G$ is a non-trivial normal subgroup of $G$ and is thus dense. Since $U$ is a topologically finitely generated pronilpotent pro-$\pi$ group, its derived group $[U, U]$ is closed and contained in every dense normal subgroup of $U$ by \cite[Theorem~5.21]{CRW2}. In particular, $U \cap N$ contains $[U, U]$, and it follows that $G/N = NU/N \simeq U / U \cap N$ is abelian. We conclude that $[G, G] \leq N \leq \psi(H)$ by \ref{it:Loc-pro-nilp1}, and assertion \ref{it:Loc-pro-nilp2} follows. 
\end{proof}

In Subsection~\ref{sec:ConcreteEx}, we give an example showing that ``topologically finitely generated" cannot be removed from the statement of Proposition~\ref{prop:Loc-pro-nilp}.

\begin{cor}\label{cor:fg_abs_simple}
If $G\in \Ss$ is abstractly simple and has a topologically finitely generated pro-nilpotent compact open subgroup, then every proper dense locally compact subgroup is discrete.	
\end{cor}
\begin{proof}
By Proposition~\ref{prop:locallypro-pi}, every compact open subgroup of $G$ is locally pro-$\pi$ for a finite set of primes $\pi$. The required assertion is thus an immediate consequence of Proposition~\ref{prop:Loc-pro-nilp}\ref{it:Loc-pro-nilp2}. 
\end{proof}
	
\begin{cor}\label{cor:AlgGroups}
Let  $\mathbf G$ be an absolutely simple, simply connected, isotropic algebraic group over a non-archimedean local field $k$. Every proper dense locally compact subgroup of $\mathbf G(k)/\Z(\mathbf G(k))$ is discrete. 
\end{cor}
\begin{proof}
The group $G = \mathbf G(k)/\Z(\mathbf G(k))$ belongs to the class $\Ss$ and is abstractly simple, see  \cite[Theorems 2.2 and 2.4]{CapStu} and \cite{Tits}. In view of Corollary~\ref{cor:fg_abs_simple}, it remains to show that $G$ has a compact open pro-$p$ subgroup that is topologically finitely generated. The result \cite[Theorem 2.6]{CapStu} ensures that $G(k)/\Z(G(k))$ admits a non-virtually abelian hereditarily just infinite compact open pro-$p$ subgroup. Hereditarily just-infinite pro-$p$ groups are topologically finitely generated, so the desired result follows.
\end{proof}

\begin{rmk}\label{rem:KM}
We remark that there also exist groups of a non-algebraic origin in $\Ss$ that satisfy the hypotheses of Corollary~\ref{cor:fg_abs_simple}. For example, many locally compact Kac--Moody groups do; see \cite{CaRe} and \cite{Marquis}. Similarly, the group $G\in \Ss$ admitting the profinite completion of the Grigorchuk group as a compact open subgroup satisfies the hypotheses; see \cite[Theorem 4.16]{BEW}. 
\end{rmk}

\section{Regionally expansive groups}

We here isolate the class of regionally expansive groups. The results herein suggest that regional expansiveness is a weak form of compact generation with better stability properties. This class will play a central role in the definition of the class $\Rs$ in the next section.

 \subsection{Definition and basic properties}\label{sec:rf_def}
 
 The following definition comes from the literature on topological dynamical systems (see for instance \cite{Lam70}).
 
 \begin{defn}
 Let $X$ be a uniform space and let $G$ be a group of homeomorphisms of $X$.  The action of $G$ is \emph{expansive} if there is some entourage $E$ such that, whenever $(x,y)$ is a pair of points such that $(g.x,g.y) \in E$ for all $g \in G$, then $x=y$.
 \end{defn}
 
When $G$ is a topological group, it is natural to equip $G$ with the right uniformity and let it act on itself by conjugation; we then say that $G$ is \emph{expansive} as a topological group if this action is expansive.  There is an easy equivalent characterization of what it means for a topological group to be expansive.
 
\begin{lem}\label{lem:expansive_basic}
A topological group $G$ is expansive if and only if there is an identity neighborhood $W$ in $G$ such that $\bigcap_{g \in G}gWg\inv = \{1\}$.
\end{lem}
 
 \begin{proof}
Suppose $G$ is expansive.  Then there is an entourage $E$ such that whenever $x,y \in G$ are such that $(gxg\inv,gyg\inv) \in E$ for all $g \in G$, then $x=y$.  Since $E$ is an entourage, there is an identity neighborhood $W$ such that $(x,1) \in E$ for all $x \in W$.  Thus if $x \in G$ is such that $gxg\inv \in W$ for all $g \in G$, then $x=1$.  We conclude that $\bigcap_{g \in G}gWg\inv = \{1\}$.
 
Conversely, suppose that $W$ is an identity neighborhood such that $\bigcap_{g \in G}gWg\inv = \triv$, and let $E$ be the entourage $\{(x,y) \mid x,y \in G, x \in Wy\}$.  Let $x,y \in G$ be such that $(gxg\inv,gyg\inv) \in E$ for all $g \in G$.  Then for all $g \in G$ we have $gxg\inv \in Wgyg\inv$, or in other words, $gxy\inv g\inv \in W$.  Since  $\bigcap_{g \in G}gWg\inv = \{1\}$, it follows that $xy\inv = 1$, so $x=y$.  Thus $G$ is expansive.
 \end{proof}
 
In particular, every discrete group is expansive, and every group in $\Ss$ is expansive.  A more subtle example of a compactly generated expansive group is $G:=A_5^{\Zb}\rtimes \Zb$. The group $G$ has a compact open normal subgroup, namely $A_5^{\Zb}$, but any proper open subgroup of $A_5^{\Zb}$ has trivial normal core.  On the other hand, it is easy to see that a non-discrete compact group cannot be expansive: indeed, every compact group has arbitrarily small invariant neighborhoods of the identity.
 
If $G$ is a \tdlc group, then we can express the property of being expansive in terms of closed normal subgroups.  We will use these equivalent forms of the definition without further comment.

\begin{lem}\label{lem:expansive_tdlc}
Let $G$ be a \tdlc group.  Then the following are equivalent:
\begin{enumerate}[label=(\roman*)]
\item $G$ is expansive;
\item There is a compact open subgroup $W$ such that $\bigcap_{g \in G}gWg\inv = \triv$;
\item Given a filtering family $(N_i)_{i \in I}$ of non-trivial compact normal subgroups of $G$, then $\bigcap_{i \in I}N_i$ is non-trivial.\qedhere
\end{enumerate}
\end{lem}

\begin{proof}
Suppose $G$ is expansive; given Lemma~\ref{lem:expansive_basic}, let $W$ be an identity neighborhood in $G$ such that $\bigcap_{g \in G}gWg\inv = \{1\}$.  Then by Van Dantzig's theorem, $W$ contains a compact open subgroup $W'$ of $G$; we then have $\bigcap_{g \in G}gW'g\inv = \{1\}$.  Thus $(i)$ implies $(ii)$.

Suppose $W$ is a compact open subgroup of $G$ such that $\bigcap_{g \in G}gWg\inv = \triv$, and let $(N_i)_{i \in I}$ be a filtering family of non-trivial compact normal subgroups of $G$.  Then for each $i \in I$, we see that $N_i \nleq W$.  Fix $i_0 \in I$.  Since $(N_i)_{i \in I}$ is a filtering family, we have $\bigcap_{i \in I}N_i = \bigcap_{i \ge i_0}N_i$; in turn, $\bigcap_{i \ge i_0}N_i$ contains $\bigcap_{i \ge i_0}(N_i \setminus W)$.  Now $(N_i \setminus W)_{i \ge i_0}$ is a filtering family of non-empty closed subsets of the compact set $N_{i_0}$, so it has non-empty intersection.  Since $1 \in W$ it follows that $\bigcap_{i \ge i_0}N_i$ is non-trivial.  Thus $(ii)$ implies $(iii)$.

We show that $(iii)$ implies $(i)$ via the contrapositive.  Suppose $G$ is not expansive; by Van Dantzig's theorem, there is a filtering family $(W_i)_{i \in I}$ of compact open subgroups of $G$ that forms a base of identity neighborhoods.  For each $i \in I$ let $N_i = \bigcap_{g \in G}gW_ig\inv$.  Then $N_i$ is non-trivial for each $i \in I$ by Lemma~\ref{lem:expansive_basic}.  Thus $(N_i)_{i \in I}$ is a filtering family of non-trivial compact normal subgroups of $G$ with trivial intersection, showing that $(iii)$ is false.  Thus $(iii)$ implies $(i)$, completing the cycle of implications.
\end{proof}

In particular, every \tdlc group $G$ is approximated by its expansive quotients, in the sense of being an inverse limit of them: given an identity neighborhood $O$ in $G$, there is a compact normal subgroup $K \subseteq O$ such that $G/K$ is expansive, where $K$ is obtained as the normal core in $G$ of some compact open subgroup contained in $O$.  Given this fact, we cannot expect to prove much about expansive \tdlc groups \textit{per se}.  However, assuming that $G$ is a non-discrete, non-compactly generated \tdlc group, it turns out to be a surprisingly powerful assumption to require a \emph{compactly generated} open subgroup of $G$ to be expansive.

\begin{defn}
 	A \tdlc group $G$ is called \emph{regionally expansive} if some compactly generated open subgroup $O$ is expansive. In other words, there is a compact open subgroup $W$ of $O$ with trivial normal core in $O$.  
\end{defn}

Note that our definition here is consistent with the use of ``regionally'' in Remark~\ref{rmk:regional}: if some compact open subgroup $W$ of $G$ has trivial normal core in the compactly generated open subgroup $O$ of $G$, then $W$ also has trivial normal core in every overgroup of $O$ in $G$, so $G$ is a directed union of expansive compactly generated open subgroups. Note further that a compactly generated \tdlc group is regionally expansive if and only if it is expansive, so within the class of expansive \tdlc groups (which, as noted, approximates every \tdlc group), ``regionally expansive'' is a generalization of ``compactly generated.''
 
The next lemma is one of the primary tools used throughout the present article. We stress that in this lemma the subgroup $R$ need not be closed, and we will indeed often use this lemma for non-closed $R$.

\begin{lem}\label{lem:no_cocompact}
	Let $G$ be a \tdlc group, $U$ be a compact open subgroup of $G$, and $R$ be a subgroup such that $G = RU$. If $R \cap U$ contains a non-trivial normal subgroup of $R$, then $U$ contains a non-trivial compact normal subgroup of $G$. 
\end{lem}
\begin{proof}
	Let $K \leq R \cap U$ be a non-trivial normal subgroup of $R$. The normalizer $\N_G(K)$ contains $R$, and since $G = RU$, we see that $U$ acts transitively on the conjugacy class of $ K$ in $G$. Any conjugate of $K$ in $G$ is thus contained in $U$. We deduce that $\ol{\lla K \rra_G}$ is a subgroup of $U$, and the lemma follows.
\end{proof}

Our first observation concerning regionally expansive groups is immediate. Recall that compactly generated, and more generally $\sigma$-compact \tdlc groups, are second countable modulo a compact normal subgroup; see \cite[Theorem 8.7]{HR}. 
 
 \begin{lem}\label{lem:FirstCountable}
  If $G$ is a regionally expansive \tdlc group, then $G$ is first countable. In particular, $\sigma$-compact regionally expansive \tdlc groups are second countable.
 \end{lem}
 
Discrete groups are regionally expansive, so we can have $G = \QZ(G)$.  However, aside from this case, no regionally expansive group has dense quasi-center. A locally compact group is called \textit{quasi-discrete} if its quasi-center is dense. 
 
 \begin{lem}\label{lem:nonqd}
 Let $G$ be a regionally expansive \tdlc group.  If $G$ is quasi-discrete, then $G$ is discrete.
 \end{lem}
 
 \begin{proof}
Suppose that $G$ is quasi-discrete; that is, $\QZ(G)$ is dense in $G$. Letting $\{O_i\}_{i\in I}$ be a directed system of compactly generated open subgroups of $G$ with directed union equal to $G$, we see that each $O_i$ is quasi-discrete, and in view of \cite[Proposition 4.3]{CM11}, each $O_i$ is also a SIN group. On the other hand, for sufficiently large $i$, the $O_i$ do not have arbitrarily small non-trivial compact normal subgroups.  The only way to satisfy these conditions is that the $O_i$ are discrete, hence $G$ is discrete.
\end{proof}
 
\subsection{Minimal normal subgroups and the socle}

\begin{defn}
Let $G$ be a \tdlc group.  Let $\mc{M}(G)$ be the (possibly empty) set of minimal non-trivial closed normal subgroups of $G$. The \emph{socle} $\Soc(G)$ of $G$ is the closed subgroup generated by all minimal non-trivial closed normal subgroups of $G$. In other words, $\Soc(G) = \overline{\la \bigcup \mc{M}(G)\ra}$. 
\end{defn}

The following result implies that for a non-trivial regionally expansive group $G$ with trivial quasi-center, the set $\mc{M}(G)$ is necessarily non-empty.
   
\begin{lem}\label{lem:minimal_geofaithful}
Let $G$ be a regionally expansive \tdlc group with trivial quasi-center.
\begin{enumerate}[label=(\roman*)]
\item Every filtering family of non-trivial closed normal subgroups of $G$ has a non-trivial intersection.  In particular, every non-trivial closed normal subgroup of $G$ contains a minimal non-trivial closed normal subgroup of $G$.
\item Let $O$ be a regionally expansive open subgroup of $G$.  For every $M \in \mc{M}(G)$, there exists  $N \in \mc{M}(O)$ such that $M =  \ngrp{N}_G$.
\item If $\mc{M}(G)$ is infinite, then $G$ has a non-trivial characteristic abelian subgroup.\qedhere
\end{enumerate}
\end{lem}

\begin{proof}
Suppose $\mc{F}$ is filtering family of non-trivial closed normal subgroups of $G$ and fix $O$ a compactly generated expansive open subgroup of $G$. The induced family $\mc{F}_O:=\{N\cap O\mid N\in \mc{F}\}$ is a filtering family of closed normal subgroups of $O$. Additionally, none of the $M\in \mc{F}_O$ are discrete since $\QZ(G)=\{1\}$. The subgroup $O$ has a compact open subgroup with trivial normal core and no non-trivial discrete normal subgroups, hence Theorem~\ref{thm:minimal_normal} ensures $\bigcap \mc{F}_O$ is non-trivial. We conclude that $\bigcap\mc{F}$ is non-trivial.  The remainder of part (i) now follows by Zorn's lemma.

Let $O$ be a regionally expansive open subgroup of $G$ and let $M \in \mc{M}(G)$.  Then $M$ is non-discrete since $\QZ(G)=\triv$, so $M \cap O$ is a non-trivial closed normal subgroup of $O$.  It follows by part (i) that there exists $N \in \mc{M}(O)$ such that $N \le M \cap O$.  The group $M' = \ngrp{N}_G$ is then a non-trivial closed normal subgroup of $G$. By construction $M' \le M$, and by minimality, $M' = M$, proving (ii).

Suppose that $\mc{M}(G)$ is infinite.  Set $R := \Soc(G)$, take $\mc{F}$ to be the set of finite subsets of $\mc{M}(G)$, and for each $F \in \mc{F}$, let $C_F := \bigcap_{M \in F}\CC_R(M)$.  The subgroup $C_F$ contains $N$ for all $N \in \mc{M}(G) \setminus F$, so it is a non-trivial normal subgroup of $G$.  The set $\{C_F \mid F \in \mc{F}\}$ is a filtering family.  By part (i), it follows that $C:= \bigcap_{F \in \mc{F}}C_F$ is non-trivial.  The construction of $C$ ensures that $C$ is characteristic and that $C \le \CC_G(M)$ for all $M \in \mc{M}(G)$. Since $C \le \Soc(G)$, we infer that $C$ is abelian, proving (iii).
\end{proof}

An extension of Lemma~\ref{lem:minimal_geofaithful} will be established in Proposition~\ref{prop:minimal_A-semisimple} below.

\subsection{Minimal normal subgroups and [A]-semisimplicity}

\begin{defn}
A \tdlc group $G$ is \emph{[A]-semisimple} if it admits no non-trivial locally normal abelian subgroups and it has a trivial quasi-center.
\end{defn}

The [A]-semisimplicity condition implies in particular that $G$ is locally C-stable in the sense of \cite{CRW1}.  As shown in \cite[Theorem A]{CRW2}, every element of $\Ss$ is [A]-semisimple.  In this subsection, we discover a more general connection between [A]-semisimplicity and the normal subgroup structure in the class of regionally expansive groups.

We remark first that [A]-semisimplicity can be characterized in terms of quasi-discrete locally normal subgroups.

\begin{prop}[{\cite[Theorem~3.19 and Proposition~6.17]{CRW1}}]\label{prop:[A]-semisimplicity_LN}
For $G$ a \tdlc group, the following are equivalent.
\begin{enumerate}[label=(\roman*)]
\item $G$ is [A]-semisimple;
\item $G$ does not have any non-trivial quasi-discrete closed locally normal subgroups;
\item Every closed locally normal subgroup of $G$ has trivial quasi-center.\qedhere
\end{enumerate}
\end{prop}

In an [A]-semisimple group, we have good control of quasi-centralizers of locally normal subgroups and some useful equivalent conditions for two locally normal subgroups to commute.

\begin{lem}[{See \cite[Theorem~3.19]{CRW1}}]\label{lem:[A]-semisimplicity_qcent}
Let $G$ be an [A]-semisimple \tdlc group and let $H$ be a closed locally normal subgroup of $G$.  Then
\[
\QC_G(H) = \CC_G(H).
\]
\end{lem}

\begin{lem}\label{lem:[A]-semisimple:commuting}
Let $G$ be an [A]-semisimple \tdlc group and let $H$ and $K$ be locally normal subgroups of $G$.  Then $\CC_H(K) = \CC_H(H \cap K)$. Moreover, the following assertions are equivalent.
\begin{enumerate}[label=(\roman*)]
\item $H \cap K = \triv$;
\item $[H,K] = \triv$;
\item There is an open subgroup of $H$ that commutes with an open subgroup of $K$.
\end{enumerate}
In particular, for any compact open subgroup $U \leq G$, we have $\CC_G(H) = \CC_G(K)$ if and only if $\CC_U(H) = \CC_U(K)$. 
\end{lem}

\begin{proof}
The equivalence of (i) and (ii) is given by \cite[Theorem~3.19]{CRW1}.  Clearly (ii) implies (iii).  Conversely, if (iii) holds then $[H \cap U, K \cap U] = \triv$ for some open subgroup $U$ of $G$.  Applying Lemma~\ref{lem:[A]-semisimplicity_qcent}, we see that $K \cap U$ commutes with $H$, and applying Lemma~\ref{lem:[A]-semisimplicity_qcent} again, $H$ commutes with $K$, proving (ii).

It remains to show that $\CC_H(K) = \CC_H(H \cap K)$.  Set $L: = \CC_H(H \cap K)$ and let $U$ be an open subgroup of $G$ that normalizes $H$ and $K$.  The groups $K \cap U$ and $L \cap U$ thus normalize each other, so their commutator is contained in their intersection. In other words, $M := [K \cap U,L \cap U]$ is a locally normal subgroup of $G$ such that $M \le K \cap L \cap U$.  Since $M$ is contained in both $H \cap K$ and a subgroup that centralizes $H\cap K$, we deduce that $M$ is abelian, hence it is trivial.  We conclude that $K \cap U$ and $L \cap U$ commute.  The equivalence of (ii) and (iii) then ensures that $K$ and $L$ commute, so $\CC_H(H \cap K) \le \CC_H(K)$.  On the other hand, $\CC_H(K) \le \CC_H(H \cap K)$, so equality holds.

Given a subgroup $J \leq G$, we define $\CC^0_G(J) := J$ and $\CC^{n+1}_G(J):= \CC_G(\CC^n_G(J))$. We observe that $\CC^2_G(J) \ge J$ and $\CC^3_G(J) = \CC_G(J)$.
Let $U \leq G$ be a compact open subgroup. Clearly, the equality $\CC_G(H) = \CC_G(K)$ implies that $\CC_U(H) = \CC_U(K)$. Assume conversely that $\CC_U(H) = \CC_U(K)$. By the equivalence of (ii) and (iii) above, we have $\CC_G(\CC_G(L)) = \CC_G(\CC_G(L) \cap U) = \CC_G(\CC_U(L))$ for any locally normal subgroup $L$. Therefore  
$$\CC_G(H) = \CC^3_G(H) = \CC^2_G(\CC_G(H)) = \CC^2_G(\CC_U(H)) =  \CC^2_G(\CC_U(K)) =  \CC^3_G(K) =  \CC_G(K).$$ 
\end{proof}

Our next proposition is an analogue of Proposition~\ref{prop:filtering} for regionally expansive groups without quasi-discrete normal subgroups.

\begin{prop}\label{prop:minimal_A-semisimple}
Let $G$ be a regionally expansive \tdlc group that has no non-trivial quasi-discrete closed normal subgroups.
\begin{enumerate}[label=(\roman*)]
\item $\mc{M}(G)$ is a finite set.
\item We have $\bigcap_{N \in \mc{M}(G)}\QC_G(N) = \triv$.
\item Let $O \le G$ be a regionally expansive open subgroup of $G$.  Then there is a surjective map from $\mc{M}(O)$ to $\mc{M}(G)$ given by sending $M \in \mc{M}(O)$ to $N = \ngrp{M}_G$.
\item Let $H$ be a regionally expansive subgroup of $G$ such that $\N_G(H)$ is open in $G$ and suppose that $M \in \mc{M}(H)$ is non-abelian.  Then $\ngrp{M}_G \in \mc{M}(G)$.
\item Suppose that there is a compactly generated open subgroup of $G$ with no non-trivial abelian normal subgroups.  Then the number $|\mc{M}(G)|$ is the least value of $|\mc{M}(O)|$ as $O$ ranges over expansive compactly generated open subgroups.\qedhere
\end{enumerate}
\end{prop}

\begin{proof}
Part (i) is immediate from Lemma~\ref{lem:minimal_geofaithful} and the fact that $G$ has no non-trivial quasi-discrete normal subgroups; note that abelian groups are quasi-discrete.

\medskip

Set $Q: = \bigcap_{N \in \mc{M}(G)}\QC_G(N)$ and suppose that $Q$ is non-trivial.  By Lemma~\ref{lem:minimal_geofaithful}, there is some $R \in \mc{M}(G)$ such that $R \le \overline{Q}$.  Notice that $Q \le \QC_G(R)$.  Our hypothesis ensures that $R$ is not quasi-discrete, so $R \cap Q$ is not dense in $R$.  Hence $R \cap Q = \triv$ by minimality.  The group $R$ then centralizes $Q$. We conclude that $R$ is central in $\overline{Q}$, which is impossible by the fact that $R$ is not quasi-discrete.  This proves (ii).

\medskip

Fix $O$ an open subgroup of $G$ that is regionally expansive.  By Lemma~\ref{lem:minimal_geofaithful}, every $N \in \mc{M}(G)$ arises as $N = \ngrp{M}_G$ for some $M \in O$.  To prove (iii), it therefore suffices to show that every $M \in \mc{M}(O)$ is contained in some $N \in \mc{M}(G)$.

Let $M \in \mc{M}(O)$.  Since $M \neq \triv$, part (ii) implies that there is some $N \in \mc{M}(G)$ such that $M \not\le \QC_G(N)$.  In particular, $M$ does not commute with $N \cap O$.  Since $M$ and $N \cap O$ are normal in $O$, it follows that $M \cap N \cap O > \triv$; since $M \in \mc{M}(O)$, in fact we must have $M \le N \cap O$, so $M \le N$.  This completes the proof of (iii).

\medskip

Let $H$ be a regionally expansive subgroup of $G$ such that $\N_G(H)$ is open and let $M \in \mc{M}(H)$ be such that $M$ is non-abelian.  By part (ii), there is some $N \in \mc{M}(G)$ such that $M$ is not contained in $\QC_G(N)$.  We will show that $N = \ngrp{M}_G$, which in particular implies that $\ngrp{M}_G \in \mc{M}(G)$.

We may assume for a contradiction that $N$ does not contain $M$.  Then the intersection $N \cap H$ is normal in $H$ and does not contain $M$, so $N \cap M = \triv$ by minimality.  In particular, we have $\N_N(H) \cap M = \triv$.  Let $N' = [\N_N(H),M]$.  Since $\N_N(H)$ is open and $M$ is not contained in $\QC_G(N)$, we see that $N'$ is non-trivial. At the same time, $M$ normalizes $\N_N(H)$, so $N' \le \N_N(H)$. Hence, $N' \nleq M$.  The  subgroup $\N_N(H)$ therefore does not normalize $M$.  Let $k \in \N_N(H)$ be such that $kMk\inv \neq M$. Hence, $kMk\inv \cap M = \triv$, and $M$ and $kMk\inv$ commute.  Since $M$ is non-abelian, we can take $m_0,m_1\in M$ such that $[m_0,m_1]\neq 1$. Thus,
\[
[m_0,[m_1,k]]=m_0m_1(km_1^{-1}k^{-1})m_0^{-1}(km_1k^{-1})m_1^{-1}=m_0m_1m_0^{-1}m_1^{-1}=[m_0,m_1].
\]
The group $N$ is normal in $G$, so $[m_0,m_1]=[m_0,[m_1,k]]\in N$; since $N \cap M = \triv$, in fact we must have $[m_0,m_1]=1$, contradicting our choice of $m_0,m_1$.  From this contradiction, we conclude that in fact $M \le N$.  Since $N$ is a minimal non-trivial closed normal subgroup of $G$, it follows that $N = \ngrp{M}_G$.  This completes the proof of (iv).

\medskip

For (v), we already know that $\mc{M}(G)$ is finite and that $|\mc{M}(G)| \le |\mc{M}(O)|$ for all regionally expansive open subgroups $O$.  All that remains to show is that there is an expansive compactly generated open subgroup $O$ such that $|\mc{M}(O)| = |\mc{M}(G)|$.

Let $O$ be a compactly generated open subgroup of $G$ with no non-trivial abelian normal subgroups.  Say that $O \le L \le G$ and let $N$ be a non-trivial closed normal subgroup of $L$. The subgroup $N$ is non-discrete since $\QZ(L) \le \QZ(G) = \triv$, so $N \cap O$ is non-trivial and hence non-abelian.  Any sufficiently large compactly generated open subgroup of $G$ thus has no non-trivial abelian normal subgroups. We may thus assume that $O$ is regionally expansive. By Lemma~\ref{lem:minimal_geofaithful}, it follows that $\mc{M}(L)$ is finite whenever $O\le L \le G$.

Let $L$ be a compactly generated open subgroup of $G$ such that $|\mc{M}(L)|$ is minimal among the compactly generated open subgroups of $G$ containing $O$ and suppose toward a contradiction that $|\mc{M}(L)| > |\mc{M}(G)|$.  There  thus exist $M_1,M_2 \in \mc{M}(L)$ such that $N:= \ngrp{M_1}_G = \ngrp{M_2}_G$.  Let $(L_i)_{i \in I}$ be a directed family of compactly generated open subgroups of $G$ with union $G$ such that $L \le L_i$ for all $i$.  The minimality of $\mc{M}(L)$ ensures that $\ngrp{M_1}_{L_i}$ and $\ngrp{M_2}_{L_i}$ are distinct, hence they have trivial intersection for all $i \in I$.  In particular, $\ngrp{M_1}_{L_i}$ and $\ngrp{M_2}_{L_i}$ commute, so $uM_1u\inv$ centralizes $M_2$ for all $u \in L_i$.  Since $G$ is the union of the $L_i$, it follows that every $G$-conjugate of $M_1$ centralizes $M_2$, and $N$ thus centralizes $M_2$.  However, $M_2 \le N$, so $M_2$ is abelian, which contradicts our earlier conclusion that $L$ has no non-trivial abelian normal subgroups.  We conclude that $|\mc{M}(L)|=|\mc{M}(G)|$, proving (v).
\end{proof}

In particular, Proposition~\ref{prop:minimal_A-semisimple} applies to all regionally expansive \tdlc groups that are [A]-semisimple.

\begin{cor}\label{cor:minimal_A-semisimple}
Let $G$ be a regionally expansive, [A]-semisimple \tdlc group.
\begin{enumerate}[label=(\roman*)]
\item $\mc{M}(G)$ is a finite set, and every non-trivial closed normal subgroup of $G$ contains some $M \in \mc{M}(G)$.
\item Let $O \le G$ be a regionally expansive open subgroup of $G$.  Then there is a surjective map from $\mc{M}(O)$ to $\mc{M}(G)$ given by sending $M \in \mc{M}(O)$ to $N = \ngrp{M}_G$.
\item The number $|\mc{M}(G)|$ is the least value of $|\mc{M}(O)|$ as $O$ ranges over expansive compactly generated open subgroups of $G$.\qedhere
\end{enumerate}
\end{cor}

\begin{cor}\label{cor:monolith_is_regional}
Let $G$ be a regionally expansive \tdlc group that is [A]-semisimple.  Then $G$ is monolithic if and only if some compactly generated open subgroup of $G$ is monolithic.
\end{cor}

We next observe that regionally expansive topologically characteristically simple groups are [A]-semisimple.  Our proof requires a small adaptation of a lemma from \cite{W_E_14}.  The proof is the same as in \cite{W_E_14}, so we leave the details to the reader.

\begin{lem}[See {\cite[Lemma~9.11]{W_E_14}}]\label{lem:CRW_comm1}
	Suppose that $G$ is a \tdlc group with a compact open subgroup $U$. Suppose further that $\mc{K}$ is a finite set of infinite compact locally normal subgroups of $G$ such that $\mc{K}$ is stable under conjugation by $U$.  Defining 
\[
	H :=\cgrp{\mc{K}}, \; V:=\bigcap_{K \in \mc{K}} \N_{H\cap U}(K), \; \text{and} \; L:=\cgrp{V\cap K\mid K \in \mc{K} },
\]
the following holds:
	\begin{enumerate}[label=(\roman*)]
		\item $U\sleq \N_G(H)$;
		\item $L$ is an infinite closed normal subgroup of $U$; and
		\item $\Comm_{UH}(L)=UH$.\qedhere
	\end{enumerate}
\end{lem}

\begin{thm}\label{thm:A-semisimple1}
Let $G$ be a \tdlc group that is topologically characteristically simple, regionally expansive, and non-discrete.  Then $G$ is [A]-semisimple.
\end{thm}

\begin{proof}
By Lemma~\ref{lem:nonqd}, the quasi-center of $G$ is not dense.  Since $G$ is topologically characteristically simple, we have $\QZ(G)=\triv$.  

Let $\mc{A}$ be the collection of non-trivial abelian compact locally normal subgroups. Suppose for a contradiction that $\mc{A}$ is non-empty, so $\grp{\mc{A}}$ is non-trivial.  Since the set $\mc{A}$ is invariant under automorphisms of $G$, it follows that $\grp{\mc{A}}$ is dense in $G$.

For the convenience of this proof, we will say an admissible triple $(\mc{B},F,U)$ consists of the following: $\mc{B}$ is a non-empty finite subset of $\mc{A}$, $F$ is a finite subset of $G$ containing $1$, and $U$ is a compact open subgroup of $G$. Fix an admissible triple $(\mc{B},F,U)$ and define the following:
\[
\mc{B}_{F,U}: = \{bAb^{-1}\mid b\in UFU\text{ and } A \in \mc{B}\};
\]
\[
V := V_{\mc{B},F,U}:=\bigcap_{A \in \mc{B}_{F,U}}\N_{\cgrp{\mc{B}_{F,U}}\cap U}(A);
\]
\[
L := L_{\mc{B},F,U}:=\cgrp{V_{\mc{B},F,U}\cap A \mid A \in \mc{B}_{F,U}}; \text{ and }
\]
\[
Q := Q_{\mc{B},F,U} :=\overline{\QC_G(L_{\mc{B},F,U})}.
\]
The set $\mc{B}_{F,U}$ is invariant under conjugation by elements of $U$. It is also a finite set, as $UFU = C_FU$ for some finite subset $C_F$ and each $A \in \mc{B}$ has only finitely many $U$-conjugates.  Lemma~\ref{lem:CRW_comm1} then ensures that $L$ is commensurated by $\cgrp{\mc{B}_{F,U}}U$.  It follows by Lemma~\ref{lem:QC} that $Q$ is normalized by $\cgrp{\mc{B}_{F,U}}U$.

For each $A \in \mc{B}_{F,U}$, the group $V \cap A$ is an abelian normal subgroup of $V$. As $\mc{B}_{F,U}$ is finite, Fitting's theorem  ensures that $L$ is nilpotent. In particular, $\Z(L)$ is non-trivial. In view of Lemma~\ref{lem:CRW_comm1}, $L$ is normal in $U$, so $\Z(L)$ is normal in $U$. Since $G$ has trivial quasi-center, the same is true of $U$; the action of $U$ on $\Z(L)$ by conjugation then has an infinite orbit, and in particular $\Z(L)$ is infinite. The group $\Z(L)$ is thus an infinite compact subgroup of $Q$, so $Q$ is not discrete.

Let us now consider how $Q_{\mc{B},F,U}$ depends on the choices of $\mc{B}$, $F$ and $U$. The set
\[
\mc{C}:= \{\mc{B}_{F,U} \mid (\mc{B},F,U) \text{ admissible}\}
\]
is a directed family.  Indeed, given finite subsets $F_1,\dots,F_n$ of $G$ and compact open subgroups $U_1,\dots,U_n$ of $G$, the set $\bigcup^n_{i=1}U_iF_iU_i$ is compact, and hence given any compact open subgroup $U'$ of $G$, there is some $F \subseteq G$ finite such that $\bigcup^n_{i=1}U_iF_iU_i \subseteq U'FU'$.  For finite subsets $\mc{B}^{(1)},\dots,\mc{B}^{(k)}$ of $\mc{A}$, we see that
\[
\bigcup^n_{i=1}\bigcup^k_{j=1}(\mc{B}^{(j)})_{F_i,U_i} \subseteq \mc{B}_{F,U'},
\]
where $\mc{B} = \bigcup^k_{j=1}\mc{B}^{(j)}$.

Now suppose that $\mc{B}_{F,U} \subseteq \mc{B}'_{F',U'}$ for admissible triples $(\mc{B},F,U)$ and $(\mc{B}',F',U')$.  The group
\[
\cgrp{V_{\mc{B}',F',U'} \cap A \mid A \in \mc{B}_{F,U} }
\]
is a subgroup of $L_{\mc{B}',F',U'}$ that contains a finite index open subgroup of every $A \in \mc{B}_{F,U}$. The group $L_{\mc{B},F,U}$ is an internal (not necessarily direct) product of the normal subgroups $V_{\mc{B},F,U}\cap A$ where $A\in \mc{B}_{F,U}$. Since $I:=L_{\mc{B},F,U}\cap L_{\mc{B}',F',U'}$ contains a finite index open subgroup of $V_{\mc{B},F,U}\cap A$ for each $A \in \mc{B}_{F,U}$, it follows that  $I$ is of finite index in $L_{\mc{B},F,U}$. Thus, $L_{\mc{B}',F',U'}$  contains a finite index subgroup of $L_{\mc{B},F,U}$, and $Q_{\mc{B}',F',U'} \le Q_{\mc{B},F,U}$.  Since $\mc{C}$ is a directed family, we conclude that
\[
\mc{F} := \{Q_{\mc{B},F,U} \mid (\mc{B},F,U) \text{ admissible}\}
\]
is a filtering family.

Set $M: = \bigcap_{Q \in \mc{F}}Q$.  By construction, $M$ is a closed characteristic subgroup of $G$, so it is either trivial or equal to $G$; we will derive a contradiction in both cases.  Fix an expansive compactly generated open subgroup $O$ of $G$.

Suppose $M = \triv$.  Let $(\mc{B},F,U)$ be an admissible triple.  Since $O$ is compactly generated and $\mc{A}$ generates a dense subgroup of $G$, there is a finite subset $\mc{B}'$ of $\mc{A}$ containing $\mc{B}$ such that $O \le \cgrp{\mc{B}'_{F,U}}U$.  As noted above, $Q_{\mc{B}',F,U}$ is normalized by  $\cgrp{\mc{B}'_{F,U}}U$. A fortiori, $Q_{\mc{B}',F,U}$ is normalized by $O$, and we have $Q_{\mc{B}',F,U} \le Q_{\mc{B},F,U}$.  Given the freedom of choice of $(\mf{B},F,U)$, we conclude that every element of $\mc{F}$ contains an element of $\mc{F}'$, where $\mc{F}'$ consists of the elements of $\mc{F}$ normalized by $O$. The family $\mc{F}'$ has a trivial intersection, since we assume $M=\{1\}$, and we obtain a trivial intersection of closed normal subgroups of $O$ as follows: $\triv = \bigcap_{Q \in \mc{F}'}(Q \cap O)$.  Applying Proposition~\ref{prop:filtering}, $Q \cap O$ is discrete for some $Q \in \mc{F}'$, and since $O$ is open, it follows that $Q$ is discrete.  However, we have already shown that $Q_{\mc{B},F,U}$ is non-discrete for every admissible triple $(\mc{B},F,U)$, which is a contradiction.

Let us now suppose that $M = G$.  Fix an admissible triple $(\mc{B},F,U)$ and let $L = L_{\mc{B},F,U}$; recall that $L$ is non-discrete. Since $O$ is expansive, we can find a compact open subgroup $W$ of $O\cap U$ which has a trivial normal core in $O$. The group $L$ has dense quasi-centralizer in $G$, ensuring that $O = \grp{C}W$ for some finite subset $C$ of $\QC_O(L)$.  We may then find a finite index $\widehat{W}\normal W$ such that $C$ centralizes $L \cap \widehat{W}$, and since $L \cap \widehat{W}$ is normalized by both $C$ and $W$, it is normal in $O$.  We have thus obtained a non-trivial normal subgroup of $O$ contained in $W$, contradicting our assumption that $W$ has trivial normal core in $O$.

From the contradictions in the previous two paragraphs, we conclude that no admissible triples exist, so in fact $\mc{A}$ is empty; that is, $G$ has no non-trivial locally normal abelian subgroups.  Since also $\QZ(G)=\triv$, we conclude that $G$ is [A]-semisimple.
\end{proof}

\subsection{Permanence properties}

In general, regional expansiveness does not pass from a dense locally compact subgroup to the ambient group; consider discrete dense locally compact subgroups.  However, regional expansiveness can be transmitted from a dense locally compact subgroup up to the ambient group under suitable assumptions. 

A \textit{normal compression} from a topological group $H$ to a topological group $G$ is a dense embedding $\psi:H\rightarrow G$ such that the image is \textit{normal}.

\begin{lem}\label{lemprop:normal-compressions}
Let $G$ and $H$ be \tdlc groups and $\psi \colon H \rightarrow G$ be a normal compression. If $H$ is regionally expansive with trivial quasi-center, then $G/\Z(G)$ is regionally expansive with trivial quasi-center.
\end{lem}

\begin{proof}
Let $\theta:G\rightarrow G/\Z(G)$ be the usual projection. Since $H$ has trivial quasi-center, $\theta\circ\psi:H\rightarrow G/\Z(G)$ is injective. It follows that $\theta\circ\psi(H)\cap \QZ(G/\Z(G))$ is trivial. Set $J:=\theta^{-1}(\QZ(G/\Z(G)))$. The subgroup $J$ is normal in $G$ and intersects $\psi(H)$ trivially. As $\psi(H)$ is dense and normal, we infer that $J\leq \Z(G)$. Hence, $G/\Z(G)$ has trivial quasi-center.  

Fix a compactly generated expansive open subgroup $L$ of $H$.  Let $O$ be a compactly generated open subgroup of $G/\Z(G)$ that contains $\theta\circ \psi(L)$, where $\theta \colon G \to G/\Z(G)$ is again the projection, and note that $\Z(O)$ is quasi-central in $G/\Z(G)$, hence trivial.  We argue that $O$ is expansive.

Let $(N_i)_{i \in I}$ be a filtering family of non-trivial closed normal subgroups of $O$ and fix $i \in I$.  The subgroup $N_i$ is not central in $O$, since $O$ has trivial center, and in particular, $N_i$ does not centralize $P:= \theta\circ\psi(H) \cap O$, as $P$ is dense in $O$. Seeing as $N_i$ and $P$ are both normal in $O$, we deduce that they must have a non-trivial intersection.  The preimage $(\theta\circ\psi)\inv(N_i)$ is then non-trivial, and it is non-discrete, since $H$ has a trivial quasi-center. The group $M_i := (\theta\circ \psi)\inv(N_i) \cap L$ is then also non-trivial.  Thus, $(M_i)_{i \in I}$ is a filtering family of non-trivial closed normal subgroups of $L$.  Lemma~\ref{lem:minimal_geofaithful}(i) ensures that $\bigcap_{i \in I}M_i$ is non-trivial, and hence $\bigcap_{i \in I}N_i$ is non-trivial.  By Lemma~\ref{lem:expansive_tdlc}, we have proved that $O$ is expansive; hence, $G/\Z(G)$ is regionally expansive.
\end{proof}

With some adjustments, we can consider a useful, more general situation where $\psi \colon  H \rightarrow G$ is a continuous, injective homomorphism such that $\psi(H)$ is normal in $G$. 

\begin{lem}\label{lem:RF_overgroup}
	Let $G$ be a \tdlc group and $H$ be a closed locally normal subgroup of $G$. If $H$ is regionally expansive and $\QC_G(H) \le H$, then $G$ is regionally expansive.
\end{lem}
%
%

\begin{proof}
	Note that $\QZ(G) \le \QC_G(K)$ for any subgroup $K$ of $G$.  In particular, since $\QC_G(H) \le H$, we have $\QZ(G) \le H$. Every discrete locally normal subgroup of $G$ is thus contained in $H$.
	
Fix a compactly generated open $O \le H$ such that $O$ is expansive and say that $W\leq G$ is a compact open subgroup of $G$ for which $O\cap W$ has trivial normal core in $O$.  Since $\N_G(H)$ is open, we may take $W \le \N_G(H)$.  The group $L:=\grp{O,W}$ is then a compactly generated open subgroup of $G$.  Letting $K$ be the normal core of $W$ in $L$, $O \cap K = \triv$, and thus, $H \cap K$ is discrete; since $K$ is compact, $H \cap K$ is indeed finite.  Now $H\cap W$ and $K$ normalize each other, so $[H\cap W,K] \le H\cap K$; in particular, $[H\cap W,K]$ is finite.  There are thus finite index open subgroups $H'$ of $H\cap W$ and $K'$ of $K$ such that $[H',K'] = \triv$, so  $K' \le \QC_G(H) \le H$.  Since $H \cap K$ is finite and $K'$ has finite index in $K$, it follows that $K$ is finite.  There is then an open subgroup $W'$ of $W$ such that $W' \cap K = \triv$, and we see that $W'$ has trivial normal core in $L$.  Thus $L$ is expansive, showing that $G$ is regionally expansive.
\end{proof}

\begin{prop}\label{prop:normal-compressions}
Let $G$ and $H$ be \tdlc groups and $\psi \colon H \rightarrow G$ be a continuous, injective homomorphism. Assume that  $H$ is regionally expansive with trivial quasi-center and that $\psi(H)$ is normal in $G$. Then the subgroup $Q:=\QC_G(\overline{\psi(H)})$ equals $\CC_G(\psi(H))$ and is closed in $G$, and $G/Q$ is regionally expansive with trivial quasi-center. In fact, $\ol{\psi(H)Q}/Q$ has trivial quasi-centralizer in $G/Q$.
\end{prop}

\begin{proof}
Set $K:=\overline{\psi(H)}$. Lemma~\ref{lemprop:normal-compressions} ensures that $K/\Z(K)$ has a trivial quasi-center. The quasi-centralizer of $K/\Z(K)$ in $G/\Z(K)$ then equals the centralizer of $K/\Z(K)$ in $G/\Z(K)$ by Lemma~\ref{lem:qcent_norm}. Therefore, $[Q ,K] \le \Z(K)$, and in particular, 
\[
[Q,\psi(H)] \le \psi(H) \cap \Z(K) = \triv.
\]
We conclude that $Q $ centralizes $\psi(H)$. In fact, $Q  = \CC_G(K) = \CC_G(\psi(H))$, establishing the first claim.  Since $Q$ is a centralizer, it is a closed subgroup of $G$.

Seeing that $Q \cap \psi(H) = \triv$, we have a continuous, injective homomorphism $\psi^*: H \rightarrow G^*$ where $G^*: = G/Q$.  We argue that $K^* := \overline{\psi^*(H)}$ has trivial quasi-centralizer in $G^*$.  Let $x \in G$ be such that $xQ$ centralizes an open subgroup $L/Q$ of $K^*/Q$.  The element $x$ then centralizes $L \cap \psi(H)$, because $Q \cap \psi(H) = \triv$.  Noting that $L \cap \psi(H)$ is dense in $L \cap K$, the element $x$ centralizes $L \cap K$. We infer that $x$ centralizes an open subgroup of $K$, and hence $x \in Q$.  Therefore, $xQ$ is the trivial element of $G^*$.

The map $\psi^*$ is a normal compression from $H$ to $K^*$. The subgroup $K^*$ has a trivial quasi-center in $G^*$, so a fortiori, $\Z(K^*)=\triv$. We apply Lemma~\ref{lemprop:normal-compressions} to deduce that $K^*$ is regionally expansive. On the other hand, $K^*$ is normal, and $\QC_{G^*}(K^*)$ is trivially contained in $K^*$. Lemma~\ref{lem:RF_overgroup} thus ensures that that $G^*=G/Q$ is regionally expansive.
\end{proof}

Regional expansiveness is also inherited by cocompactly embedded subgroups.
 
\begin{prop}\label{prop:cocompact_RF}
Let $G$ and $ H$ be \tdlc groups and $\psi \colon H \rightarrow G$ be a continuous, injective homomorphism such that $\psi(H)$ is cocompact in $G$.  If $G$ regionally expansive, then $H$ is regionally expansive.
\end{prop}

\begin{proof}
Let $O$ be an expansive compactly generated open subgroup of $G$ and say that $U\leq O$ is a compact open subgroup with trivial normal core in $O$.  By Lemma~\ref{lem:cocompact_restriction}, there exist $h_1,\dots,h_n \in \psi\inv(O)$ such that $\la \psi(h_1),\dots,\psi(h_n) \ra$ is cocompact in $O$.  Setting $H_2: = \langle h_1,\dots,h_k, V \rangle$ where $V$ is a compact open subgroup of $\psi\inv(O)$, the group $H_2$ is a compactly generated open subgroup of $H$ such that $\psi(H_2)$ is cocompact in $O$.  Let us replace $G$ with $O$ and $H$ with $H_2$ to reduce to the case when both $G$ and $H$ are compactly generated and $U$ has trivial normal core in $G$.  In this case, it suffices to show that $H$ has an open subgroup that contains no non-trivial normal subgroup of $H$.

Since $\psi(H)$ is cocompact in $G$, we can write $G$ as $\psi(H)X$ where $X$ is compact.  The subgroup $U$ has trivial normal core in $G$, so
\[
\triv = \bigcap_{g \in G}gUg\inv = \bigcap_{h \in H}\bigcap_{x \in X}\psi(h)xUx\inv \psi(h)\inv = \bigcap_{h \in H}\psi(h)\left(\bigcap_{x \in X}xUx\inv\right)\psi(h)\inv.
\]
We conclude that $W := \bigcap_{x \in X}xUx\inv$ does not contain any non-trivial normal subgroup of $\psi(H)$, hence $\psi\inv(W)$ does not contain any non-trivial normal subgroup of $H$.  At the same time, $W$ is open in $G$, since $X$ is contained in a finite union of left cosets of $U$, so $\psi\inv(W)$ is open in $H$. The group $H$ therefore has an open subgroup with trivial normal core.  
\end{proof}

\subsection{An application}
We pause to generalize a prior result on the structure of topologically characteristically simple groups (\cite[Corollary D]{CM11}) from the compactly generated case to the regionally expansive case. 

\begin{thm}\label{thm:tcs_quasiproduct}
	Suppose that $G$ is a non-discrete topologically characteristically simple \tdlc group.  Then $G$ is regionally expansive if and only if $G = \cgrp{\mc{M}(G)}$, where $\mc{M}(G) = \{S_1,\dots,S_n\}$ is finite and each $S_i$ is isomorphic to some \tdlc group $S$ that is non-discrete, regionally expansive, and topologically simple.
\end{thm}
\begin{proof}
Suppose that $G$ is regionally expansive. The group $G$ is [A]-semisimple by Theorem~\ref{thm:A-semisimple1}, and hence $\mc{M}(G)$ is finite by Corollary~\ref{cor:minimal_A-semisimple}. Letting $\mc{M}(G)=\{S_1,\dots,S_n\}$, every non-trivial closed normal subgroup of $G$ contains some $S_i$ by Lemma~\ref{lem:minimal_geofaithful}.  Since $G$ is topologically characteristically simple, we see that $G = \cgrp{\mc{M}(G)}$ and that $\Aut(G)$ permutes $\mc{M}(G)$ transitively, so the elements of $\mc{M}(G)$ are isomorphic.  By \cite[Proposition 5.13]{RW_P_15}, each $S_i$ is topologically simple. Each $S_i$ is also non-discrete since the quasi-center of $G$ is trivial.

Let us suppose for contradiction that some, equivalently all, $S_i$ fail to be regionally expansive.  Fix a compactly generated $O\leq G$ and $W$ a compact open subgroup of $O$ such that $W$ has trivial normal core in $O$. The intersection $\grp{\mc{M}(G)}\cap O$ is dense in $O$, so we may find a finite set $F\subseteq \bigcup\mc{M}(G)$ such that $\grp{F}W=O$, by Lemma~\ref{lem:Cayley-Abels}. For each $f\in F$ with $f\in S_i$, let $X_f$ be a compact neighborhood of $f$ in $S_i$ and let $A_f: = \cgrp{wX_fw\inv \mid w \in W}$; note that $A_f$ is a compactly generated open subgroup of $S_i$ such that $f \in A_f$ and $W \le \N_G(A_f)$. The group $O': = \grp{A_f\mid f\in F}W$ is a supergroup of $O$, and so $W$ has trivial normal core in $O'$. On the other hand, each $A_f$ fails to be regionally expansive, so there is a closed non-trivial $K_f\leq W$ such that $K_f\normal A_f$. The subgroup $\prod_{f\in F}K_f$ is normalized by $\grp{A_f\mid f\in F}$ and contained in $W$. Lemma~\ref{lem:no_cocompact} then supplies a non-trivial normal subgroup of $O'$ contained in $W$. This contradicts the fact that $W$ has trivial normal core in $O'$, and we deduce that each $S_i$ is regionally expansive.


Conversely, suppose that $G = \cgrp{S_1,\dots,S_n}$ where each of the subgroups $S_i$ is a non-discrete, regionally expansive, topologically simple closed normal subgroup of $G$.  Each $S_i$ has trivial quasi-center by Theorem~\ref{thm:A-semisimple1} and thus is not central in $G$. In particular, $\Z(G) \neq G$, and since $G$ is characteristically simple, it must be the case that $\Z(G)=\triv$.   For $1 \le i \le n$, let $O_i$ be an expansive compactly generated open subgroup of $S_i$ and let $U$ be a compact open subgroup of $G$.  Note that $U$ has compact orbits on $S_i$ for each $i$, so by enlarging the subgroups $O_i$ as necessary, we may assume that $U \le \N_G(O_i)$ for all $i$.  By choosing $U$ sufficiently small, we may also ensure that $O_i \cap U$ contains no non-trivial closed normal subgroup of $O_i$.  

Now let $O := \langle O_1,\dots,O_n,U\rangle$; the group $O$ is a compactly generated open subgroup of $G$.  Let $K$ be a closed normal subgroup of $O$ such that $K \le U$.  The choice of $U$ ensures that $O_i \cap K = \triv$ for all $i$. Both $O_i$ and $K$ are normal in $O$, so they must commute.  Thus, $K \le \CC_G(O_i) \le \QC_G(S_i)$ for all $i$.  Since $\QC_G(S_i) \cap S_i = \QZ(S_i) = \triv$, we see that $\QC_G(S_i)$ actually centralizes $S_i$.  The subgroup $K$ thereby centralizes $S_i$ for all $i$, so $K \le \Z(G) = \triv$.  We conclude that $U$ contains no non-trivial compact normal subgroup of $O$, and hence, $G$ is regionally expansive.
\end{proof}

\begin{rmk}In \cite[Corollary D]{CM11} it is stated that a compactly generated topologically characteristically simple locally compact group must be compact, discrete, $\Rb^n$ or generated topologically by a finite set of isomorphic topologically simple closed normal subgroups.  Let us explain how this follows easily from Theorem~\ref{thm:tcs_quasiproduct} in the totally disconnected case.  Suppose that $G$ is a compactly generated and topologically characteristically simple \tdlc group.  If $G$ discrete, there is nothing to prove.  If $G$ is non-discrete and not regionally expansive, then $G$ has a non-trivial compact normal subgroup and is thus generated by compact normal subgroups, since it is characteristically simple.  Since $G$ is compactly generated, there is then some compact open subgroup $U$ of $G$ and a finite set $K_1,\dots,K_n$ of compact normal subgroups such that $G = \langle K_1,\dots,K_n\rangle U$; from here it is easy to see that $G$ must be compact.  If $G$ is non-discrete and regionally expansive, then Theorem~\ref{thm:tcs_quasiproduct} applies.\end{rmk}

\section{Robustly monolithic groups}

\subsection{Definition and basic properties}
 
\begin{defn}
	A \tdlc group $G$ is called \emph{robustly monolithic} if $G$ is monolithic and the monolith is non-discrete, regionally expansive, and topologically simple. We denote by $\Rs$ the class of robustly monolithic groups.
\end{defn}

\begin{prop}\label{prop:Rs_is_RF_[a]} Every element of $\Rs$ is regionally expansive and $[A]$-semisimple.
\end{prop}
\begin{proof}
	Take $G\in \Rs$.  The monolith $M$ of $G$ is regionally expansive, and by Theorem~\ref{thm:A-semisimple1}, it has a trivial quasi-center. Proposition~\ref{prop:normal-compressions} now ensures that $\QC_G(M)$ is closed. The monolith $M$ is not quasi-discrete, so $\QC_G(M)$ is in fact trivial.  A second application of Proposition~\ref{prop:normal-compressions} implies that $G$ is regionally expansive.
	
	 To see that $G$ is $[A]$-semisimple, we first note that $\QZ(G)\leq \QC_G(M)=\{1\}$, so $G$ has a trivial quasi-center.  Let $A$ be a closed locally normal abelian subgroup of $G$; it remains to show that $A$ is trivial.  The intersection $M \cap A$ is a closed locally normal abelian subgroup of $M$.  By Theorem~\ref{thm:A-semisimple1}, $M$ is $[A]$-semisimple, so $M \cap A = \triv$.  Let $U$ be an open subgroup of $\N_G(A)$.  The groups $M \cap U$ and $A \cap U$ normalize each other and have trivial intersection, so $[M \cap U,A \cap U] = \triv$; in particular, $A \cap U$ centralizes an open subgroup of $M$.  Since $\QC_G(M)=\triv$, we have $A \cap U = \triv$, so $A$ is discrete.  Since every discrete locally normal subgroup of a \tdlc group lies in the quasi-center, we have $A \le \QZ(G) = \triv$.  This completes the proof that $G$ is $[A]$-semisimple.
\end{proof}

All results about regionally expansive groups thus apply to $\Rs$. It is also immediate that $\Ss \subseteq \Rs$; Section~\ref{sec:examples1} gives examples showing this inclusion is strict.
 
We make a trivial observation about monolithic groups that will be used frequently. The monolith of a group in $\Rs$ is infinite and non-abelian, since the monolith is non-discrete and regionally expansive, so the following lemma applies to groups in $\Rs$.

\begin{lem}\label{lem:monolith_centralizer}
Let $G$ be a monolithic topological group such that the monolith $M$ is infinite and non-abelian.  Then given a commuting pair of closed normal subgroups $A$ and $B$ of $G$, at least one of $A$ and $B$ is trivial.  In particular, $\CC_G(M)=\triv$.
\end{lem}

\begin{proof}
Let $A$ and $B$ be closed normal subgroups of $G$ such that $[A,B] = \triv$.  The intersection $A \cap B$ is central in $A$, so $A \cap B$ is abelian. In particular, $M \nleq A \cap B$.  Since $M$ is the monolith of $G$, it follows that $A \cap B = \triv$.  Therefore, $A$ and $B$ do not both contain $M$. If $M \nleq A$, then $A = \triv$, and if $M \nleq B$ then $B = \triv$.  We see that $\CC_G(M)=\triv$ by considering the case $A = M$, $B = \CC_G(M)$.
\end{proof}

Let us also note that the monoliths of groups in $\Rs$ again lie in $\Rs$. 

\begin{lem}\label{lem:monolith_overgroup}
Let $G \in \Rs$ and $H$ be a non-trivial closed subgroup of $G$ such that $\Mon(G) \le \N_G(H)$.  Then $\Mon(G) \le H$ and $H \in \Rs$.  In particular, $\Mon(G)$ is a topologically simple group in $\Rs$.
\end{lem}

\begin{proof}
As in the proof of Proposition~\ref{prop:Rs_is_RF_[a]}, the quasi-centralizer of $M:=\Mon(G)$ in $G$ is trivial.  Since $\Mon(G)$ and $H$ normalize each other and do not commute, they must have non-trivial intersection, and as $\Mon(G)$ is topologically simple, $\Mon(G) \le H$.  By the same argument, for any non-trivial closed normal subgroup $N$ of $H$, it is the case that $\Mon(G) \le N$.  Therefore, $H$ is monolithic with $\Mon(H) \ge \Mon(G)$.  Since $\Mon(G)$ is itself a non-trivial closed normal subgroup of $H$, we have $\Mon(H) = \Mon(G)$.  Thus, $H \in \Rs$.
\end{proof}

Lemma~\ref{lem:monolith_overgroup} implies that every closed $H\leq G$ containing $\Mon(G)$ is an element of $\Rs$. From the proof of Lemma~\ref{lem:monolith_overgroup}, we see additionally that $\Mon(H)=\Mon(G)$.
\subsection{Passing to regional subgroups}

We now argue that the property of being robustly monolithic is in fact a regional property. This turns out to be an important feature. 

\begin{lem}\label{lem:monolith_union}
Let $G$ be a \tdlc group and  $\{O_i\}_{i\in I}$ be a directed system of compactly generated open subgroups of $G$ with $\bigcup_{i\in I}O_i=G$.  Suppose that $O_i$ is monolithic for all $i \in I$ and that $\QZ(G)=\triv$.  Then $G$ is monolithic, and the monolith $M$ of $G$ is given by $M = \ol{\bigcup_{i\in I}M_i}$ where $M_i$ is the monolith of $O_i$, and $\bigcup_{i\in I}M_i$ is normal in $G$.
\end{lem}

\begin{proof}
Set $N:= \bigcup_{i\in I}M_i$. Since $\QZ(G)=\triv$, we also have $\QZ(O_i)=\triv$ for each $i \in I$, so each $M_i$ is not discrete.  Let $i,j \in I$ be such that $O_i \le O_j$.  The group $M_j$ is non-discrete, so $M_j \cap O_i$ is a non-trivial closed normal subgroup of $O_i$ and hence contains $M_i$.  By the same argument, any non-trivial closed normal subgroup $R$ of $G$ must contain $M_i$ for every $i \in I$. Therefore, $G$ is monolithic, and the monolith of $G$ contains $\ol{N}=M$.

For $g \in G$, the family $\{gO_ig\inv\}_{i\in I}$ is a directed system of compactly generated open subgroups of $G$ with union $G$. Taking $i\in I$ and $X_i$ a compact generating set for $O_i$, there is some $j \in I$ such that $gO_jg\inv$ contains $X_i$ and hence contains $O_i$.  We infer that $gM_jg\inv \ge M_i$, so $gNg\inv \ge M_i$.  As $i \in I$ is arbitrary, $gNg\inv \ge N$, and hence $gNg\inv = N$ by symmetry.  Therefore, $N$ is normal in $G$.  The group $M$ is then a non-trivial closed normal subgroup of $G$, and it  is the unique minimal such, by the previous paragraph. We conclude that $G$ is monolithic with monolith $M$.
\end{proof}

\begin{thm}\label{thm:approx_R} 
Let $G$ be a \tdlc group and let  $\{O_i\}_{i\in I}$ be a directed system of compactly generated open subgroups of $G$ with $\bigcup_{i\in I}O_i=G$. The following are equivalent. 
\begin{enumerate}[label=(\roman*)]
	\item \label{it:approx_R:1}
	$G \in \Rs$. 	
	
	\item \label{it:approx_R:2}
	There is $i\in I$ such that $O_j\in \Rs$ for all $j\geq i$. \qedhere
\end{enumerate}
\end{thm}

\begin{proof}
\ref{it:approx_R:1} $\Rightarrow$ \ref{it:approx_R:2}. 
	By passing to $I':=\{i\in I\mid i\geq j \}\subseteq I$ for some $j$, we may assume that there is a unique minimal $0\in I$ and that every $O_i$ is expansive, since $G$ is regionally expansive by Proposition~\ref{prop:Rs_is_RF_[a]}.  In view of Proposition~\ref{prop:Rs_is_RF_[a]}, we see that $G$, and also every open subgroup of $G$, is [A]-semisimple. We are now in a position to apply Corollary~\ref{cor:monolith_is_regional}, and thus we may choose $O_0$ to be monolithic. Corollary~\ref{cor:minimal_A-semisimple} then ensures that $O_i$ is monolithic for all $i \in I$.
	
Letting $M_i$ be the monolith of $O_i$, by Lemma~\ref{lem:monolith_union} the union $\bigcup_{i\in I}M_i$ is a dense normal subgroup of $M$, the monolith of $G$. The monolith $M$ is regionally expansive since $G \in \Rs$, so there is a compactly generated open subgroup $R\leq M$ and a compact open subgroup $W\leq R$ such that $W$ has trivial normal core in $R$. As the union of the groups $M_i$ is dense in $M$, there is $j\in I$ such that $(M_j\cap R)W=R$. The group $M_j\cap R$ is thus a compactly generated open subgroup of $M_j$ by Proposition~\ref{prop:cocompact_generation}. In view of Lemma~\ref{lem:no_cocompact}, the normal core of $M_j\cap W$ in $M_j\cap R$ must be trivial. Hence, $M_j$ is regionally expansive. The same argument applies to any $M_k$ for $k\geq j$, so passing to the system $\{O_k\}_{k\geq j}$, we may assume every $M_k$ is regionally expansive.

Via Theorem~\ref{thm:tcs_quasiproduct}, each $M_i$ is generated by a finite set $\mc{M}(M_i)$ of minimal non-trivial closed normal subgroups, and each $N \in \mc{M}(M_i)$ is non-discrete, regionally expansive, and topologically simple. Since $|\mc{M}(M_i)|$ is finite, $W_i:=\bigcap_{N\in \mc{M}(M_i)}\N_{O_i}(N)$ is of finite index in $O_i$. A fortiori, each $N\in \mc{M}(M_i)$ has  open normalizer in $M_j$ for $j\geq i$.  Applying Proposition~\ref{prop:minimal_A-semisimple}, each $N\in \mc{M}(M_i)$ is such that $\ngrp{N}_{M_j} \in \mc{M}(M_j)$ for $j\geq i$.  On the other hand, take $N' \in \mc{M}(M_j)$ for some $j\geq i$.  The intersection $N' \cap W_i$ is non-trivial, since $N'$ is non-discrete, so $[N' \cap W_i,M_i] \neq \triv$. There is thus some $N''\in \mc{M}(M_i)$ such that $[N' \cap W_i,N''] \neq \triv$.  The group $W_j$ has finite index in $O_j$, so $N''$ is contained in $W_j$. That is to say, $N''$ normalizes $N'$. It follows that $N''\leq N'$, and we infer that $N' = \ngrp{N''}_{M_j}$.  For $j\geq i$, we thus obtain a surjective map from $\mc{M}(M_i)$ to $\mc{M}(M_j)$ given by $N \mapsto \ngrp{N}_{M_j}$.

It remains to show that $M_i$ is topologically simple for all sufficiently large $i \in I$; by Theorem~\ref{thm:tcs_quasiproduct}, $M_i$ is topologically simple if and only if $|\mc{M}(M_i)|=\triv$.  Let $i \in I$ be such that $|\mc{M}(M_i)|$ is minimized and suppose toward a contradiction that $N_1$ and $N_2$ are distinct elements of $\mc{M}(M_i)$. For $j\geq i$, the minimality of $|\mc{M}(M_i)|$ ensures that the groups $\ngrp{N_1}_{M_j}$ and $\ngrp{N_2}_{M_j}$ are distinct, hence they have trivial intersection and so commute.  The group $N_1$ thereby centralizes every $ \bigcup_{i \in I}M_i$-conjugate of $N_2$.  Since $ \bigcup_{i \in I}M_i$ is dense in $M$, in fact $N_1$ centralizes every $M$-conjugate of $N_2$, so $N_1$ centralizes a non-trivial normal subgroup of $M$. Since $M$ is topologically simple, we in fact have $N_1 \le \CC_G(M)$.  By Lemma~\ref{lem:monolith_centralizer}, $\CC_G(M)=\triv$, so $N_1=\triv$, which contradicts the hypothesis that $N_1 \in \mc{M}(M_i)$.  We conclude that $|\mc{M}(M_i)|=1$, and it follows that $|\mc{M}(M_j)|=1$ for all $j\geq i$. Therefore, $M_j$ is topologically simple for all $j\geq i$, and $O_j \in \Rs$ for all $j \ge i$.

\medskip

\noindent \ref{it:approx_R:2} $\Rightarrow$ \ref{it:approx_R:1}. 
Since $G$ has a regionally expansive open subgroup, it is itself regionally expansive. For $g \in \QZ(G)$, there is $j \geq i$ such that $g \in O_j$. Since $O_j$ is open, we have $g \in \QZ(O_j)$. Robustly monolithic groups have trivial quasi-center by Proposition~\ref{prop:Rs_is_RF_[a]}, so we infer that $\QZ(G)= \triv$. Lemma~\ref{lem:minimal_geofaithful} now implies that the set $\mathcal M(G)$ of minimal closed normal subgroups of $G$ is non-empty. Moreover, for each $M \in \mathcal M(G)$, there is $N \in \mathcal M(O_i)$ with $M = \overline{\lla N \rra_G}$. Since $O_i$ is robustly monolithic,  the set $\mathcal M(O_i)$ has only one element, and it follows that $\mathcal M(G)$ also has a single element. The group $G$ is thus monolithic by Lemma~\ref{lem:minimal_geofaithful}. Setting $M_j := \Mon(O_j)$ for each $j$, Lemma~\ref{lem:monolith_union} ensures that $D :=\bigcup_{j \geq i} M_j$ is a dense subgroup of $M := \Mon(G)$ and is normal in $G$.

Take $N$ to be a closed normal subgroup of $M$. If $D \cap N \neq \triv$, then $M_j \cap N \neq \triv$ for some $j \geq i$, hence $M_ j \leq N$ since $M_j$ is topologically simple. Therefore, $M_{j'} \leq N$ for all $j' \geq j$. We conclude that $D \leq N$, and hence $M = N$, since $D$ is dense in $M$. On the other hand, if $D \cap N = \triv$, then $[D, N]= \triv$, and $N \leq \Z(M)$, since $D$ is dense. The group $M$ is topologically characteristically simple and non-abelian, since each $M_j$ is non-abelian. We thus have $\Z(M) = \triv$, so $N =\triv$. This shows that $M$ is topologically simple.  

It remains to show that $M$ is regionally expansive. Let $U \leq O_i$ be a compact open subgroup. Since $M_i$ is [A]-semisimple by Proposition~\ref{prop:Rs_is_RF_[a]}, $\QZ(M_i)=\triv$. Hence, $\QC_{O_i}(M_i) = \CC_{O_i}(M_i) $ via Lemma~\ref{lem:qcent_norm}. The latter centralizer is trivial, since $O_i$ is [A]-semisimple, hence $\QC_{O_i}(M_i)=\triv$. Lemma~\ref{lem:RF_overgroup} ensures that $UM_i$ is regionally expansive.  There must then exist a compactly generated open subgroup $P \leq UM_i$ containing $U$ such that the normal core of some open subgroup $W\normal U$ in $P$ is trivial. We see that $P = U (M_i \cap P)$, so  $M_i \cap P$ is a cocompact closed subgroup of $P$ and is thus compactly generated, by Proposition~\ref{prop:cocompact_generation}. As $M_i \leq M$, we infer that $M \cap P = (M\cap U)(M_i \cap P)$. In particular, $M_i \cap P$ is a cocompact subgroup of $M \cap P$, hence $M \cap P$ is compactly generated.  Let $K: = \bigcap_{g \in M \cap P} g (M\cap W )g\inv$. We will conclude the proof by showing that $K = \triv$; this will show that $M$ is regionally expansive: $M \cap P$ is a compactly generated open subgroup of $M$ with a compact open subgroup $M \cap W$ that has trivial normal core $K$.

Since $U$ normalizes $M \cap W$ and $P = U(M_i \cap P)$, we see that $M_i \cap P$ acts transitively on the conjugacy class of $M \cap W$ in $P$. Since $M_i \cap P \le M \cap P$, the group $M \cap P$ also acts transitively on the conjugacy class of $M \cap W$ in $P$.  Thus $K = \bigcap_{g \in P}g(M \cap W)g\inv$, and since $\bigcap_{g \in P}gWg\inv=\{1\}$, $K = \triv$ as required.
\end{proof}

\subsection{Dense embeddings with normal image}
Similar to Proposition~\ref{prop:normal-compressions}, we obtain a general circumstance in which groups in $\Rs$ extend to a larger group in $\Rs$.  In order to appeal to some results from \cite{RW_P_15}, we must make the additional assumption that we are starting from a second countable group in $\Rs$. By Lemma~\ref{lem:FirstCountable}, every member of $\Rs$ is first countable. Hence, an element of $\Rs$ is second countable if and only if it is $\sigma$-compact, via classical point-set topology.

\begin{prop}\label{prop:R_overgroup}
Let $G$ and $H$ be \tdlc groups and $\psi \colon H \rightarrow G$ be a continuous, injective homomorphism. Assume that $H \in \Rs$, $H$ is second countable, and $\psi(H)$ is normal in $G$. The subgroup $Q:=\QC_G(\overline{\psi(H)})$ equals $\CC_G(\psi(H))$ and is closed in $G$, and $G/Q$ is in $\Rs$.
\end{prop}

\begin{proof}
The hypotheses are a special case of Proposition~\ref{prop:normal-compressions}. Thus, $Q=\CC_G(\psi(H))$ and is closed in $G$ and $G/Q$ is regionally expansive with trivial quasi-center. We have a continuous homomorphism $\psi^*:H\rightarrow G/Q$ by $h\mapsto hQ$.

Set $K: = \overline{\psi^*(H)}$; by Proposition~\ref{prop:normal-compressions}, the quasi-centralizer of $K$ in $G/Q$ is trivial.   Consider a non-trivial closed normal subgroup $N$ of $K$.  Since $K$ has trivial center and $\psi^*(H)$ is dense in $K$, the subgroup $N$ cannot centralize $\psi^*(H)$, and thus $N \cap \psi^*(H) > 1$.  The preimage $(\psi^*)\inv(N)$ contains the monolith $S$ of $H$, hence $N \ge \overline{\psi^*(S)}$. The intersection $T$ of all non-trivial closed normal subgroups of $K$ contains $\overline{\psi^*(S)}$, and we infer that $K$ is monolithic with monolith $T$. As $S$ is infinite and topologically simple, $S$ cannot be injectively mapped into a profinite group, so $T$ is also non-compact.  

As $H$ is $\sigma$-compact, $K$ is $\sigma$-compact.  The group $K$ is indeed second countable, because it does not have arbitrarily small non-trivial compact normal subgroups.  Appealing to \cite[Corollary~3.7]{RW_P_15}, the image $\psi^*(S)$ is normal in $K$.  It follows that $T = \overline{\psi^*(S)}$ and that $T$ is a normal compression of $S$.  The monolith $T$ is non-abelian and topologically characteristically simple, so it has trivial center.  Applying \cite[Theorem~1.4]{RW_P_15}, we infer that $T$ is topologically simple, and Lemma~\ref{lemprop:normal-compressions} ensures that $T$ is regionally expansive. We thus deduce that $K \in \Rs$.

It remains to show that $G/Q \in \Rs$.  The subgroup $T$ is topologically characteristic in $K$ and hence normal in $G/Q$. Therefore, $\QC_{G/Q}(T)\normal G/Q$. As $T$ is the monolith of $K$ and $K$ is in $\Rs$, it follows that $\QC_{G/Q}(T)\cap K=\{1\}$. We deduce that $\QC_{G/Q}(T)$ centralizes $K$, hence $\QC_{G/Q}(T)=\{1\}$. For any closed non-trivial $M\normal G/Q$, it now follows that $M\cap T$ must be non-trivial. The subgroup $T$ is thus the monolith of $G/Q$, so $G/Q \in \Rs$ as required.
\end{proof}

We highlight the following special case,  where \tdlcsc stands for \textit{totally disconnected locally compact second countable} and \lcsc is defined likewise.

\begin{cor}\label{cor:RM_extension}
	Suppose that $G$ is a robustly monolithic \tdlcsc group with monolith $M$. If $H$ is a \lcsc group acting continuously and faithfully on $G$ by topological group automorphisms, then $(G\rtimes H)/\CC_{G\rtimes H}(M)$ is robustly monolithic.
\end{cor}

\begin{proof}
Proposition~\ref{prop:tdlc_action} ensures that $G \rtimes H$ is a \tdlcsc group, and $M$ is normal in $G \rtimes H$.  Proposition~\ref{prop:R_overgroup} now implies that $(G\rtimes H)/\CC_{G\rtimes H}(M) \in \Rs$.
\end{proof}

\subsection{Passing to a dense locally compact subgroup}

A more striking closure property is that $\Rs$ is closed under taking non-discrete dense locally compact subgroups, without any assumptions on normalizers.

\begin{thm}\label{thm:R_dense_embedding}
Suppose that $G\in \Rs$ and that $H$ is a non-discrete \tdlc group. If $H$ admits a dense embedding into $G$, then $H\in \Rs$. In particular, $H$ has a trivial quasi-center.
\end{thm}

\begin{proof}
Let $\psi:H\rightarrow G$ be a continuous, dense embedding.	In view of Proposition~\ref{prop:Rs_is_RF_[a]}, $G$ is regionally expansive, and applying Proposition~\ref{prop:cocompact_RF}, the group $H$ is also regionally expansive.  Write $M:= \Mon(G)$.  Note that $\CC_G(M) = \triv$ by Lemma~\ref{lem:monolith_centralizer}.

Let $L = \ol{\psi(\QZ(H))}$ and suppose toward a contradiction that $L$ is non-trivial.  Then $L$ is a non-trivial normal subgroup of $G$, so $M \le L$.  It follows by Lemma~\ref{lem:monolith_overgroup} that $L \in \Rs$, and hence $L$ is regionally expansive by Proposition~\ref{prop:Rs_is_RF_[a]}.  Fix $O\leq L$ an expansive compactly generated open subgroup of $L$ and let $U\leq O$ be a compact open subgroup with trivial normal core. 	

Since $\QZ(H)$ has dense image in $L$, there is a finite subset $A$ of $\QZ(H)$ such that $\psi(\grp{A})U=O$.  Then there is an open subgroup $W$ of $\psi\inv(U)$ centralized by $A$; in particular, $\psi(W)$ is normal in $\psi(\grp{A,W})$, and furthermore, $\psi(\grp{A,W})U=O$. Lemma~\ref{lem:no_cocompact} then implies that $O$ has a non-trivial compact normal subgroup contained in $U$, which contradicts our choice of $O$ and $U$.  From this contradiction, we see that $L = \triv$; since $\psi$ is injective, it follows that $\QZ(H)=\triv$.
	
Lemma~\ref{lem:minimal_geofaithful} now ensures the presence of minimal non-trivial closed normal subgroups of $H$. Suppose that $A$ and $B$ are distinct minimal non-trivial closed normal subgroups. The subgroups $A$ and $B$ commute and are normal in $H$, so $\overline{\psi(A)}$ and $\overline{\psi(B)}$ commute and are normal in $\overline{\psi(H)} = G$.  By Lemma~\ref{lem:monolith_centralizer}, at least one of $\overline{\psi(A)}$ and $\overline{\psi(B)}$ is trivial, so (since $\psi$ is injective) at least one of $A$ and $B$ is trivial, a contradiction. The group $H$ is thus monolithic.
	
Let $N$ be the monolith of $H$. Since $M$ is the monolith of $G$, $M\leq\overline{\psi(N)}$, and on the other hand, $N\leq \psi^{-1}(M)$, since $N$ is the monolith of $H$. We deduce that $\overline{\psi(N)}=M$. Applying Proposition~\ref{prop:cocompact_RF}, $N$ is regionally expansive, and as above, $N$ has a trivial quasi-center. Since $M$ is topologically simple, in particular monolithic, the same argument as in the previous paragraph shows that $N$ is monolithic. Since $N$ must be topologically characteristically simple, we conclude that $N$ is topologically simple. Thus, $H\in \Rs$.
\end{proof}

\subsection{Dense subgroups of compactly generated simple groups}
The following result shows several of the previous results can be upgraded in the case of dense locally compact subgroups of $G\in \Ss$.  Recall that a topological group $G$ is \emph{regionally compact} if it is expressible as a directed union of compact open subgroups.  Equivalently, $G$ is regionally compact if every compactly generated closed subgroup of $G$ is compact. Recall additionally that the discrete residual of $G$, denoted by $\Res{}(G)$, is the intersection of all open normal subgroups.

\begin{prop}\label{prop:DENSE_IN_S}
	Let $G\in \Ss$ and $H$ be a non-discrete \tdlc group with a dense embedding $\psi \colon H \to G$. Let $U \leq G$ be a compact open subgroup and $\{O_i\}_{i\in I}$ be a directed system of compactly generated open subgroups  of $H$ with $\bigcup_{i\in I}O_i=H$. There is $j \in I$ such that the following assertions hold for all $i\geq j$ . 
	\begin{enumerate}[label=(\roman*)]
		\item \label{it:DENSE_IN_S:1}
		 $O_i\in \Rs$.
		 
		\item  \label{it:DENSE_IN_S:2}
		$O_i = \Mon(O_i)(\psi\inv(U)\cap O_i)$. 
	
		\item \label{it:DENSE_IN_S:3}
		$\Res{}(O_i) = \Mon(O_i)$. 
		
		\item \label{it:DENSE_IN_S:4}
		If $\psi\inv(U)$ is regionally compact, then $\Mon(O_i) \in \Ss$ and $O_i/\Mon(O_i)$ is compact. In particular, $H$ is regionally $\Ss$-by-compact.\qedhere
	\end{enumerate}
\end{prop}

\begin{proof}
	\ref{it:DENSE_IN_S:1}.
	Theorem~\ref{thm:R_dense_embedding} ensures that $H \in \Rs$.  Via Theorem~\ref{thm:approx_R}, there is $j\in I$ such that each $O_i$ is monolithic with a topologically simple and regionally expansive monolith $M_i:= \Mon(O_i)$ for all $i\geq j$. Replacing $I$ by $\{i\in I\mid i\geq j\}$, we may assume that every $O_i$ is an element of $\Rs$.
	
	\medskip \noindent
	\ref{it:DENSE_IN_S:2}.
	 The set of monoliths $\{M_i\}_{i\in I}$ is directed, and $K:=\bigcup_{i\in I}M_i$ is a normal subgroup of $H$, by Lemma~\ref{lem:monolith_union}. The image $\psi(K)$ has a dense normalizer in $G$, so $\psi(K)$ is dense in $G$. Fixing a compact open subgroup $U$ of $G$, it follows from  Lemma~\ref{lem:Cayley-Abels} that there is some $j$ such that $\psi(M_j)U=G$, since $G$ is compactly generated. We conclude that $M_i(\psi\inv(U)\cap O_i)=O_i$ for all $i \geq j$. Replacing $I$ by $\{i\in I\mid i\geq j\}$, we may assume every $M_i$ has this property.
	 
	 	\medskip \noindent
	 \ref{it:DENSE_IN_S:3}.
	Set  $E:=\psi^{-1}(U)$. Since $U$ is residually finite and $\psi$ is continuous and injective, it follows that $E$ is residually discrete. We deduce from Proposition~\ref{prop:ResDiscr->SIN} that every compactly generated open subgroup of $E$ is a SIN-group.  Since $O_i$ is compactly generated and $O_i = M_i(E \cap O_i)$, there is a compactly generated open subgroup $E_i$ of $E$ such that $O_i = M_iE_i$. The quotient $O_i/M_i \simeq E_i/E_i\cap M_i$ is then SIN by Lemma~\ref{lem:SIN_quotient}. Assertion \ref{it:DENSE_IN_S:3} is now clear.
	
	\medskip \noindent
	\ref{it:DENSE_IN_S:4}.
	If $E$ is regionally compact, then the quotient $O_i/M_i \simeq E\cap O_i/E\cap M_i$ is compactly generated and regionally compact, hence it is compact. In particular, $M_i$ is compactly generated by Proposition~\ref{prop:cocompact_generation} and thus belongs to the class $\Ss$.
\end{proof}

Let us provide a general  criterion ensuring that the hypothesis of Proposition~\ref{prop:DENSE_IN_S}\ref{it:DENSE_IN_S:4} is fulfilled.

\begin{lem}\label{lem:E_locally_elliptic}
	Let $G$ and $H$ be \tdlc groups with compact open subgroups $U$ and $T$, respectively, and $\psi \colon H \to G$ be a dense embedding with $\phi(T)\leq U$. Suppose that $G$ has a basis of identity neighborhoods $(U_n)_{n\in I}$ consisting of open normal subgroups of $U$ such that $\N_{U_n}(\psi(T) \cap U_n)/\psi(T) \cap U_n$ is a locally finite group for all $n\in I$. Then $ \psi\inv(U)$ is regionally compact. 
\end{lem}

\begin{proof}
		By  Proposition~\ref{prop:ResDiscr->SIN},  every compactly generated open subgroup of $E:=\psi\inv(U)$ is a SIN group. Let $X \subset E$ be any finite set. The group $\la T \cup X\ra$ is a compactly generated open subgroup of $E$, so there is a compact open subgroup $T' \leq T$ that is normal in $\la T \cup X\ra$.  The map $\psi$ restricts to a homeomorphism from $T$ to $\psi(T)$ and the subgroups $\psi(T) \cap U_n$ form a basis of identity neighborhoods in $\psi(T)$.  For $n \geq 0$ large enough, we have $U_n \cap \psi(T) = U_n \cap \psi(T')$. Thus, $\psi(\la T \cup X\ra) \leq \N_U(U_n \cap \psi(T))$. 
		
		By hypothesis, $\N_{U_n}(\psi(T) \cap U_n)/\psi(T) \cap U_n$ is a locally finite group, and thus $\N_{U}(\psi(T) \cap U_n)/\psi(T) \cap U_n$ is also locally finite. The finitely generated subgroup $\psi(\la T \cup X\ra)/\psi(T)\cap U_n$ is therefore finite.  Since $\psi$ is injective, it follows that $T$ has finite index in $\la T \cup X \ra$, so $\la T \cup X \ra$ is compact.  We conclude that every compactly generated open subgroup of $E$ is compact. That is to say, $E$ is regionally compact.
\end{proof}

\subsection{Solvable subgroups of robustly monolithic groups}

Using localizations, we restrict the solvable subgroups of groups in $\Rs$. 

\begin{cor}\label{cor:dense_comm}
	Suppose that $G \in \Rs$. If $S \leq G$ is an infinite compact subgroup such that $\Comm_G(S)$ is dense, then the only virtually solvable normal subgroup of $S$ is the trivial group. 
\end{cor}
\begin{proof}
Form $G_{(S)}$ the localization of $G$ at $S$. The group $G_{(S)}$ is a dense locally compact subgroup of $G$, so $G_{(S)}\in \Rs$ by Theorem~\ref{thm:R_dense_embedding}. Proposition~\ref{prop:Rs_is_RF_[a]} ensures $G_{(S)}$ is [A]-semisimple, so $G_{(S)}$ admits no non-trivial locally normal abelian subgroups. The only virtually solvable normal subgroup of $S$ is thus the trivial group.
\end{proof}

For $G$ a \tdlc group and $U$ a compact open subgroup, Theorem~\ref{thm:Reid_localizations} shows that the group $\Comm_G(U_p)$ is dense in $G$ for any pro-$p$-Sylow subgroup $U_p$ of $U$. The next corollary is then immediate from Corollary~\ref{cor:dense_comm}.

\begin{cor}\label{cor:ab_sylow_R}
	Suppose that $G\in \Rs$, $U$ is a compact open subgroup of $G$, and $S$ is a pro-$p$-Sylow subgroup of $U$. If $S$ is infinite, then the only virtually solvable normal subgroup of $S$ is the trivial group. In particular, $S$ is not solvable.
\end{cor}

\section{Regionally elementary groups}

\subsection{Definition and basic properties}

We recall the class of elementary \tdlcsc groups as defined in \cite{W_E_14}.

\begin{defn}\label{defn:elementary}
	The collection of \emph{elementary groups} is the smallest class $\Es$ of \tdlcsc groups such that 
	\begin{enumerate}
		\item $\Es$ contains the second countable profinite groups and countable discrete groups.
		\item $\Es$ is closed under group extensions within the class of \tdlcsc groups.
		\item $\Es$ is closed under taking closed subgroups.
		\item $\Es$ is closed under taking Hausdorff quotients.
		\item $\Es$ is closed under countable directed unions of open subgroups.\qedhere
	\end{enumerate}
\end{defn}

The definition of elementary groups easily extends to groups that are not $\sigma$-compact.

\begin{defn} 
	A \tdlc group $G$ is called \emph{regionally elementary} if every compactly generated open subgroup is elementary. It is called \emph{regionally SIN} if every compactly generated open subgroup is a SIN group.
\end{defn}

Regionally elementary groups are necessarily first countable.  Since the class $\Es$ is closed under closed subgroups and countable directed unions, one sees that a second countable \tdlc group is elementary if and only if it is regionally elementary.

The class of regionally elementary groups enjoys several closure properties. These all follow from the definition of the class of elementary groups except for claim (ii). This claim follows from \cite[Theorem 3.8]{W_E_14}, which shows claim (ii) holds for the class of elementary groups.

\begin{prop}\label{prop:reg_elem:closure_props}
The class of regionally elementary groups enjoys the following closure properties within the class of \tdlc groups:
	\begin{enumerate}[label=(\roman*)]
		\item It is closed under forming group extensions.
		\item It is closed under taking preimages via continuous, injective maps. In particular, it is closed under taking closed subgroups.
		\item It is closed under taking Hausdorff quotients.
		\item It is closed under taking directed unions of open subgroups.\qedhere
	\end{enumerate}
\end{prop}

\subsection{Decomposition rank}

The class of elementary groups admits an ordinal valued rank called the \textit{decomposition rank} (see \cite[Section 4]{W_E_14}). It is the unique ordinal valued function $\xi:\Es\rightarrow \omega_1$ with the following properties:

\begin{enumerate}[label=(\alph*)]
	\item $\xi(\{1\})=1$;
	\item If $G\in \Es$ is non-trivial and $(O_i)_{i\in \Nbb}$ is an $\subseteq$-increasing exhaustion of $G$ by compactly generated open subgroups, then
	\[
	\xi(G)=\sup_{i\in \Nbb}(\xi(\Res{}(O_i)))+1.\qedhere
	\]
\end{enumerate}

The lowest possible rank of a non-trivial group is two.  Given a non-trivial \tdlcsc group $G$, we have $\xi(G)=2$ if and only if $G$ is a directed union of compactly generated open SIN-groups; see \cite[Lemma~3.15]{RW_LC_15}. 

The  decomposition rank on elementary groups extends to regionally elementary groups: For $G$ a non-trivial regionally elementary group, we define the decomposition rank to be
\[
\xi_r(G):=\sup\{\xi(\Res{}(O))\mid O\leq G\text{ compactly generated and open}\}+1.
\]
If $G$ is trivial, we put $\xi_r(G)=1$. For an elementary group $G$, it is clear that that $\xi_r(G)=\xi(G)$.  We will thus abuse notation and simply write $\xi(G)$ for the decomposition rank of a regionally elementary group. We note also that a non-trivial group $G$ is regionally SIN if and only if $\xi_r(G)=2$.

The decomposition rank is well-behaved for regionally elementary groups. These results follow easily from the corresponding facts for elementary groups.
\begin{prop}\label{prop:xi_regionally_ele}
	The decomposition rank for regionally elementary groups enjoys the following properties:
	\begin{enumerate}[label=(\roman*)]
	\item If $\{1\}\rightarrow K\rightarrow G\rightarrow Q\rightarrow\{1\}$ is a short exact sequence of regionally elementary groups, then $\xi(G)\leq \xi(K)+\xi(Q)$. \textup{(cf. \cite[Lemma 3.8]{RW_LC_15})}
	\item Suppose that $H$ is a \tdlc group and $G$ is regionally elementary. If there is a continuous embedding of $H$ into $G$, then $H$ is regionally elementary with $\xi(H)\leq \xi(G)$. \textup{(cf. \cite[Corollary 4.10]{W_E_14})}\qedhere
	\end{enumerate}
\end{prop}

The recursion that produces the decomposition rank gives a chain condition that characterizes regionally elementary groups.

\begin{prop}\label{prop:reg_elem_char}
Let $G$ be a   \tdlc group.  Then $G$ is regionally elementary if and only if $G$ is first countable and for every descending chain of compactly generated closed subgroups
\[
K_0\geq K_1\geq K_2\geq \dots
\]
such that $K_{i+1} \le \Res(K_i)$ and for all $i\in \Nb$, there exists $i \in \Nb$ with $K_i = \{1\}$.
\end{prop}

\begin{proof}
Suppose that $G$ is regionally elementary. Then $G$ has an open subgroup that is second countable, so $G$ is first countable. Let now $(K_i)_{i\in \Nb}$ be an infinite descending chain of compactly generated closed subgroups such that $K_{i+1}\leq \Res(K_i)$ for all $i$.  In particular, $\xi(K_{i+1}) \le \xi(\Res(K_i))$ by Proposition~\ref{prop:xi_regionally_ele}. Whenever $K_i$ is non-trivial, we see that
\[
\xi(K_i) = \xi(\Res(K_i)) +1 \ge \xi(K_{i+1})+1.
\]
Therefore $\xi(K_i)> \xi(K_{i+1})$ for all $i$ such that $K_i$ is non-trivial.  Since $\xi$ takes ordinal values, we cannot have an infinite descending chain $\xi(K_1) > \xi(K_2) > \dots$, so the sequence $(K_i)_{i\in \Nb}$ eventually stabilizes at $\{1\}$. We conclude that $G$ satisfies the chain condition.

\medskip

Conversely,  suppose that $G$ is first countable but not regionally elementary.  There is thus some compactly generated closed subgroup $K_0\leq G$ that is not regionally elementary.  The quotient $K_0/\Res(K_0)$ is regionally elementary; indeed it is compact-by-discrete by Proposition~\ref{prop:ResDiscr->SIN}.  It follows by Proposition~\ref{prop:reg_elem:closure_props} that $\Res(K_0)$ is not regionally elementary. There thus exists a compactly generated open $K_1\leq \Res(K_0)$ that is non-elementary. Continuing in this fashion produces an infinite chain $K_0\geq K_1\geq K_2\dots$ such that $K_{i+1}\leq \Res(K_i)$ and $K_i$ is non-elementary for all $i\in \Nb$. We conclude that $G$ does not satisfy the chain condition.
\end{proof}

\subsection{Robustly monolithic groups are not regionally elementary}

Using the permanence properties, we here see that $\Rs$ contains no regionally elementary groups.  In fact, we find that $\Rs$ has a recursive property of its own that is incompatible with the class of regionally elementary groups.

\begin{prop}\label{prop:R_chain}
Let $G \in \Rs$.  Then at least one of the following holds:
\begin{enumerate}[label=(\roman*)]
\item There is a closed $K \le G$ such that $K \in \Ss$.
\item There is an infinite descending chain
\[
M_1 > M_2 > M_3 > \dots
\]
of topologically simple closed subgroups of $G$ such that for all $i$, $M_i\in \Rs$, $M_i$ is not compactly generated, and $M_{i+1}=\Mon(K_{i})$ for $K_i$ some compactly generated open subgroup of $M_i$.\qedhere
\end{enumerate}
\end{prop}

\begin{proof}
Let $G \in \Rs$.  Theorem~\ref{thm:approx_R} supplies a compactly generated open $K_0 \le G$ such that $K_0 \in \Rs$.  Letting $M_1$ be the monolith of $K_0$, Lemma~\ref{lem:monolith_overgroup} ensures that $M_1 \in \Rs$.  There is then a compactly generated open $K_1 \le M_1$ such that $K_1 \in \Rs$, and so on.

If we can choose $K_{i+1} = K_{i}$ for some $i$, then $K_i$ must be compactly generated and equal to its own monolith, so $K_i \in \Ss$ and (i) holds.  If instead we are forced to choose $K_{i+1} < K_i$ for all $i$, then we obtain a chain as in (ii) by taking $M_{i+1} = \Mon(K_i)$.
\end{proof}

\begin{cor}\label{cor:R_no_elementary}
	The class $\Rs$ contains no regionally elementary groups.
\end{cor}

\begin{proof}
Let $(K_i)_{i\in \Nb}$ be the chain produced in the proof of Proposition~\ref{prop:R_chain}. Each $K_i$ is compactly generated, and $K_{i+1} \le \Mon(K_i) \le \Res(K_i)$ for all $i$.  In view of Proposition~\ref{prop:reg_elem_char}, $G$ is not regionally elementary; note that it can be the case that $K_{i+1}= \Res(K_i)=K_i$ for all sufficiently large $i$, but Proposition~\ref{prop:reg_elem_char} nonetheless applies.
\end{proof}

\begin{cor} 
Let  $G\in \Rs$ and $H$ be a regionally elementary \tdlc group admitting a dense embedding into  $G$. Then $H$ is discrete.
\end{cor}

\begin{cor} 
If $G$ is a non-discrete topologically simple regionally elementary group, then $G$ is not regionally expansive.
\end{cor}

Examples satisfying the hypotheses of the previous corollary may be found in \cite[Proposition 3.2]{Wi07}. We remark that all known examples have decomposition rank two.

We also have some control over the rank of $G/\Mon(G)$ with respect to continuous dense embeddings.

\begin{thm}\label{thm:R_finite_rank}
Let  $G\in \Rs$ and $H$ be a non-discrete \tdlc group with a dense embedding $\psi \colon H \to G$.    If $G/\Mon(G)$ is regionally elementary of finite rank, then so is $H/\Mon(H)$.
\end{thm}

\begin{proof}
We have $H \in \Rs$ by Theorem~\ref{thm:R_dense_embedding}.  Set $M: = \Mon(G)$ and $N := \Mon(H)$.  Setting $E:=\psi^{-1}(M)$, we see that $H/E\injects G/M$, and therefore $H/E$ is regionally elementary of finite rank. For $U\leq G$ a compact open subgroup, we have $M\leq \overline{\psi(N)}\leq \psi(N)U$. Setting $V:=\psi^{-1}(U)$, we deduce that $E\leq NV$. The subgroup $V$ is residually discrete and so is a regionally SIN group. It follows the quotient $E/N$ is then regionally elementary with rank two. Since regionally elementary groups are closed under group extension and the decomposition rank is subadditive, we conclude that $H/N$ is regionally elementary with finite rank.
\end{proof}

\section{The centralizer lattice}

\subsection{Preliminaries}
Following \cite{CRW1}, let us define the \textit{structure lattice} $\LN(G)$. For a \tdlc group $G$, let $\LN(G)$ be the set of closed locally normal subgroups of $G$ modulo the equivalence relation $\sim$ where $K \sim L$ if $K \cap L$ is open in both $K$ and $L$. Define a partial ordering on $\LN(G)$ by $\alpha\leq\beta$ if there are representatives $K\in \alpha$ and $L\in \beta$ such that $K\subseteq L$.

As explained in \cite{CRW1}, associated to any [A]-semisimple \tdlc group there is a canonical Boolean algebra called the \textit{centralizer lattice}.  This lattice can be defined either locally or globally.

\begin{itemize}

\item[] \textit{Locally:}  Define the centralizer map $\bot: \LN(G) \rightarrow \LN(G)$ by $[K]^{\bot} := [\QC_G(K)]$.  The centralizer lattice $\LC(G)$ is defined to be the image of $\LN(G)$ under $\bot$ with the partial order inherited from $\LN(G)$.

\medskip
\item[] \textit{Globally:} We define $\mathrm{LC}(G)$ as the set
\[
\{ \CC_G(K) \mid K \le G\text{ and } \N_G(K) \text{ is open in } G\},
\]
ordered by inclusion.
\end{itemize}
These two versions of the centralizer lattice are canonically isomorphic. Indeed, given a locally normal subgroup $K$, the centralizer $\CC_G(K)$ depends only on $[K]$ by Lemma~\ref{lem:[A]-semisimplicity_qcent}. Moreover,  the global version maps onto the local version via $\CC_G(K) \mapsto [\CC_G(K)]$ by Lemma~\ref{lem:[A]-semisimple:commuting}. We shall freely switch between the local and global perspective of the centralizer lattice.  It is easy to verify that $\Aut(G)$ admits an action on $\LC(G)$, and in particular, $G$ acts on $\LC(G)$.  We also remark that the meet operation for the centralizer lattice is especially straightforward: in $\LC(G)$, it is the operation induced by intersecting representatives, and in $\mathrm{LC}(G)$ it is simply intersection.

Let us note a general situation in which the structure and centralizer lattices can be pulled back along a homomorphism.  

\begin{lem}\label{lem:PullBack}
Let $G$ and $H$ be [A]-semisimple \tdlc groups, $\psi \colon H \to G$ be a continuous, injective homomorphism, and $V \leq H$ be a compact open subgroup. Then the map $\psi^* \colon \LN(G) \to \LN(H)$ defined by
\[
 [L] \mapsto  [\psi^{-1}(L \cap \psi(V))]
\]
is well-defined, order-preserving, meet-preserving, and $H$-equivariant where $H$ acts on $\LN(G)$ via $h.[L] := [\psi(h)L\psi(h\inv)]$. Let us assume furthermore that $\mathrm C_{\psi(V)}(L \cap \psi(H))= \mathrm C_{\psi(V)}(L)$ for every compact locally normal subgroup $L$ of $G$. Then:

\begin{enumerate}[label=(\roman*)]
\item $\psi^*(\alpha^\perp) = (\psi^*(\alpha))^\perp$ for all $\alpha \in \LN(G)$, and
\item if in addition $\psi^*(\alpha) >0$ for all $\alpha >0$, then $\psi^*$ induces an injective homomorphism $\LC(G) \to \LC(H)$ of Boolean algebras. \qedhere
\end{enumerate}
\end{lem}

\begin{proof}
That $\psi^*$ is well-defined, order-preserving, and  $H$-equivariant is straightforward.  For any locally normal subgroups $L$ and $M$ of $G$, we have
\[
\psi\inv(L \cap M \cap \psi(V)) = \psi\inv(L \cap \psi(V)) \cap \psi\inv(M \cap \psi(V)).
\]
It follows that $\psi^*$ preserves meets.

For (i), take $U \leq G$ a compact open subgroup containing $\psi(V)$ and fix $[L]=\alpha\in \LN(G)$. By definition, $\alpha^\perp = [\CC_U(L)]$. Moreover,
\[
\psi^*(\alpha^\perp)  = [\psi^{-1}(\CC_U(L) \cap \psi(V))] = [\psi^{-1}(\CC_{\psi(V)}(L))].
\]
Invoking the extra hypothesis that  $\mathrm C_{\psi(V)}(L \cap \psi(H))= \mathrm C_{\psi(V)}(L)$, we have 
$$ [\psi^{-1}(\CC_{\psi(V)}(L))] = [\psi^{-1}(\CC_{\psi(V)}(L \cap \psi(H)))] = [\CC_V(\psi^{-1}(L))],$$ 
since $\psi$ is injective. On the other hand, 
\[
(\psi^* (\alpha))^\perp = [\CC_{V}(\psi^{-1}(L \cap \psi(V)))] = [\CC_{V}(\psi^{-1}(L) \cap V)].
\]	
The group $H$ is [A]-semisimple and $\psi\inv(L)$ is locally normal in $H$, so by Lemma~\ref{lem:[A]-semisimple:commuting}, we have $\CC_V(\psi^{-1}(L)) = \CC_{V}(\psi^{-1}(L) \cap V)$.  We conclude that $\psi^*(\alpha^\perp) = (\psi^*(\alpha))^\perp$ for all $\alpha \in \LN(G)$.

For (ii), we assume in addition that $\psi^*(\gamma) >0$ for all $\gamma >0$. Take $\alpha$ and $\beta$ in $\LC(G)$ and suppose that $\psi^*(\alpha)=\psi^*(\beta)$. Set  $\gamma := \alpha \wedge \beta^\perp$.  Since $\gamma \leq \alpha$, we have $\psi^*(\gamma) \leq \psi^*(\alpha)$.  We also see that 
\[
\psi^*(\beta)  \wedge \psi^*(\gamma) \le \psi^*(\beta) \wedge \psi^*(\beta^\perp) = \psi^*(\beta) \wedge \psi^*(\beta)^\perp = 0.
\]
With these observations in hand, we deduce that
\[
\psi^*(\gamma)  = \psi^*(\gamma)   \wedge \psi^*(\alpha)  = \psi^*(\gamma)  \wedge \psi^*(\beta)  = 0.
\]
Thus, $\alpha \wedge \beta^\perp=\gamma = 0$. Reversing the roles of $\alpha$ and $\beta$, it is also the case that $\alpha^\perp \wedge \beta = 0$. It now follows that $\alpha = \alpha \wedge \beta = \beta$, and so the map $\LC(H) \to \LN(G)$ via $\alpha \mapsto \psi^*(\alpha)$ is injective.  Given part (i) and the fact that $\psi^*$ preserves meets, we see that it restricts to a Boolean algebra homomorphism from $\LC(H)$ to $\LC(G)$, verifying (ii).
\end{proof}

Lemma~\ref{lem:PullBack} implies in particular that one can restrict the structure and centralizer lattices to suitable locally normal subgroups.

\begin{prop}\label{prop:FaithfulLC}
Let $G$ be an $[A]$-semisimple \tdlc group, let $K$ be a closed locally normal subgroup of $G$, and let $V$ be a compact open subgroup of $K$.  Suppose that $\QC_G(K) = \triv$ and that $K$ is $[A]$-semisimple.  Then the map $\chi:\LN(G) \to \LN(K)$ defined by 
\[
[L] \mapsto [L\cap K]=:[L]_{(K)}
\]
enjoys the following properties.
\begin{enumerate}[label=(\roman*)]
 \item \label{prop:FaithfulLC:1} It is order-preserving, meet-preserving, $\Comm_G(V)$-equivariant, and every non-zero element has a non-zero image;
 \item \label{prop:FaithfulLC:2} $(\alpha^\perp)_{(K)} = (\alpha_{(K)})^\perp$; and
 \item \label{prop:FaithfulLC:3} the restriction to $\LC(G)$ yields an injective homomorphism  $\LC(G) \to \LC(K)$ of Boolean algebras.\qedhere
\end{enumerate}
\end{prop}

\begin{proof}
Let $G_{(V)}$ be the group $\Comm_G(V)$ equipped with the $V$-localized topology. The structure lattice $\LN(K)$ can be identified with $\LN(G_{(V)})$, since $K$ is open in $G_{(V)}$. In view of this identification, $\chi$ is $\Comm_G(V)$-equivariant.  Suppose that $L$ is a non-trivial closed locally normal subgroup of $G$ and let $\alpha := [L]$. Lemma~\ref{lem:[A]-semisimple:commuting} ensures that $\mathrm C_{K}(L) = \mathrm C_{K}(L \cap K)$. Moreover 
 $L \cap K$ is non-trivial, since $\QC_G(K)=\{1\}$.  In fact, $L \cap K$ must be non-discrete, so $[L \cap K] > 0$.  Every non-zero element of $\LN(G)$ thus has a non-zero image.  Finally, the inclusion map $\psi: G_{(V)} \rightarrow G$ is injective and continuous, and the structure and centralizer lattices of $G_{(V)}$ can be identified with those of $V$. Thus all the hypotheses of Lemma~\ref{lem:PullBack} are satisfied, and the remaining parts of the proposition follow.
\end{proof}

Proposition~\ref{prop:FaithfulLC} applies in particular when $G \in \Rs$ and $K$ is the monolith of $G$.

\subsection{The structure of the Stone space}

Associated to any Boolean algebra $\mc{A}$ is a topological space $\mf{S}(\mc{A})$, called the \textit{Stone space} of $\mc{A}$.  The points of $\mf{S}(\mc{A})$ are the ultrafilters of $\mc{A}$, and the basic open sets of the topology are the sets of the form $\tilde{\alpha} := \{\mf{p} \in \mf{S}(\mc{A}) \mid \alpha \in \mf{p}\}$.  The complement of $\tilde{\alpha}$ as a subset of $\mf{S}(\mc{A})$ is exactly the open set $\tilde{\beta}$ where $\beta = \alpha^\perp$, so $\tilde{\alpha}$ is clopen for all $\alpha \in \mc{A}$.  The Stone space is a compact zero-dimensional space, and its set of clopen subsets is exactly $\{\tilde{\alpha} \mid \alpha \in \mc{A}\}$. The collection of clopen sets ordered by inclusion is canonically isomorphic with $\mc{A}$ as a Boolean algebra.  We will abuse notation and identify $\alpha$ with $\tilde{\alpha}$, so given $\alpha \in \mc{A}$ and $\mf{p} \in \mf{S}(\mc{A})$, the expressions ``$\mf{p} \in \alpha$" and ``$\alpha \in \mf{p}$" should be understood as synonymous.  In particular, we will regard elements of $\LC(G)$ as subsets of $\mf{S}(\LC(G))$; this allows us to define infinite intersections and unions of elements of $\LC(G)$. If a group $G$ acts on $\mc{A}$, we say that $G$ fixes $\alpha\in \mc{A}$ if $g.\alpha=\alpha$ for all $g\in G$. Equivalently, $G$ fixes \textit{setwise} the clopen set $\tilde{\alpha}$ in the Stone space.  If $G\acts \LC(G)$ faithfully, then the action of $G$ on $\LC(G)$ is \textit{micro-supported}, via \cite[Theorem II]{CRW1}. That is to say, for every non-empty, proper clopen $\upsilon \subseteq \mf{S}(\LC(G))$, $\rist_G(\upsilon):=\{g\in G\mid \forall x\notin \upsilon\; g(x)=x\}$ is non-trivial.

Given a group $G$ and subgroups $H$ and $K$ of $G$, define $\CC^0_H(K) := K$ and $\CC^{n+1}_H(K):= \CC_H(\CC^n_H(K))$.  If $H \ge K$, we observe that $\CC^2_H(K) \ge K$ and $\CC^3_H(K) = \CC_H(K)$. In an [A]-semisimple \tdlc group, the elements of the global centralizer lattice can be characterized as the locally normal subgroups $K$ of $G$ that satisfy $K = \CC^2_G(K)$. The set of elements of $\LC(G)$ that are fixed under the action of $G$ is denoted by $\LC(G)^G$\index{$\LC(G)^G$}.

\begin{prop}\label{prop:lcent_fixedpt:mono}
Let $G$ be a \tdlc group that is [A]-semisimple and monolithic.  Then $\LC(G)^G = \{0,\infty\}$, and the action of $G$ on $\mf{S}(\LC(G))$ is topologically transitive.
\end{prop}

\begin{proof}
Let $\upsilon_1$ and $\upsilon_2$ be non-empty open subspaces of $\mf{S}(\LC(G))$; we must show that there is $g \in G$ such that $g\upsilon_1$ has non-empty intersection with $\upsilon_2$.  Since the Stone space is zero-dimensional, we can take $\upsilon_1$ and $\upsilon_2$ to be clopen, hence they correspond to non-trivial elements of $\mathrm{LC}(G)$. Say that $\upsilon_i$ corresponds to $L_i \in \mathrm{LC}(G)$ for $i=1,2$, and note that $L_i$ is a non-trivial closed locally normal subgroup of $G$.  Since $G$ is $[A]$-semisimple, the monolith of $G$ is non-abelian, and it follows by Lemma~\ref{lem:monolith_centralizer} that $\ngrp{L_1}$ and $\ngrp{L_2}$ do not commute.  There is therefore some $g \in G$ such that $gL_1g\inv$ and $L_2$ do not commute.  By Lemma~\ref{lem:[A]-semisimple:commuting}, the intersection $gL_1g\inv \cap L_2$ is non-discrete. We deduce that $g\upsilon_1 \wedge \upsilon_2 > 0$ as elements of $\LC(G)$, so $g\upsilon_1$ and $\upsilon_2$ have non-empty intersection as subspaces of $\mf{S}(\LC(G))$.

We conclude that the action of $G$ on $\mf{S}(\LC(G))$ is topologically transitive; the same argument shows that $\LC(G)^G = \{0,\infty\}$.
\end{proof}

In general, the action of $G$ on $\LC(G)$ need not be faithful. We can, however,  say something about fixed points of topologically simple closed normal subgroups. Recall that $\mc{M}(G)$ is the set of non-trivial minimal closed normal subgroups of $G$. For $M\in \mc{M}(G)$, recall also that $[\CC_G^2(M)]$ is the least upper bound of $[M]$ in $\LC(G)$; this follows by considering the map $\bot:\LN(G)\rightarrow \LC(G)$.

\begin{lem}\label{lem:lcent_fixedpt_of_simple}
Let $G$ be a \tdlc group that is [A]-semisimple, $M \in \mc{M}(G)$, and $\alpha_M:=[\CC_G^2(M)]$. 
\begin{enumerate}[label=(\roman*)]
\item Let $K$ be a non-trivial closed locally normal subgroup of $G$.  Suppose that $K \cap M > \triv$, $\CC_M(K) > \triv$, and $[\CC_M(K)]$ is fixed by the conjugation action of $M$.  Then $\CC_M(K)$ is a proper non-trivial closed normal subgroup of $M$.  In particular, $M$ cannot be topologically simple.
\item If $M$ is topologically simple, then there are no fixed points $\beta$ of the action of $M$ on $\LC(G)$ such that $0 < \beta < \alpha_M$.\qedhere
\end{enumerate}
\end{lem}
 
\begin{proof}
Take $K$ as hypothesized for (i).  Let $L := \QC_M(\CC_M(K))$, take $g \in L$, and let $m \in M$.  The conjugate $mgm\inv$ centralizes an open subgroup of $m\CC_M(K)m\inv$, and  $m\CC_M(K)m\inv$ contains an open subgroup of $\CC_M(K)$, since $[\CC_M(K)]$ is fixed by the conjugation action of $M$. Thus, $mgm\inv \in L$, and we deduce that $L$ is normal in $M$. Applying Lemma~\ref{lem:[A]-semisimplicity_qcent}, it follows that $L = \CC^2_M(K)$.

We see that $\CC_M(L) = \CC^3_M(K) = \CC_M(K)$.  The centralizer $\CC_M(K)$ is thus a non-trivial closed normal subgroup of $M$, since $L$ is a normal subgroup of $M$. On the other hand, $\CC_M(K)$ intersects $K \cap M$ trivially, so $\CC_M(K)$ is a proper subgroup of $M$, proving (i).

For (ii), let $0 < \beta < \alpha_M$ and say that $\beta = [\CC_G(K)]$ where $\CC_G(K) < \CC^2_G(M)$.  Since $\alpha_M$ is the least upper bound of $[M]$ in $\LC(G)$, we see that $M \nleq \CC_G(K)$, so by Lemma~\ref{lem:[A]-semisimple:commuting}, we infer that $K \cap M > \triv$.

Suppose that $M \cap \CC_G(K) = \triv$.  Both $M$ and $\CC_G(K)$ are locally normal subgroups, so Lemma~\ref{lem:[A]-semisimple:commuting} ensures that $M$ and $\CC_G(K)$ commute. Hence, $\CC_G(K) \le \CC_G(M) \cap \CC^2_G(M)$.  Since $\CC_G(M) \cap \CC^2_G(M) = \Z(\CC_G(M)) = \triv$, we have a contradiction to the hypothesis that $\beta > 0$.  We conclude that $M \cap \CC_G(K) =\CC_M(K) > \triv$.  Since $M$ is topologically simple, part (i) implies that $\beta$ is not fixed by the action of $M$.
\end{proof}

\begin{prop}\label{prop:LC_faithful}
Let $G$ be a \tdlc group that is [A]-semisimple and monolithic with a topologically simple monolith. Then
\[
\LC(G)^G = \LC(G)^M = \{0,\infty\},
\]
and $G$ acts faithfully on $\LC(G)$ unless $|\LC(G)|=2$.
\end{prop}

\begin{proof}
Recalling that $\alpha_M = \infty$ when $M$ is the monolith of $G$, the proposition is immediate from Lemma~\ref{lem:lcent_fixedpt_of_simple}.
\end{proof}

\subsection{Dynamics on the Stone space}

The centralizer lattice is a local invariant, in the sense that $\LC(G)$ is canonically isomorphic to $\LC(U)$ for $U$ any open subgroup of $G$.  In particular, if $G$ acts faithfully on its centralizer lattice, then every open subgroup of $G$ acts faithfully on its own centralizer lattice.  In the case that $G$ is regionally expansive, we can use the action of an expansive compactly generated open subgroup $O$ of $G$ to impose important homogeneity properties on the action of $G$.

In our first proposition, we consider the slightly more general setting of $G$-invariant subalgebras of $\LC(G)$. We do so for two reasons. First, there is a canonical subalgebra of $\LC(G)$ called the \textit{decomposition lattice}, see \cite{CRW1, CRW2}. Second, the full centralizer lattice can be very large and difficult to determine, so it can be convenient to reduce to a $G$-invariant subalgebra that is countable or can be obtained more explicitly (for example, such a subalgebra can be obtained from a micro-supported action of $G$ on a compact zero-dimensional space). 

\begin{prop}\label{cor:rf_compressible:mono}
Let $G$ be a \tdlc group that is regionally expansive, [A]-semisimple, and monolithic. Set $M := \Mon(G)$, say that $\mc{A}\subseteq \LC(G)$ is a $G$-invariant subalgebra, and suppose  that $M$ acts non-trivially on $\mc{A}$.
\begin{enumerate}[label=(\roman*)]
\item \label{cor:rf_compressible:1} There exists $\beta \in \mc{A}$ with $0 < \beta$ such that for all $\gamma \in \mc{A}$ with $0 < \gamma$, there is $g \in G$ such that $g\beta < \gamma$.
\item \label{cor:rf_compressible:2} Regarded as a subspace of $\mf{S}(\mc{A})$, the set $\upsilon: = \bigcup_{g \in G}g\beta$, for $\beta$ as in (i), is a dense subset that is both the unique smallest non-empty $G$-invariant open subset and the complement of the fixed-point set of $M\acts\mf{S}(\mc{A})$.\qedhere
\end{enumerate}
\end{prop}

\begin{proof}
For the proof of (i), let us begin with a reduction. In view of Corollary~\ref{cor:minimal_A-semisimple}, we may find an expansive compactly generated open subgroup $O$ of $G$ such that $O$ is monolithic. Taking $O$ perhaps larger, we may also assume that the monolith of $O$ acts non-trivially on $\mc{A}$. The Boolean algebra $\LC(O)$ is naturally isomorphic to $\LC(G)$. We are thus free to assume that $G$ is compactly generated. Say that $G =\langle X \rangle$ where $X$ is a compact open symmetric identity neighborhood.  Note also that as $G$ is monolithic and $M$ acts non-trivially on $\mc{A}$, it follows that $G$ acts faithfully on $\mc{A}$.

Fix a compact open subgroup $U$ of $G$ that does not contain $M$ and set $V: = U \cap M$. The group $V$ is commensurated and locally normal in $G$.  We see additionally that
\[
 W := \bigcap_{x \in X}xVx\inv=\bigcap_{x\in X}x(U\cap M)x\inv=\bigcap_{x\in X}xUx\inv \cap M.
 \]
 The compactness of $X$ ensures that $\bigcap_{x\in X}xUx\inv $ is actually an intersection of finitely many conjugates of $U$, so $W$ is an open subgroup of $V$.  Since $V$ is compact and acts faithfully on $\mc{A}$, there is a finite subset $\mc{B}$ of $\mc{A}$ such that the pointwise stabilizer $V_{(\mc{B})}$ is contained in $ W$. Given that we are working in a Boolean algebra on which $V$ has finite orbits, we can take $\mc{B} = \{\beta_1,\dots,\beta_d\}$ to be a partition of $\mc{A}$ that is preserved setwise by $V$. 
 
Take $\gamma \in \mc{A}$ with $0 < \gamma$.  Let $R$ be a representative of $\gamma$ and put $S := \CC^2_U(R \cap U)$. The group $S$ has non-trivial intersection with $M$, since $\QC_G(M)=\{1\}$ by Proposition~\ref{prop:[A]-semisimplicity_LN}. Since $V$ is a proper subgroup of the monolith $M$, there is some conjugate $g(S \cap M)g\inv$ of $S \cap M$ that is not contained in $V$.  Choose $h \in G$ of minimal word length with respect to $X$ such that $h(S \cap M)h\inv$ is not in $V$.  We may write $h = xg\inv$ for some $g \in G$ of shorter length. Thus, $g\inv(S \cap M)g \le V$, and since $xWx\inv \le V$, we see that $g\inv(S \cap M)g$ is a subgroup of $V$ but not $W$. There is $\beta_i\in \mc{B}$ and $v\in S\cap M$ such that $g\inv vg\beta_i \neq \beta_i$.  The conjugate $g\inv v g$ preserves the partition $\mc{B}$ setwise, so we have $g\inv v g\beta_i \wedge \beta_i = 0=v g\beta_i \wedge g\beta_i$. Seeing as $\gamma^\bot = [\CC_G(S)]$, the subgroup $S$ fixes every $\delta \in \LC(G)$ such that $\delta \le \gamma^\bot$, and also $S$ fixes $\gamma$. It follows that $g\beta_i < \gamma$.

At this point, we have a set $\mc{B}$ of non-zero elements of $\mc{A}$ such that for all $\gamma \in \mc{A}$ with $0 < \gamma$, there exists $g \in G$ and $\beta_i \in \mc{B}$ such that $g\beta_i < \gamma$.  Let us consider $\mc{B}$ of smallest possible size, and suppose there exist $\beta_i, \beta_j \in \mc{B}$ and $g \in G$ such that $g\beta_i \wedge \beta_j = \gamma > 0$.  Then $h\beta_k < \gamma$ for some $\beta_k \in \mc{B}$ and $h\in G$, and the minimality of $\mc{B}$ ensures $\beta_i = \beta_k$ and $\beta_j=\beta_k$.  Thus $g\beta_i \wedge \beta_j = 0$ for all $g \in G$ and $\beta_i, \beta_j \in \mc{B}$ distinct. On the other hand, by Proposition~\ref{prop:lcent_fixedpt:mono}, $G$ acts topologically transitively on $\mf{S}(\LC(G))$. For any non-empty open subsets $\beta_i,\beta_j$ of $\mf{S}(\mc{A})$, some $G$-translate of $\beta_i$ must intersect $\beta_j$.  We deduce that $|\mc{B}| = 1$, completing the proof of claim (i).

We now consider claim (ii); here we do not assume that $G$ is compactly generated. Let $\beta$ be as in part (i) and let $\upsilon := \bigcup_{g \in G}g\beta$ be regarded as a subspace of $\mf{S}(\mc{A})$.  The set $\upsilon$ is contained in every non-empty $G$-invariant open subset of $\mc{A}$, and $\upsilon$ is itself non-empty and $G$-invariant. Hence, $\upsilon$ is the unique smallest such set.  In particular, the complement of $\upsilon$ has empty interior, so $\upsilon$ is dense in $\mf{S}(\mc{A})$.  Let $\upsilon'$ be the complement of the fixed-point set of $M$.  Then $\upsilon'$ is open, non-empty, and $G$-invariant, so $\upsilon \subseteq \upsilon'$.  On the other hand, $M = \ngrp{T \cap M}_G$ where $T$ is a representative of $\beta$.  The support of $T$ is contained in $\beta$, so every point of $\mf{S}(\mc{A})$ outside of $\upsilon$ is fixed by $\lla T \rra_G$. We conclude that $M$ fixes the complement of $\upsilon$. We deduce that $\upsilon^c\subseteq (\upsilon')^c$, so $\upsilon = \upsilon'$ as claimed.
\end{proof}

\begin{rmk}
One can consider more generally a \tdlc group $G$ that is regionally expansive, [A]-semisimple, but not necessarily monolithic.  In this case Proposition~\ref{prop:minimal_A-semisimple} applies and $\mc{M}(G)$ is non-empty but finite.  The elements of $\mc{M}(G)$ then give rise to a canonical partition of $\mf{S}(\LC(G))$ with blocks corresponding to $\alpha_M = [\CC^2_G(M)]$ for $M \in \mc{M}(G)$.

Given $M \in \mc{M}(G)$, and given a $G$-invariant subalgebra $\mc{A}$ of $\LC(G)$ containing $\alpha_M$ such that $M$ acts faithfully on $\mc{A}$, then $\alpha_M$ corresponds to a clopen $G$-invariant subspace of $\mf{S}(\mc{A})$, and the dynamics of $G$ on this subspace are as described in Proposition~\ref{cor:rf_compressible:mono}, via a similar proof.  Alternatively, given $M \in \mc{M}(G)$ that acts faithfully on $\LC(G)$, the kernel of the action of $G$ on the subspace $X$ of $\LC(G)$ corresponding to $\alpha_M$ is exactly $\CC_G(M)$, so one can pass to the monolithic quotient $G/\CC_G(M)$ of $G$.  The action of $G/\CC_G(M)$ on $X$ is faithful weakly decomposable in the sense of \cite[Theorem~II]{CRW1}, so $G/\CC_G(M)$ is [A]-semisimple.  In addition, $G/\CC_G(M)$ is regionally expansive by Proposition~\ref{prop:normal-compressions}.  Thus Proposition~\ref{cor:rf_compressible:mono} applies directly to the quotient $G/\CC_G(M)$ of $G$.
\end{rmk}

We conclude by invoking the results from  \cite[Section 6]{CRW2} to derive the following for robustly monolithic groups, which was established in loc.\ cit.\ for groups in the class $\Ss$. 

\begin{thm}\label{thm:StronglyProx}
For $G \in \Rs$, the $\Mon(G)$-action on $\mf{S}(\LC(G))$, hence also the $G$-action, is minimal, strongly proximal, and has a compressible open set. In particular, if $G$ is amenable, then $\LC(G) = \{0, \infty\}$. 
\end{thm}

\begin{proof}
By Proposition~\ref{prop:FaithfulLC}, there is a $\Mon(G)$-equivariant injective homomorphism of Boolean algebras from $\LC(G)$ to $\LC(\Mon(G))$.  Thus $\mf{S}(\LC(G))$ occurs as a quotient of $\mf{S}(\LC(\Mon(G))$, so to show that the $\Mon(G)$-action is minimal, strongly proximal, or has a compressible open set on $\mf{S}(\LC(G))$, it suffices to show that the $\Mon(G)$-action on $\mf{S}(\LC(\Mon(G))$ has the corresponding property. The monolith $\Mon(G)$ is an element of $\Rs$ by Lemma~\ref{lem:monolith_overgroup}, and it is topologically simple. Furthermore, if $G$ is amenable, then so is $\Mon(G)$.  Replacing $G$ with $\Mon(G)$, we may thus assume that $G$ is topologically simple.  We may also assume $|\LC(G)| > 2$, as all the conclusions are trivial in the case that $\LC(G) = \{0, \infty\}$.

Suppose for a contradiction that $\mf{p} \in \mf{S}(\LC(G))$ is fixed by $G$.  Let $H$ be the subgroup of $G$ consisting of those elements which fix pointwise some neighborhood of $\mf{p}$.  The group $H$ is normalized by $G$ since $G$ fixes $\mf{p}$, and $H$ is non-trivial, as $G$ is micro-supported on $\mf{S}(\LC(G))$.  Thus $H$ is dense in $G$.

By Theorem~\ref{thm:approx_R}, there exists a compactly generated open subgroup $O \leq G$ such that $O \in \Rs$.  Let $\alpha \in \LC(G)$ represent a non-empty clopen set of $\mf{S}(\LC(G))$ not containing $\mf{p}$, let $U$ be a compact open subgroup of $O$ normalizing $\rist_O(\alpha)$, and set $L := \rist_U(\alpha)$; note that $L$ is non-trivial.  Since $O$ is compactly generated and $H$ is dense in $G$, Lemma~\ref{lem:Cayley-Abels} ensures the existence of elements $g_1, \dots, g_n \in H \cap O$ such that $O = \langle g_1, \dots, g_n\rangle U$. It follows that $ \langle g_1, \dots, g_n\rangle $ is transitive on the $O$-conjugacy class of $L$. The group $\langle \{g_1, \dots, g_n\} \cup L \rangle$ thus contains the abstract normal closure $\langle \langle L \rangle \rangle_O$ of $L$ in $O$. On the other hand, there exists a clopen neighborhood $\beta$ of $\mf{p}$ which is pointwise fixed by $ \{g_1, \dots, g_n\} \cup L$. Therefore, $\ngrp{L}_O$ fixes $\beta$ pointwise, so the monolith $\Mon(O)$ also fixes $\beta$ pointwise.  However, by Proposition~\ref{cor:rf_compressible:mono}, the set of fixed points of $\Mon(O)$ on $\mf{S}(\LC(G))$ has empty interior, a contradiction.

The group $G$ thus has no fixed points on $\mf{S}(\LC(G))$.  Applying Proposition~\ref{cor:rf_compressible:mono} (recalling that $G = \Mon(G)$), $G$ acts minimally with a non-empty compressible open set.  The action is strongly proximal by \cite[Proposition~6.24]{CRW2}.  If $\LC(G) \neq \{0,\infty\}$, then $G$ is non-amenable by \cite[Proposition~6.25]{CRW2}.
\end{proof}

Similar to \cite{CRW2}, we obtain the following consequence regarding abstract simplicity. 

\begin{cor}\label{cor:AbstractSimplicity}
Let $G\in \Rs$ be topologically simple.  If $G$ has an open subgroup of the form $K \times L$ such that $K$ and $L$ are non-trivial closed subgroups, then $G$ is abstractly simple.
\end{cor}

\begin{proof}
The hypotheses ensure that $G$ has a non-trivial decomposition lattice $\mc{A}$. The action of $G$ on $\mc{A}$ is faithful by Proposition~\ref{prop:LC_faithful}. The action is minimal by Theorem~\ref{thm:StronglyProx}, and by the same theorem, there exists $\alpha \in \mc{A}$ such that $g\alpha < \alpha$.  By \cite[Proposition~6.29]{CRW2}, it follows that the group $G^\dagger:= \langle \overline{\con(g)} \mid g \in G\rangle$ is open in $G$. Applying \cite[Corollary~6.28]{CRW2}, we deduce that $G$ is abstractly simple.
\end{proof}

\section{The local prime content}

\subsection{Preliminaries}
Let $\mathbb{P}$ be the set of primes.
\begin{defn}
The \emph{local prime content} $\eta(G)$ of a \tdlc group $G$ is the subset of $\mathbb{P}$ such that $p \in \eta(G)$ if $G$ contains an infinite pro-$p$ subgroup.  
Given a set $\pi \subset \mathbb P$, we say that $G$ is \emph{locally pro-$\pi$} if $G$ has an open pro-$\pi$ subgroup. 
\end{defn}

The local prime content is a local invariant of $G$.  If $G$ contains an infinite pro-$p$ subgroup, then so does every compact open subgroup of $G$. 
For any set of primes $\pi$, if $G$ is locally pro-$\pi$, then $\eta(G) \subseteq \pi$. A \tdlc group $G$ is discrete if and only if it is locally pro-$\varnothing$. However,  a non-discrete \tdlc group $G$ can have an empty local prime content. An example is given by the pro-cyclic group $\prod_{p \in \mathbb P} C_p$. That group is not locally pro-$\pi$ for any finite set of primes $\pi$.

 
\begin{prop}\label{prop:RF_pro-pi}
For every regionally expansive \tdlc group $G$, the local   prime content $\eta(G)$ is finite, and $G$ is locally pro-$\eta(G)$. In particular, for each $p\in \eta(G)$ and compact open subgroup $U \leq G$, the group $U$ has an infinite pro-$p$ subgroup, and thus $G_{(p)}$ is non-discrete.
 \end{prop}
 \begin{proof}
  Suppose that $G$ is a regionally expansive \tdlc group. We can find $O\leq G$ a compactly generated open subgroup such that $O$ admits a compact open subgroup $W\leq O$ which has a trivial normal core in $O$. Applying Proposition~\ref{prop:locallypro-pi}, we conclude that $W$ is pro-$\eta$ for some finite set of primes $\eta$. Therefore, $G$ is locally pro-$\eta$. 
 \end{proof}
 
\subsection{The local prime content of locally normal subgroups}

\begin{lem}\label{lem:O^p}
	Let $G$ be a profinite group, $R \leq G$ be a closed normal subgroup, and $S \leq G$ be a pro-$p$-Sylow subgroup for some prime $p$. Then the closure of $\Comm_{R}(S)$  in $G$ contains $O^p(R)$. 
\end{lem}

\begin{proof}
		Consider the closed subgroup $H := RS$ of $G$. The subgroup $S$ is a pro-$p$-Sylow subgroup of $H$, and Theorem~\ref{thm:Reid_localizations} ensures that $\Comm_H(S)$ is dense in $H$. We view  $H_{(p)}:=\Comm_H(S)$ as the $p$-localization of $H$ and denote by $\varphi \colon H_{(p)} \to H$ the continuous dense embedding given by the inclusion map. Set $N: = \varphi^{-1}(R)$. The group $N$ is a closed normal subgroup of $H_{(p)}$. Since $S$ commensurates itself, we have $S \leq H_{(p)}$. The kernel of the continuous map $H_{(p)} \to H \to H/R$ is $N$, and the restriction of that map  to the subgroup $S \leq H_{(p)}$ is surjective. We therefore have  $H_{(p)} = NS$. Furthermore, $H = \overline{\varphi(H_{(p)})} = \overline{\varphi(N)}S$, so $O^p(H) \leq  \overline{\varphi(N)}$. Hence,
		$$O^p(R) \leq O^p(H) \leq \overline{\varphi(N)} \leq R.$$
By definition, $N = \Comm_R(S)$, and the required assertion follows.   
\end{proof}

\begin{lem}\label{lem:O_pprime}
	Let $G$ be a profinite group, $R \leq G$ be a closed normal subgroup, and $S \leq G$ be a pro-$p$-Sylow subgroup for some prime $p$. If $R \cap S = \{1\}$, then $\QC_R(S)$ is dense in $R$. 
\end{lem}

\begin{proof}
	Consider the closed subgroup $H := RS \simeq R \rtimes S$ of $G$. Since $R \cap S = \{1\}$, we have $R = O^p(R)$. Therefore, Lemma~\ref{lem:O^p} ensures that $\Comm_R(S)$ is dense in $R$. 
	
	We infer that $\Comm_R(S)$ is normal in $\Comm_G(S) = G_{(p)}$. Since $R \cap S = \{1\}$, it follows that  $\Comm_R(S)$ is in fact a discrete normal subgroup of the localization $G_{(p)}$. Hence $\Comm_R(S) \leq \QZ(G_{(p)})$. This implies that $\Comm_R(S) = \QC_R(S)$, and the desired result follows. 
\end{proof}

\begin{prop}\label{prop:Primes-in-locally-normal}
	Let $G \in \mathscr R$, $U \leq G$ be a compact open subgroup, and $S \leq U$ be an infinite pro-$p$-Sylow subgroup of $U$ for some prime $p$. For every non-trivial locally normal subgroup $K$ of $G$, the intersection $K \cap S$ is infinite. 
\end{prop}
\begin{proof}
	Let $K$ be a locally normal subgroup of $G$ with $K \cap S$ finite; without loss of generality $K$ is compact. There is a compact open subgroup $V$ and a normal subgroup $L \normal V$ commensurate with $K$ such that $L \cap S = \triv$. Let $T$ be a pro-$p$-Sylow subgroup of $V$ containing $S \cap V$. The intersection $L \cap T$ is finite, so upon replacing $L$ by its intersection with a sufficiently small open normal subgroup of $V$, we may assume that $L \cap T$ is trivial. 
	
	Lemma~\ref{lem:O_pprime} ensures that $\QC_L(T)$ is dense in $L$. On the other hand, the group $\QC_L(T)$ is a subgroup of the localization $G_{(p)} = \Comm_G(T)$ which is contained in $\QZ(G_{(p)})$. By Theorem~\ref{thm:R_dense_embedding}, the quasi-center $\QZ(G_{(p)})$ is trivial. We infer that $\QC_L(T)$ is trivial, and so $L$ is trivial. The group $K$ is thus finite. Any finite locally normal subgroup is contained in the quasi-center. Invoking again Theorem~\ref{thm:R_dense_embedding},  we deduce that $K$ is trivial. Every non-trivial locally normal subgroup therefore has an infinite intersection with any infinite pro-$p$-Sylow subgroup.
\end{proof}

\begin{cor}\label{cor:LNpro-p->pro-p}
	Let $G \in \Rs$   and $\eta$ be any set of primes. 
	Then  $G$ has a non-trivial locally normal virtually pro-$\eta$ subgroup if and only if $G$ is locally pro-$\eta$.
\end{cor}
\begin{proof}
	Suppose that there exists a non-trivial locally normal virtually pro-$\eta$ subgroup $L$ of $G$ and fix $U$ a compact open subgroup of $G$. The group $L$ is infinite, via Theorem~\ref{thm:R_dense_embedding}.  Proposition~\ref{prop:Primes-in-locally-normal}  ensures any infinite pro-$p$-Sylow subgroup of $U$ has infinite intersection with $L$. The local prime content of $U$ is thus contained in $\eta$, hence  $U$ is virtually pro-$\eta$. We conclude that $G$ is locally pro-$\eta$.  The converse is trivial.
\end{proof}

The following consequence implies that the $p$-localization can only have topologically finitely generated compact open subgroups if the ambient group was already locally pro-$p$. 

\begin{cor}\label{cor:FinitelyGen_Localization}
	Let $G \in \Rs$ and $p$ be a prime. Then the following are equivalent.
	\begin{enumerate}[label=(\roman*)]
		\item Every compact open subgroup of $G$ is topologically finitely generated and virtually pro-$p$.
		\item Some compact open subgroup of $G$ has a topologically finitely generated infinite pro-$p$-Sylow subgroup.\qedhere
	\end{enumerate}
	
	In particular, if $G_{(p)}$ is has topologically finitely generated compact open subgroups, then either $G_{(p)}$ is discrete or $G_{(p)}=G$. 
\end{cor}
\begin{proof}
	The implication $(i)\Rightarrow (ii)$ is clear. Conversely, let $U \leq G$ be a compact open subgroup and $S\leq U$ be an infinite topologically finitely generated pro-$p$-Sylow subgroup of $U$. Let $p'$ denote the set of all primes different from $p$.  Then $G$ is not locally pro-$p'$, so by Corollary~\ref{cor:LNpro-p->pro-p}, none of the non-trivial locally normal subgroups of $G$ are pro-$p'$.  In particular, the $p'$-core $O_{p'}(U)$ of $U$ is trivial, so applying Corollary~\ref{cor:Tate}, $U$ is virtually pro-$p$. We conclude that $G$ is locally pro-$p$. The group $S$ is thus open, and assertion (i) follows. 
\end{proof}

We also record the following consequence for future references. 

\begin{lem}\label{lem:CentralizerO^p}
		Let $G \in \Rs$ and $p$ be a prime. If  $\eta(G)  \neq \{p\}$, then for every compact locally normal subgroup $L \leq G$, we have $\mathrm C_G(L) = \mathrm C_G(O^p(L))$. 
\end{lem}
\begin{proof}
	Let $U \leq G$ be a compact open subgroup of $G$ containing $L$ as a normal subgroup. We know that $G$ is [A]-semisimple by Proposition~\ref{prop:Rs_is_RF_[a]}, so by Lemma~\ref{lem:[A]-semisimple:commuting}, it suffices to show that $\mathrm C_U(L) = \mathrm C_U(O^p(L))$. Setting $K :=  \mathrm C_U(O^p(L))$, we have $K \cap O^p(L) = \triv$ since $G$ is [A]-semisimple. Therefore, $K \cap L$ maps continuously and injectively into the pro-$p$ group $L/O^p(L)$, so $K \cap L \leq O_p(L)$. In view of Corollary~\ref{cor:LNpro-p->pro-p}, the hypothesis that  $\eta(G)  \neq \{p\}$ ensures that $O_p(L)$ is trivial. Hence, $K \cap L = \triv$, and $K \leq \mathrm C_U(L)$ by Lemma~\ref{lem:[A]-semisimple:commuting}. This shows that $\mathrm C_U(O^p(L)) \leq \mathrm C_U(L)$. The reverse inclusion is obvious. 
\end{proof}

\subsection{The local prime content of groups of automorphisms}

We can upgrade Proposition~\ref{prop:RF_pro-pi} to \tdlc groups of automorphisms of $G$, provided $G$ has a trivial quasi-center.

\begin{prop}\label{prop:prime_content_RF}
Let  $G$ be a regionally expansive \tdlc group with a trivial quasi-center. If $H$ is a locally compact group and $\psi \colon H \to \Aut(G)$ is a continuous, injective homomorphism, then $H$ is a \tdlc group with finite local prime content $\eta(H)$, and $H$ is locally pro-$\eta(H)$.
\end{prop}

\begin{proof}
Proposition~\ref{prop:tdlc_action} ensures that $H$ is a \tdlc group.  Setting $L:=G\rtimes H$, we see by Proposition~\ref{prop:normal-compressions} that $L/\CC_L(G)$ is regionally expansive. Proposition~\ref{prop:RF_pro-pi} implies that $L/\CC_L(G)$ locally pro-$\eta_1$ for some finite set $\eta_1$. The group $G$ is also regionally expansive, so $G$ is locally pro-$\eta_2$ for some finite set $\eta_2$. Applying Lemma~\ref{lem:centralizer_action}, there is a continuous, injective homomorphism $\chi:\CC_L(G)\rightarrow G$, so $\CC_L(G)$ is locally pro-$\eta_2$. It now follows that $L$ is locally pro-$\eta$ where  $\eta:=\eta_1\cup\eta_2$. The group $H$ is thus locally pro-$\eta$, and thus has finite local prime content.
\end{proof}

We place stronger restriction on the local prime content for automorphisms of robustly monolithic groups.

\begin{thm}\label{thm:prime_content_RM}
Let $G \in \Rs$ with monolith $M$. If $H$ is a locally compact group and $\psi \colon H \to \Aut(G)$ is a continuous, injective homomorphism, then $H$ is a \tdlc group with $\eta(H) \subseteq \eta(M) = \eta(G)$. Moreover $H$ is locally pro-$\eta(H)$. 
\end{thm}

\begin{proof}
Set $\eta: = \eta(G)$; note that $\eta(M) = \eta$ as a consequence of Corollary~\ref{cor:LNpro-p->pro-p}.

By Proposition~\ref{prop:prime_content_RF}, $H$ is a locally pro-$\pi$ \tdlc group for some finite set of primes $\pi$.  To show that $\eta(H) \subseteq \eta$, it suffices to show $\eta(H') \subseteq \eta$ for some compact open subgroup $H'$ of $H$; we may thus assume that $H$ is a compact pro-$\pi$ group.

Suppose for a contradiction $H$ is not virtually pro-$\eta$.  For each $i \in \Nb$, there are then finite discrete quotients $H/K_i$ of $H$ such that the $\eta'$-order of $H/K_i$ (that is, the largest factor of $|H/K_i|$ that is coprime to all $p \in \eta$) tends to infinity as $i \rightarrow \infty$.  Since $H$ is compact, it has compact orbits on $G$, so $G$ can be expressed as a directed union of $H$-invariant compactly generated open subgroups of $G$.  As each $K_i$ has finite index in $H$, there is an $H$-invariant compactly generated open subgroup $O_i$ of $G$ such that the kernel $K'_i$ of the action of $H$ on $O_i$ satisfies $K'_i \le K_i$.  Possibly taking the $O_i$ larger,  we may assume that $O_i \in \Rs$ for all $i \in \Nb$ in view of Theorem~\ref{thm:approx_R} and that $(O_i)_{i\in I}$ is $\subseteq$-increasing.

Let $O := \bigcup_{i \in \Nb}O_i$, let $K: = \bigcap_{i \in \Nb}K'_i$, and write $\widetilde{H}: = H/K$.  By construction, $O$ is a $\sigma$-compact open subgroup of $G$, and $O \in \Rs$ by Theorem~\ref{thm:approx_R}. In view of Lemma~\ref{lem:FirstCountable}, $O$ is second countable.  The map $\psi$ induces a continuous injective homomorphism $\psi' \colon \widetilde{H} \to \Aut(O)$, and as $O$ is second countable, the topology of $\Aut(O)$ is second countable. The group $\widetilde{H}$ is compact, so $\psi'$ is a closed map, implying that $\widetilde{H}$ is also second countable.

Form the semidirect product $O \rtimes \widetilde{H}$, using the action of $H$ on $O$ given by $\psi'$, and identify $O$ and $\widetilde{H}$ with the subgroups $O \rtimes \triv$ and $\triv \rtimes \widetilde{H}$ respectively.  Let $N: = \CC_{O\rtimes \widetilde{H}}(\Mon(O))$ and let $L: = (O\rtimes \widetilde{H})/N$.  Corollary~\ref{cor:RM_extension} ensures that $L$ is robustly monolithic.  The subgroups $N$ and $O$ normalize each other and have trivial intersection, since $\CC_O(\Mon(O)) = \triv$, so $N$ and $O$ commute inside $O \rtimes \widetilde{H}$.

Note that $\eta(O) = \eta$ since $O$ is open in $G$.  Take $W \leq O$ a compact open pro-$\eta$ subgroup. By the definition of the Braconnier topology (see \cite[Definition~(26.3)]{HR}), the set-wise stabilizer of $W$ in $\Aut(O)$ is an open subgroup of $\Aut(O)$, ensuring that $W$ is locally normal in $O \rtimes \widetilde{H}$.  The corresponding subgroup $WN/N$ of $L$ is then a non-trivial locally normal pro-$\eta$ subgroup of $L$.  It follows by Corollary~\ref{cor:LNpro-p->pro-p} that $L$ is locally pro-$\eta$.  Since $\widetilde{H}N/N$ is a compact subgroup of $L$, it follows in turn that $\widetilde{H}N/N$ is virtually pro-$\eta$, so $\widetilde{H}/(\widetilde{H} \cap N)$ is virtually pro-$\eta$.  Finally, $N$ is locally pro-$\eta$ via Lemma~\ref{lem:centralizer_action}(v), so $\widetilde{H} \cap N$ is virtually pro-$\eta$.  We conclude that $\widetilde{H}$ is virtually pro-$\eta$.  On the other hand, $\widetilde{H}$ maps continuously onto the finite groups $H/K_i$, whose $\eta'$-order is unbounded as $i \rightarrow \infty$. This contradicts the fact that $\widetilde{H}$ is virtually pro-$\eta$.  From this contradiction, we conclude that $H$ is in fact virtually pro-$\eta$, so, $\eta(H) \subseteq \eta$, as required.
\end{proof}

\begin{cor}
If $G\in \Rs$ (e.g.\ $G\in \Ss$) is locally pro-$p$, then any locally compact group that continuously, faithfully acts on $G$ by topological group automorphisms is locally pro-$p$.
\end{cor}

The above theorem shows that the class of automorphism groups of robustly monolithic groups, and in particular of groups in $\Ss$, admits non-trivial restrictions. For example, $\Sym(\Nb)$ does not continuously embed into $\Aut(G)$ for any $G\in \Rs$. One can go further and formulate likely very difficult analogues of the Schreier Conjecture for the classes $\Ss$ and $\Rs$.
\begin{qu} Is every locally compact group $H$ that acts faithfully, continuously by outer automorphisms on some $G\in \Ss$ regionally elementary?
\end{qu}
\begin{qu} 
Is every $G\in \Rs$ such that $G/\Mon(G)$ is regionally elementary?
\end{qu}

\subsection{The centralizer lattice of a $p$-localization}

Proposition~\ref{prop:Rs_is_RF_[a]} ensures every member of $\Rs$ is [A]-semisimple. The class $\Rs$ is closed under taking dense locally compact subgroups, so the $p$-localization $G_{(p)}$ of any $G\in \Rs$ is an element of $\Rs$ as soon it is non-discrete. In view of \cite{CRW1}, we deduce that the centralizer lattice of non-discrete $p$-localizations $G_{(p)}$ is well defined. We now obtain an analogue of Proposition~\ref{prop:FaithfulLC} for $p$-localizations, giving in particular a canonical embedding of $\LC(G)$ into $\LC(G_{(p)})$. 

\begin{thm}\label{thm:EmbeddingLC}
Let $G \in \mathscr R$, $U \leq G$ be a compact open subgroup, and $S \leq U$ be an infinite pro-$p$-Sylow subgroup for some prime $p$. The map $\LN(G) \to \LN(G_{(p)})$ defined by
\[
 [L] \mapsto [L \cap S]=:[L]_{(p)} 
\]
enjoys the following properties, where $G_{(p)}$ is the $p$-localization $\Comm_G(S)$.
\begin{enumerate}[label=(\roman*)]
\item It is order-preserving, $G_{(p)}$-equivariant, and every non-zero element has a non-zero image;
 \item $(\alpha^\perp)_{(p)} = (\alpha_{(p)})^\perp$; 
 \item the restriction to the centralizer lattice yields an injective homomorphism  $\LC(G) \to \LC(G_{(p)})$ of Boolean lattices; and
\item if in addition $\LC(G) \neq \{0, \infty\}$, then $G$ acts faithfully on $\LC(G)$, and $G_{(p)}$ acts faithfully on $\LC(G_{(p)})$. \qedhere
\end{enumerate}
\end{thm}

\begin{proof}
We may assume that $G$ is not locally pro-$p$, since otherwise $G= G_{(p)}$ and the required assertions are trivially satisfied. Every non-zero class $[L]\in \LN(G)$ has non-zero image by Proposition~\ref{prop:Primes-in-locally-normal}. Let $L \leq G$ be a compact locally normal subgroup. We then have $\mathrm C_G(L) = \mathrm C_G(O^p(L))$ by Lemma~\ref{lem:CentralizerO^p}. Let $W$ be a compact open subgroup of $G$ containing $L$ as a normal subgroup and $T$ be a pro-$p$ Sylow subgroup of $W$ containing $S \cap W$. The groups $S$ and $T$  are commensurate, so they have the same commensurator in $G$. Moreover, the closure of $\Comm_L(S) = \Comm_L(T) $ contains $O^p(L)$ by Lemma~\ref{lem:O^p}. Therefore, 
$$\mathrm C_{G}(L \cap G_{(p)}) = \mathrm C_{G}(\Comm_L(S)) = \mathrm C_{G}(\overline{\Comm_L(S)}) \leq \mathrm C_{G}(O^p(L))= \mathrm C_G(L).$$ 
In particular, $\mathrm C_{S}(L \cap G_{(p)})  =  \mathrm C_S(L)$. All the hypotheses of Lemma~\ref{lem:PullBack} are now verified, and the remaining claims of (i) and (ii) and also claim (iii)  follow. 

For (iv), Proposition~\ref{prop:LC_faithful} ensures that if $\LC(G) \neq \{0, \infty\}$, then the $G$-action on $\LC(G)$ is faithful. Since the map $\LC(G) \to \LC(G_{(p)})$ by $ \alpha \mapsto \alpha_{(p)}$ is injective and $G_{(p)}$-equivariant, we deduce that the $G_{(p)}$-action on its centralizer lattice is faithful. 
\end{proof}

\begin{rmk}
	The map  $\alpha \mapsto \alpha_{(p)}$ is usually not injective on the whole structure lattice $\LN(G)$. For instance, if $G = \Aut(T)^+$ is the group of type-preserving automorphisms of the regular tree of degree~$d \geq 4$ and $p$ is an odd prime such that $p < d$, then the abelianization of an edge-stabilizer $U \leq G$ is an infinite elementary abelian $2$-group, and the pro-$p$-Sylow subgroups of $U$ are infinite. The group $U$ and its closed derived group $\ol{[U, U]}$ are not commensurate, but they have the same image in $\LN(G_{(p)})$. 
\end{rmk}

\section{Examples}\label{sec:examples1}

\subsection{Germs of automorphisms}\label{sec:germs}

In the context of \tdlc groups, the local structure of the group is manifested in its compact open subgroups, allowing for a useful notion of local homomorphisms.  We borrow the terminology of \cite{CapDeM}.

\begin{defn}
A \emph{local homomorphism} between two \tdlc groups $G$ and $H$ is a continuous homomorphism $\phi: U \rightarrow H$, where $U$ is an open subgroup of $G$.  It is a \emph{local isomorphism} if $\phi$ restricts to an isomorphism from $U$ to an open subgroup of $H$.  We say $G$ and $H$ are \emph{locally isomorphic} if a local isomorphism exists.  Two local homomorphisms $\phi_1,\phi_2$ are \emph{equivalent} if they agree on some open subgroup $W$ of $G$ that is contained in the domain of both $\phi_1$ and $\phi_2$.  The equivalence class $[\phi]$ of a local homomorphism is then called the \emph{germ} of $\phi$. 
\end{defn}

In the case of groups with trivial quasi-center, Barnea, Ershov and Weigel \cite{BEW} show that there is always a unique largest group in a given local isomorphism class.

\begin{thm}[{see \cite{BEW} and \cite{CapDeM}}]\label{thm:AbstractCommensurator}
Let $G$ be a \tdlc group with trivial quasi-center.  Define the group $\ms{L}(G)$ of {germs of automorphisms} of $G$ to be the set of germs of local isomorphisms from $G$ to itself.
\begin{enumerate}[label=(\roman*)]
\item $\ms{L}(G)$ has a group structure induced by composition of local isomorphisms.

\item Let $\ad: G \rightarrow \ms{L}(G)$ be defined by $\ad(g) = [y \mapsto gyg\inv]$.  Then $\mathrm{ad}$ is an injective group homomorphism, and there is a unique group topology on $\ms{L}(G)$ such that $\ad$ is continuous and open.

\item If $\phi_1: U \rightarrow G$ and $\phi_2: U \rightarrow H$ are continuous, open, and injective homomorphisms where $U$ and $H$ are \tdlc groups, then there is a unique continuous and open homomorphism $\theta: H \rightarrow \ms{L}(G)$ with kernel $\QZ(H)$ such that the following diagram commutes:
\[
\xymatrixcolsep{3pc}\xymatrix{
U \ar^{\phi_1}[r] \ar^{\phi_2}[d] & G \ar^{\ad}[d] \\
H \ar^{\theta}[r] & \ms{L}(G). }
\]
\end{enumerate}
\end{thm}

As noted in \cite{BEW} and \cite{CapDeM}, the above theorem leads to a largest \emph{topologically simple} group in a given local isomorphism class. 

\begin{cor}\label{cor:LargestSimple}
Let $G$ be a topologically simple \tdlc group with $\QZ(G)=\triv$. Then there is a topologically simple \tdlc group $\widetilde G$ with $\QZ(\widetilde G)=\triv$, unique up to isomorphism, such that the following hold.
\begin{enumerate}[label=(\roman*)]
\item $G$ embeds as an open subgroup of $\widetilde{G}$;
\item For any topologically simple \tdlc group $H$ locally isomorphic to $G$, there is an open embedding $H\rightarrow \widetilde{G}$.
\end{enumerate}
Furthermore, if $G$ is regionally expansive, then $G$, $\widetilde{G}$, and $\ms{L}(G)$ are all robustly monolithic.
\end{cor}

If $G \in \Rs$, then also $\widetilde{G} \in \Rs$.  However, in general, even if we start with $G \in \Ss$, the corresponding universal simple group $\widetilde{G}$ for the local isomorphism class need not be compactly generated or even $\sigma$-compact.  The following is an illustration of this situation.

\begin{example}\label{ex:Smith}
In \cite{Smith}, S.M.~Smith constructs a family $(G_i)_{i \in I}$ of $2^{\aleph_0}$ pairwise non-isomorphic groups in $\Ss$ that are all abstractly simple and locally isomorphic to one another. Let $\widetilde G$ be the topologically simple group in $\mathscr R$ afforded by applying Corollary~\ref{cor:LargestSimple} to some (any) of the $G_i$. Since a $\sigma$-compact \tdlc group has only countably many compactly generated open subgroups, we deduce that $\widetilde G$ is not $\sigma$-compact.  Given that each $G_i$ is abstractly simple, it is clear that $\widetilde G$ has no proper dense normal subgroup, so it is abstractly simple.  The group $\widetilde G$ is therefore a simple group contained in $\Rs$ but not in $\Ss$. 

We also note that $\widetilde G$  contains open subgroups belonging to $\Rs$ that are topologically simple, $\sigma$-compact and not compactly generated. In order to see this, we build a countable  ascending chain of open simple subgroups $(S_n)_{n\in \Nb}$ of $\widetilde G$ as follows. Pick any $i \in I$ and set $S_0 := G_i$. For $n>0$, since $S_n$ has countably many compactly generated open subgroups, there exists $j \in I$ such that $G_j$ is not contained in $S_n$. Set $S_{n+1} := \la S_n \cup G_j \ra$. The group $S_{n+1}$ is then a compactly generated open subgroup of $\widetilde G$ containing $S_n$ as a proper subgroup.  Since $\QZ(\widetilde{G}) = \triv$, the open subgroup $S_{n+1}$ also has trivial quasi-center and hence has no non-trivial discrete normal subgroups.  It then follows that every non-trivial normal subgroup $N$ of $S_{n+1}$ intersects the open subgroup $G_i$ non-trivially, hence contains $G_i$, for all $i \in I$ such that $G_i \le S_{n+1}$. Since $S_{n+1}$ is generated by subgroups $G_i$, we infer that $S_{n+1}$ is abstractly simple.  We now see that $\bigcup_n S_n$ is an open, $\sigma$-compact abstractly simple subgroup of $\widetilde G$ that is not compactly generated. Additionally, $\bigcup_n S_n$ belongs to the class $\Rs$ since it has trivial quasi-center and each $S_n$ is compactly generated and expansive.
\end{example}

\subsection{Restricted Burger-Mozes groups}\label{sec:U(F)}

Following  \cite{Serre80}, a \textit{graph} $\Gamma = (V,E,o,r)$ consists of a vertex set $V = V\Gamma$, a directed edge set $E = E\Gamma$, a map $o:E \rightarrow V$ assigning to each edge an \textit{initial vertex}, and a bijection $r: E \rightarrow E$, denoted by $e \mapsto \ol{e}$ and called \textit{edge reversal} such that $r^2 = \mathrm{id}$. A \textit{tree} is a connected graph without cycles. A tree is \textit{$d$-regular} if for each vertex $v$ there are $d$ many distinct edges $e$ such that $o(e)=v$.

Let $T$ be the $d$-regular tree for $d>2$ (at this point, we allow $d$ to be an infinite cardinal). A \textit{coloring} of $T$ is a map $c \colon ET\rightarrow X$ such that 
\[
c_v:=c\rest_{E(v)}:E(v)\rightarrow X
\]
is a bijection for every $v\in VT$ where $E(v)$ is the collection of edges $e$ with $o(e)=v$.  For $g\in \Aut(T)$ and $c$ a coloring of $T$, the \textit{local action} of $g$ at $v$ is defined to be 
\[
\sigma_c(g,v):=c_{g(v)}\circ g\circ c^{-1}_v\in \Sym(X).
\]
\begin{defn}
Let $T$ be the $d>2$ regular tree. For a permutation group $F\leq\Sym(d)$ and $c \colon ET\rightarrow [d]$ a coloring, the \emph{Burger--Mozes group} with local action prescribed by $F$ via $c$ is
\[
U_c(F):=\{g\in \Aut(T)\mid \forall v\in VT\;\sigma_c(g,v)\in F\}.
\]
\end{defn}

From now until the end of this subsection, we assume that $d$ is finite. The Burger--Mozes groups $U_c(F)$ are always closed subgroups of $\Aut(T)$. The subspace topology therefore induces a \tdlc group topology on $U_c(F)$. We always take $U_c(F)$ to be equipped with this topology.

Let us pause to recall some basic facts about groups acting on locally finite trees.

\begin{lem}[{\cite[Lemma~2.4]{CapDeM}}]\label{lem:cocompact_tree}
Let $T$ be a locally finite tree, let $G$ be a closed subgroup of $\Aut(T)$, and let $X \subseteq T$ be a minimal $G$-invariant subtree.  Then $G$ is compactly generated if and only if $G$ acts on $X$ with finitely many orbits.
\end{lem}

\begin{lem}[{\cite[Corollaire 3.5]{Ti70}}]\label{lem:min_inv_stree}
	Suppose that $G$ is a compactly generated \tdlc group acting continuously on a locally finite tree $T$ and suppose that $G$ does not fix a vertex or end. Then there is a unique minimal $G$-invariant subtree $X\subseteq T$.  
\end{lem}

There is a well-behaved family of colorings, under which to consider the groups $U_c(F)$.
\begin{defn}
We say that a coloring $c$ of a $d$-regular tree $T$ is \emph{legal} if $c(e)=c(\overline{e})$ for each edge $e\in ET$, where $\overline{e}$ is the reverse edge.
\end{defn}

One can obtain a legally colored $d$-regular tree $T$. For example, take the group
\[
\Gamma: = \grp{x_1} \ast \grp{x_2} \ast \dots \ast \grp{x_d},\; \text{ where } |\grp{x_i}|=2 \text{ for } 1 \le i \le d,
\]
let $T$ be the Cayley graph of $\Gamma$ with respect to $\{x_1,\dots,x_d\}$, and set $c(g,gx_i) = i$ for all $g \in \Gamma$.

\begin{prop}[{Burger--Mozes, \cite[Section 3.2]{BuMo}}] 
Let  $d>2$, $T$ be the $d$-regular tree, and $F\leq \Sym(d)$. If $c$ and $k$ are legal colorings of $T$, then $U_c(F)\simeq U_k(F)$.
\end{prop}
When $c$ is a legal coloring, the groups $U_c(F)$ have several important features.

\begin{prop}[See {\cite[Proposition~3.2.1]{BuMo}}]\label{prop:SimpleBurgerMozes}
Let $T$ be the $d >2$ regular tree, $c$ a legal coloring of $T$, and $F \leq \Sym(d)$ a permutation group which does not act freely on $[d]$. Then, the group $U_c(F)$ is compactly generated and monolithic, and its monolith is abstractly simple and coincides with the subgroup $U_c(F)^+$, which is generated by the pointwise stabilizers of edges. Additionally, the following conditions are equivalent.

\begin{enumerate}[label=(\roman*)]
	\item $U_c(F)$ is virtually in $\Ss$. 
	\item $U_c(F)^+$ is in $\Ss$.
	\item $[U_c(F): U_c(F)^+]$ is finite. 
	\item  $[U_c(F): U_c(F)^+] = 2$. 
	\item $F$ is transitive and generated by its point stabilizers.\qedhere
\end{enumerate}	
\end{prop}

\begin{proof}
By \cite[Proposition~3.2.1]{BuMo}, $U_c(F)^+$ is abstractly simple and is the monolith of $U_c(F)$, and (iii)-(v) are equivalent.  To show that (i), (ii) and (iii) are equivalent, it is enough to show that $U_c(F)^+$ is compactly generated if and only if $[U_c(F): U_c(F)^+]$ is finite.

The group $U_c(F)$ acts transitively on the vertices of $T$, so it is compactly generated by Lemma~\ref{lem:cocompact_tree}. Thus, if $U_c(F)^+$ has finite index, then it is compactly generated.  Conversely, if $U_c(F)^+$ is compactly generated, then the unique minimal $U_c(F)^+$-invariant subtree, as provided by Lemma~\ref{lem:min_inv_stree}, is also $U_c(F)$-invariant, so this subtree must in fact be $T$ itself. By Lemma~\ref{lem:cocompact_tree}, $U_c(F)^+$ has finitely many orbits on the vertices of $T$.  Since $U_c(F)^+$ contains a finite index subgroup of a vertex stabilizer in $U_c(F)$, it follows that $[U_c(F): U_c(F)^+]$ is finite.
\end{proof}

There is an important generalization of the Burger--Mozes groups due to Le Boudec \cite{LeBou16}, which we shall see gives examples of dense locally compact subgroups. 

\begin{defn}
Say that $d>2$ and $T$ is the $d$-regular tree. For a permutation groups $F\leq F'\leq\Sym(d)$ and $c \colon ET\rightarrow [d]$ a coloring, the \emph{restricted Burger--Mozes group} with local action prescribed by $F$ and $F'$ via $c$ is
\[
G_c(F,F'):=\{g\in \Aut(T)\mid \forall v\in VT\;\sigma_c(g,v)\in F'\text{ and }\forall^\infty v\in VT \;\sigma_c(g,v)\in F\},
\]
where ``$\forall^{\infty}$" is read ``for all but finitely many."
\end{defn}

\begin{rmk}
The name for the groups $G_c(F,F')$ recalls the analogy with restricted direct products. That is to say, a restricted Burger--Mozes group is to a Burger--Mozes group as a restricted direct product is to a direct product. 
\end{rmk}
The group $G_c(F,F')$ is always a subgroup of $U_c(F')$, and it contains $U_c(F)$. There is a natural \tdlc group topology on $G_c(F,F')$  under which $U_c(F)$ continuously embeds as an open subgroup. This is proved in \cite{LeBou16} for $c$ a legal coloring, but the same proof works for illegal colorings. 

\begin{rmk}
In \cite{LeBou16}, the author colors the undirected edges of trees, which is equivalent to assuming that all colorings are legal. It turns out that the restricted Burger-Mozes groups for \textit{illegal} colorings give a large family of interesting examples, well beyond those considered in \cite{LeBou16}. In a forthcoming article, we will explore these groups and their properties. In particular, we will use these groups to describe $p$-localizations of many Burger-Mozes groups.
\end{rmk}

Given a partition $\Omega$ of $[d]$, the \textit{Young group} associated to $\Omega$ is the set of $\sigma\in \Sym(d)$ such that $\sigma$ set-wise fixes each part of $\Omega$. If $\Omega$ is given by the orbits of some $D\leq \Sym(d)$, we denote the Young group by $\wh{D}$. Note that $D\leq \wh{D}$.

\begin{prop}[Le Boudec, {\cite[Proposition 3.5, Corollary 3.8]{LeBou16}}]\label{prop:LeBoudec}
Let $d>2$, $T$ be the $d$-regular tree, and $c$ be a legal coloring of $T$. If $F\leq F'\leq \wh{F}$ are subgroups of $\Sym(d)$, then $G_c(F,F')$ is a dense subgroup of $U_c(F')$, and $G_c(F,F')$ is compactly generated.
\end{prop}

In view of Proposition~\ref{prop:SimpleBurgerMozes}, $G_c(F,F')$ is in $\ms{R}$ whenever $F'$ is transitive and generated by its point stabilizers and $F$ does not act freely. One thus obtains many examples of dense locally compact subgroups of groups $G\in\Ss$ via the groups $G_c(F,F')$. Let us note a specific example.
\begin{example}
Let $c$ be a legal coloring of the $5$-regular tree. The type preserving subgroup of $G_c(\Alt(5),\Sym(5))$ is a dense locally compact subgroup of $\Aut^+(T_5)$. 
\end{example}

We now characterize when $G_c(F,F')$ is virtually simple. First, some preparatory lemmas.

\begin{lem}\label{lem:cocompact_orbits}
	Suppose that $G$ is a compactly generated \tdlc group which acts vertex transitively and continuously on a graph $\Gamma$ with regionally compact vertex stabilizers. If $N\normal G$ is closed and $N$ acts with finitely many orbits on $\Gamma$, then $N$ is cocompact in $G$.
\end{lem}
\begin{proof}
	Let $v$ be a vertex of $\Gamma$ and let $\alpha$ denote the $N$-orbit of $v$. Since $N$ is normal, the group $G$ permutes the $N$-orbits of vertices. By hypothesis, there are finitely many such, so that the setwise stabilizer  $G_{\alpha}$ of $\alpha$  is of finite index in $G$. The transitivity of $N$ on $\alpha$ implies that $G_{\alpha}=G_{(v)}N$, where $G_{(v)}$ is the stabilizer of the vertex $v$. Since $G_{(v)}$ is regionally compact, $G_{(v)}=\bigcup_{i\in I}U_i$ where $(U_i)_{i\in I}$ is a directed sequence of compact open subgroups. We thus deduce that $G_{(v)}N=\bigcup_{i\in I}U_iN$. On the other hand, $G_{(v)}N$ has finite index in $G$, so it is also compactly generated. There is thus some $i$ such that $U_iN=G_{(v)}N$. We conclude that $N$ is cocompact in $G$, verifying the lemma.
\end{proof}

Via the discussion in \cite[Section 3.1]{LeBou16}, the action of $G_c(F,F')$ on the tree $T$ is such that vertex stabilizers are regionally compact.  Since a closed subgroup of a regionally compact group is regionally compact, we note the following.

\begin{lem}\label{lem:locally_elliptic}
Let $d>2$, $T$ be the $d$-regular tree, $c$ be a legal coloring of $T$, and $F\leq F'\leq \wh{F}$ be subgroups of $\Sym(d)$.  If $H$ is a closed subgroup of $G_c(F,F')$ that fixes a vertex of $T$, then $H$ is regionally compact.
\end{lem}

Our characterization is now in hand.

\begin{prop}\label{prop:virt_simple_G(F,F')}
Take $d>2$, let $c$ be a legal coloring of the $d$-regular tree, and $F'\leq \Sym(d)$ be such that the action of $F'$ is not free. If $F\leq F'\leq \wh{F}$, then $G_c(F,F')$ is virtually simple if and only if $F'$ is transitive and generated by its point stabilizers.
\end{prop}
\begin{proof}
Suppose first that $G_c(F,F')$ has a simple normal subgroup $M$ of finite index. Since $F'$ does not act freely, Proposition~\ref{prop:SimpleBurgerMozes} ensures that $U_c(F')^+$ is a non-discrete simple open subgroup of $U_c(F')$, and this subgroup is the monolith of $U_c(F')$.   The group $G_c(F,F')$ is dense in $U_c(F')$, so $\overline{M}$ contains $U_c(F')^+$.  The intersection $N:=G_c(F,F')\cap U_c(F')^+$ is an infinite normal subgroup of $G_c(F,F')$, so it must intersect $M$ non-trivially. As $M$ is simple, in fact $M \le N$.  That $M$ has finite index in $G_c(F,F')$ now implies that $U_c(F')^+$ has finite index in $U_c(F')$. Via  Proposition~\ref{prop:SimpleBurgerMozes}, $F'$ is transitive and generated by its point stabilizers.

Conversely, suppose that $F'$ is transitive and generated by its point stabilizers. Via \cite[Corollary 4.6]{LeBou16}, every non-trivial normal subgroup of $G:=G_c(F,F')$ contains the commutator subgroup $[G_{(e)},G_{(e)}]$ for every edge $e$. Furthermore, \cite[Lemma 4.8]{LeBou16} ensures that these commutator subgroups are non-trivial. It follows that $G_c(F,F')$ is monolithic, and appealing to \cite[Corollary 4.9]{LeBou16},  $G_c(F,F')$ admits a simple monolith $M$. The closure $\overline{M}$ is a normal subgroup of $U_c(F')$ and therefore contains $U_c(F')^+$. In particular, $\overline{M}$ acts with finitely many orbits on $T$. Since vertex stabilizers are open, $M$ indeed acts with finitely many orbits on $T$.

By Lemma~\ref{lem:locally_elliptic}, the vertex stabilizers in $G_c(F,F')$ for the action on $T$ are regionally compact. Lemma~\ref{lem:cocompact_orbits} implies that $M$ is in fact cocompact in $G_c(F,F')$. Applying again \cite[Corollary 4.9]{LeBou16}, $M$ is open in $G_c(F,F')$, so $M$ has finite index in $G_c(F,F')$.
\end{proof}

The previous proposition applies even in the case that $F$ acts freely, making $G_c(F,F')$ discrete. We, however, are primarily interested in the non-discrete examples $G_c(F,F')$, and our next theorem gives further information on these groups. It should be compared with the corresponding property of the Burger--Mozes group $U_c(F)$ recalled in Proposition~\ref{prop:SimpleBurgerMozes}. 

\begin{thm}\label{thm:char_G(F,F')_Rs}
	Take $d>2$ and let $c$ be a legal coloring of the $d$-regular tree. Suppose that $F\leq F'\leq \wh{F}\leq \Sym(d)$ are such that $F$ does not act freely. Then the following are equivalent:
	\begin{enumerate}[label=(\roman*)]
	\item $F'$ is transitive and generated by its point stabilizers;
	\item $G_c(F,F')$ is virtually in $\Ss$;
	\item $G_c(F,F')$ is in $\Rs$.	\qedhere
	\end{enumerate}
\end{thm}

\begin{proof}
	That $(i)$ and $(ii)$ are equivalent follows from Proposition~\ref{prop:virt_simple_G(F,F')} and Proposition~\ref{prop:LeBoudec}.  Suppose that $(i)$ holds. Via Proposition~\ref{prop:SimpleBurgerMozes}, the group $U_c(F')$ is in $\Rs$. The group $G_c(F,F')$ is a dense locally compact subgroup of $U_c(F')$, so $G_c(F,F')$ is in $\Rs$ as soon as it is non-discrete, in view of Theorem~\ref{thm:R_dense_embedding}. The group $G_c(F,F')$ is non-discrete exactly when $F$ does not act freely on $[d]$, which we assume to hold. Hence, $G_c(F,F')\in \Rs$, and $(iii)$ holds.
	
To show that $(iii)$ implies $(ii)$, suppose that $G_c(F,F')\in \Rs$. Let $M\normal G_c(F,F')$ be the monolith; this subgroup is necessarily open since it contains $U_c(F)^+$.  Via Theorem~\ref{thm:approx_R}, there is $O\leq M$ such that $O $ is a compactly generated open subgroup of $M$ and $O\in \Rs$.  Since $O \in \Rs$, the monolith of $O$ is regionally expansive, and in particular, the monolith is not regionally compact and so cannot be contained in any regionally compact subgroup.   Lemma~\ref{lem:locally_elliptic} ensures that the monolith of $O$ fixes none of the vertices of $T$. Therefore, given any non-trivial normal subgroup $W$ of $O$, the group $W$ does not fix any vertex of $T$.

In the stabilizer of any end of $T$, the elliptic isometries form an open normal subgroup; thus if $O$ fixes an end, then the monolith of $O$ must consist of elliptic isometries. This implies the monolith is regionally compact, which contradicts the fact that the monolith is regionally expansive. We deduce that $O$ fixes no vertex or end. Lemma~\ref{lem:min_inv_stree} now supplies a minimal invariant subtree $X\subseteq T$ for $O$, and $O$ acts on $X$ with finitely many orbits by Lemma~\ref{lem:cocompact_tree}.
	
	If $X$ is not equal to $T$, then we can find an edge $e$ such that $X\subseteq T_i$ for one of the two half trees $T_0$ and $T_1$ determined by $e$; we assume without loss of generality that $X\subseteq T_1$. The group $U_c(F)$ has Tits' property $P$, so $U_c(F)_{(e)}=U_c(F)_{(T_0)}\times U_c(F)_{(T_1)}$, and thus, $U_c(F)_{(T_1)}\leq U_c(F)^+\leq M$. The group $U_c(F)_{(T_1)}$ is infinite and compact in $M$, so $W:=O\cap U_c(F)_{(T_1)}$ is infinite. On the other hand, $W$ fixes $X$ pointwise, so $W$ is in the kernel of the action of $O$ on $X$. This contradicts our earlier conclusion that no non-trivial normal subgroup of $O$ fixes any vertex of $T$. We conclude that $X=T$, so $O$ acts with finitely many orbits on $T$.
	
	The monolith $M$ of $G_c(F,F')$ thus acts with finitely many orbits on $T$. Applying Lemma~\ref{lem:cocompact_orbits}, we conclude that $M$ in fact has finite index in $G_c(F,F')$. Since $M$ is simple, we have established $(ii)$.
\end{proof}

\begin{rmk}
In \cite{LeBou16}, several sufficient conditions ensuring that $G_c(F,F')$ is virtually simple are found. Our result Proposition~\ref{prop:virt_simple_G(F,F')} subsumes all of these. We note that more precise information is obtained on the index of the simple subgroup in \cite{LeBou16}, for the special cases considered.
\end{rmk}

\subsection{Further examples of simple groups in $\Rs$}\label{sec:ConcreteEx}

The goal of this final section is to provide further motivation to consider the class $\Rs$ by constructing  examples of non-discrete dense locally compact subgroups of groups in $\Ss$. In Proposition~\ref{prop:ex-non-sigma-cpt}, we obtain an example of a non-discrete simple \tdlc group   in $\Rs$ that is not $\sigma$-compact, but embeds as dense locally compact subgroup of a group in $\Ss$. In Example~\ref{ex:Radu}, we obtain an example of a non-discrete dense locally compact subgroup $H$ of a group $G$ in $\Ss$ such that no compactly generated subgroup of $H$ is dense in $G$. 

Following \cite{CapDeM}, for a finite permutation group $F\leq \Sym(n)$, $W(F)$ denotes the profinite wreath branch group given by $F$, which is defined as the projective limit of iterated permutational wreath products of $F$ with itself. If $F$ is a subgroup of $J\leq \Sym(n)$, we may naturally view $W(F)$ as a closed subgroup of $W(J)$. We note that the commensurator $\Comm_{W(\Sym(n))}(W(F))$ is dense in $W(\Sym(n))$ for \textit{any} $F$, since the commensurator contains the union of the finitely iterated wreath products of copies of $\Sym(n)$. We also observe that $\Comm_{W(\Sym(n))}(W(F))$ is a branch group, and the rigid vertex stabilizers are isomorphic to $\Comm_{W(\Sym(n))}(W(F))$. The proofs of these facts are exercises in the definitions.

Let $T$ be the $5$-regular tree, let $c:ET\rightarrow [5]$ be a legal coloring, and fix an edge $e$. Set $H:=\Aut(T)$ and $L:=U_c(\Alt(5))$. The edge stabilizers $H_{(e)}$ and $L_{(e)}$ are isomorphic to $W(\Sym(4))^2$ and $W(\Alt(4))^2$, respectively, where the symbol $G^2$ denotes the direct product $G \times G$. Letting $V\normal \Alt(4)$ be the Klein four group, $U:=W(V)^2$ is the pro-$2$-Sylow subgroup of $ W(\Alt(4))^2$.  Via Theorem~\ref{thm:Reid_localizations}, $\Comm_L(U)$ is dense in $L$. In particular,  $\ol{\Comm_H(U)}$ is   non-compact since it contains $L$. On the other hand, $\ol{\Comm_H(U)}$ contains $W(\Sym(4))^2$, by the first paragraph. We deduce that $\ol{\Comm_H(U)}$ is  a non-compact open subgroup of  $H$. The only such  subgroups of $H$ are $H^+:=\Aut^+(T)$ and $H$ (this can for example be deduced from the fact that $H$ has the Howe--Moore property, see  \cite[Theorem~4.2]{BuMo2}; an alternative direct proof can be found in \cite[Theorem~A]{CapDeM}). Therefore, $G:=\Comm_{H^+}(U)$ is dense in the group $H^+$, and   $H^+$ is an element of $\Ss$ by Proposition~\ref{prop:SimpleBurgerMozes}. 

We now consider $G:=\Comm_{H^+}(U)$ with the $U$-localized topology. Since $H^+$ is in $\Ss$, the group $G$ is a member of $ \Rs$ by Theorem~\ref{thm:R_dense_embedding}. Let $M\normal G$ be the monolith and observe that $M$ is a dense locally compact subgroup of $H^+$. Additionally, $M$ is a topologically simple group.

\begin{prop} \label{prop:ex-non-sigma-cpt}
$M$ is not $\sigma$-compact.
\end{prop}
\begin{proof}
Set $B:=\Comm_{W(\Sym(4))}(W(V))$ and recall that $B$ is a branch group. Since $M$ is non-discrete, the intersection $M\cap B^2$ is non-trivial and open in $M$. As $M$ has a trivial quasi-centralizer, it follows that $M\cap B$ is non-trivial (indeed infinite) by Proposition~\ref{prop:FaithfulLC}. Via the proof of \cite[Theorem 4]{G00} or, alternatively,  of \cite[Theorem 3.2]{LMW17}, we may find a vertex $v$ such that the derived subgroup $K:=[\rist_B(v),\rist_B(v)]$ is a subgroup of $M$. 

The group $\rist_B(v)\simeq B$ contains a copy of $P:=\N_{W(\Sym(4))}(W(V))$, so $[P,P]\leq K$.  In view of \cite[Lemma 6.6]{CapDeM}, $P/W(V)\simeq \Sym(3)^{\Nb}$, hence $[P,P]/[P,P]\cap W(V)\simeq C_3^{\Nb}$. Letting $U=W(V)^2$ be as fixed above, it follows that  $K/K\cap U$ is uncountable, so $M$ is not $\sigma$-compact, since $U\cap M$ is a compact open subgroup of $M$.
\end{proof}

The group $M$ is thus a non-$\sigma$-compact topologically simple group in $\Rs$ which is a dense locally compact subgroup of a group in $\Ss$. This shows that we cannot limit our considerations to second countable groups, even if we restrict to the topologically simple members of $\Rs$ which embed into groups in $\Ss$.

The group $G= \Comm_{H^+}(U)$ serves to illustrate two other phenomena. The $2$-localization of $H^+=\Aut^+(T)$ is $(H^+)_{(2)}=\Comm_{H^+}(W(D_{8})^2)$ with the $W(D_8)^2$-localized topology. Our group $G$ embeds densely into $(H^+)_{(2)}$. The subgroup $U = W(V)^2$ is not commensurated in $(H^+)_{(2)}$, so $G$ is properly contained in $(H^+)_{(2)}$. Furthermore, one can check that $G$ is not normal in $(H^+)_{(2)}$.  The group $(H^+)_{(2)}$ is thus a locally pro-$2$ group that admits a proper dense locally compact subgroup $G$ such that the conclusions of (i) and (ii) of Proposition~\ref{prop:Loc-pro-nilp} fail.  This example shows why, in Proposition~\ref{prop:Loc-pro-nilp}, one cannot remove the hypothesis that $G$ have a compact open subgroup that is topologically finitely generated.
 
 The group $G$ also lies in a length two inclusion chain of groups in $\Rs$ with each a proper dense locally compact subgroup of the next:
 \[
 G\injects (H^+)_{(2)}\injects H^+.
 \]
 
 \begin{rmk} It is indeed possible to find abstractly simple groups $G_0$ and $G_1$ in $\Ss$ which are both locally pro-$2$ and such that $G_0$ is a dense locally compact subgroup of $G_1$. (The groups $G_0$ and $G_1$ are restricted Burger-Mozes groups for an illegal coloring.) The conclusion of Corollary~\ref{cor:fg_abs_simple} thus fails for the group $G_1$; the only hypothesis of Corollary~\ref{cor:fg_abs_simple} that fails is that $G_1$ does not have a compact open subgroup that is topologically finitely generated.  This illustrates, as with Proposition~\ref{prop:Loc-pro-nilp}, the relevance of the hypothesis in Corollary~\ref{cor:fg_abs_simple} that $G$ have a compact open subgroup that is topologically finitely generated. Additionally, $G_0$ and $G_1$ are members of a length two inclusion chain in $\Ss$ with each group a dense locally compact subgroup of the next. In a later article, we will discuss the details of this construction.
 \end{rmk}

Finally, we give an example of a dense locally compact subgroup $H$ in a group $G \in \Ss$ such that no compactly generated subgroup of $H$ is dense in $G$. Our construction is based on the following.  

\begin{lem}\label{lem:ChainApprox}
Let $G$ be a \tdlc group and   $R(0) \leq R(1) \leq \dots \leq G$ be an ascending chain of closed subgroups.  Let $U \leq G$ be compact open subgroup, $p$ be a prime, and $S \leq U$ be a pro-$p$-Sylow subgroup of $U$. Assume that $S \leq R(0)$. For each $n$, let $H(n) = \Comm_{R(n)}(S)$ and set $H = \bigcup_n H(n)$, viewed as a \tdlc group endowed with the $S$-localized topology. If $G = \overline{\bigcup_n R(n)}$, then $H$ is a dense locally compact subgroup of $G$. If in addition $R(n) \neq G$ for all $n$, then no compactly generated subgroup of $H$ is dense in $G$. 
\end{lem}	
\begin{proof}
	Since $S \leq R(0) \leq R(n)$ for all $n$, it follows that $S$ is a  pro-$p$-Sylow of $U \cap R(n)$ for all $n$. In particular, $H(n)$ is the $p$-localization $R(n)_{(p)}$, so $\overline{H(n)} = R(n)$ by Theorem~\ref{thm:Reid_localizations}. Therefore $\overline H$ contains $R(n)$ for all $n$. It follows that $H$ is dense in $G$ as soon as the union $\bigcup_n R(n)$ is dense. 
	
	The group $H(n)$ is an open subgroup of $H$ for each $n$. Since the $H(n)$ form an ascending chain whose union is the whole group $H$, every compactly generated subgroup $J \leq H$ is contained in $H(n)$ for some $n$. Thus, $\overline J \leq \overline{H(n)} = R(n)$, and the desired assertion follows. 
\end{proof}

\begin{example}\label{ex:Radu}
Let $d \geq 4$ be an integer, let $T$ be the $d$-regular tree and $G = \Aut(T)^+$, so that $G \in \Ss$ by Proposition~\ref{prop:SimpleBurgerMozes}. Let $U = G_{(e)}$ be an edge stabilizer, let $ p \leq d-1$ be an odd prime, and let $S$ be a pro-$p$-Sylow subgroup of $U$. In the notation introduced above, we have $U = G_{(e)} \cong W(\Sym(d-1))^2$ and $S \cong W(P)^2$, where $P$ is a $p$-Sylow subgroup of $\Sym(d-1)$. Since $p$ is odd, the group $P$ is also a $p$-Sylow subgroup of $\Alt(d-1)$, so that $S \cong W(P)^2$ is contained in $U(\Alt(d))_{(e)}$ (necessarily as a pro-$p$-Sylow subgroup). Moreover $S$ is non-discrete since $p \leq d-1$. 

 The existence of an ascending chain $R(n)$ of closed subgroups of $G$ satisfying all the hypotheses of Lemma~\ref{lem:ChainApprox} can be extracted from the work of N.~Radu from~\cite{Radu}. We describe the construction briefly, using freely the notation from \cite[\S4]{Radu}.  In particular, it follows from the Lemma that $G$ has a non-discrete dense locally compact subgroup $H$ such that no compactly generated subgroup of $H$ is dense in $G$. 

Assume first that $d$ is odd. For each $n \geq 0$, set $R(n) = G_{(i)}^+([0, n], [0, n])$ in the notation of  \cite[Definition~4.1]{Radu}, where $[0, n] = \{0, 1, \dots, n\}$. Then $R(0) = U(\Alt(d))^+$. In particular, as observed above, we have $S \leq R(0)$  and the group $S$ is non-discrete. By \cite[Proposition~4.8]{Radu}, we have $R(n) \neq G$ for all $n$. Using that $d$ is odd, one verifies that $R(n) \leq R(n+1)$ for all $n$. Finally, the fact that $G = \overline{\bigcup_n R(n)}$ can be established using a similar argument as in the proof of \cite[Fact~1 in Appendix~A]{CapRad}. 

In the case where $d$ is even, it is no longer true that $G_{(i)}^+([0, n], [0, n])$ is contained in $G_{(i)}^+([0, n+1], [0, n+1])$ for all $n$, so the definition of $R(n)$ must be modified. In that case, we set $R(n) = G_{(i)}^+(X_n, X_n)$, where the set $X_n \subseteq [0, n]$ is defined inductively by setting $X_0 = \{0\}$ and $X_{n} = \alpha(X_{n-1})$ for $n>0$, where $\alpha$ is the function defined after Lemma~4.2 in \cite{Radu_Latt}. The inclusion $R(n) \leq R(n+1)$ then follows from  \cite[Lemma~4.5]{Radu_Latt}, and the other verifications are similar as in the case where $d$ is odd. 
\end{example}

\acks

 We would like to thank the Isaac Newton Institute for Mathematical Sciences, Cambridge for support and hospitality during the program \textit{Non-positive curvature group actions and cohomology}. We would also like to thank the \textit{Winter of Disconnectedness} program for its support and hospitality. Substantial progress on this work was made while attending these programs. P.-E. Caprace is a F.R.S.-FNRS senior research associate, supported in part by EPSRC grant no EP/K032208/1. C. D. Reid was an ARC DECRA fellow, supported in part by ARC Discovery Project DP120100996. We thank the referee for his or her many detailed and useful comments.

\bibliographystyle{authordate1}
\bibliography{biblio_IMRN}

\end{document}